\documentclass[11pt]{amsart}

\usepackage[a4paper,inner=3.6cm,outer=3.6cm,top=3.7cm,bottom=3.7cm]{geometry}
\usepackage{amsfonts}
\usepackage{amssymb}
\usepackage{amsmath}
\usepackage[]{graphicx}
\usepackage[all,cmtip,2cell]{xy}
\usepackage{comment}
\usepackage{array}
\usepackage{enumitem}
\usepackage{mathtools}
\usepackage{multirow}
\usepackage{rotating,tabularx}

\setlist[itemize]{leftmargin=*}
\setlist[enumerate]{leftmargin=*}

\usepackage[usenames,dvipsnames]{xcolor}

\def\co{\colon\thinspace}
\newcommand{\C}{\mathbb{C}}
\newcommand{\Q}{\mathbb{Q}}
\newcommand{\Z}{\mathbb{Z}}
\newcommand{\N}{\mathbb{N}}
\newcommand{\rk}{\mathrm{rk}\,}

\newcommand{\R}{\mathbb{R}}
\newcommand{\bd}{\partial}
\newcommand{\ai}{$A_\infty$\ }
\newcommand{\F}{\mathcal{F}}

\newcommand{\cL}{\mathcal{L}}

\newcolumntype{C}[1]{>{\centering\arraybackslash$}p{#1}<{$}}

\newcommand{\x}{*}

\renewcommand{\F}{\mathcal{F}}

\DeclareMathOperator{\Spec}{Spec}

\renewcommand{\hom}{\mathit{hom}}
\DeclareMathOperator{\Sym}{\mathrm{Sym}}
\newcommand{\Bl}{\mathit{Bl}}

\newcommand{\Gr}{\mathrm{Gr}}

\newcommand{\bL}{\mathbf{L}}

\newcommand{\cM}{\mathcal{M}}

\newcommand{\Fc}{\mathcal{F}_\mathrm{c}}

\DeclareMathOperator{\interior}{int}
\DeclareMathOperator{\Ob}{Ob}
\DeclareMathOperator{\Id}{Id}
\DeclareMathOperator{\Hom}{Hom}
\DeclareMathOperator{\supp}{supp}

\newtheorem{theorem}{Theorem}[section]
\newtheorem{proposition}[theorem]{Proposition}
\newtheorem{lemma}[theorem]{Lemma}
\newtheorem{corollary}[theorem]{Corollary}
\newtheorem{thm}[theorem]{Theorem}
\newtheorem{lem}[theorem]{Lemma}

\newtheorem{prop}[theorem]{Proposition}

\theoremstyle{definition}
\newtheorem{definition}[theorem]{Definition}
\newtheorem{hypothesis}[theorem]{Hypothesis}
\newtheorem{defn}[theorem]{Definition}
\theoremstyle{remark}
\newtheorem{remark}{Remark}[section]
\newtheorem{example}[remark]{Example}
\newtheorem{rem}[remark]{Remark}

\numberwithin{equation}{section}

\newcommand{\red}{\color{red}}

\begin{document}

\title[Wall-crossing formula and Lagrangian mutations]{The wall-crossing formula \\ and Lagrangian mutations}

\author{James Pascaleff}
\address[JP]{University of Illinois at Urbana-Champaign, 1409 W. Green St., Urbana, IL 61081, United States}
\email{jpascale@illinois.edu}

\author{Dmitry Tonkonog}
\address[DT]{UC Berkeley, Berkeley CA 94720, United States}
\address[DT]{
Uppsala University, 751~06 Uppsala, Sweden}
\email{dtonkonog@berkeley.edu}
\thanks{
	JP was partially supported by NSF grant DMS-1522670.
	DT was partially supported by the Simons
	Foundation (grant \#385573, Simons Collaboration on Homological Mirror Symmetry), and by the Knut and Alice Wallenberg Foundation (project grant Geometry and Physics).}
\maketitle

\begin{abstract}
We prove a general form of the wall-crossing formula which relates the disk potentials of monotone Lagrangian submanifolds with their Floer-theoretic behavior away from a Donaldson divisor.
  We define geometric operations called mutations of Lagrangian tori in del Pezzo surfaces and in toric Fano varieties of higher dimension, and study the corresponding wall-crossing formulas that compute the disk potential of a mutated torus from that of the original one.

   In the case of del Pezzo surfaces, this justifies the connection between Vianna's tori and the theory of mutations of Landau-Ginzburg seeds. In higher dimension, this provides new Lagrangian tori in toric Fanos corresponding to different chambers of the mirror variety, including ones which are conjecturally separated by infinitely many walls from the chamber containing the standard toric fibre.
\end{abstract}

\section{Introduction}

Consider a monotone symplectic manifold $X$ (for example, a Fano K\"{a}hler manifold), and a monotone Lagrangian submanifold $L$ of $X$. An important enumerative invariant of $L$ is the \emph{disk potential} (also called Landau-Ginzburg potential) that counts Maslov index 2 disks in $X$ with boundary on $L$ weighted by the holonomy of a local system on $L$:
\begin{align}
  W_L : \Hom(H_1(L,\Z),\C^*) \to \C\\
  W_L(\rho) = \sum_{\beta} n_\beta \rho(\partial \beta)
\end{align}
where the sum is over classes $\beta \in H_2(X,L)$ with Maslov index 2, $n_\beta$ is the count of disks in class $\beta$, and $\rho(\partial \beta)$ is the holonomy of the local system determined by $\rho$ along the boundary loop of the disk. (For concreteness, we use the coefficient field $\C$ throughout the article, although all results are valid for any coefficient field without restrictions on its characteristic.)

The disk potential $W_L$ is an invariant of $L$ up to symplectomorphism, so it can be used to distinguish monotone Lagrangian embeddings. On the other hand, this invariant is often difficult to compute by direct enumeration of disks. This article contains results that allow one to relate the disk potentials $W_{L_1}$ and $W_{L_2}$ of two monotone Lagrangians $L_1$ and $L_2$ when these Lagrangians are related by certain \emph{non-Hamiltonian} transformations; these transformations are called \emph{Lagrangian mutations}. When $L_1$ and $L_2$ are related by such a Lagrangian mutation, our results allow us to determine the disk potential $W_{L_2}$ from $W_{L_1}$ (or vice versa). These functions are related by an \emph{algebraic mutation}, a transformation that appears in the theory of cluster algebras. Such a result, that Lagrangian mutations correspond to algebraic mutations of the disk potentials, is called a \emph{wall-crossing formula}.

Setting aside the explicit description of Lagrangian mutations for the moment, our first main result is a theorem that shows that a relation between the disk potentials can be deduced from the nonvanishing of a Floer cohomology group computed in a Liouville subdomain containing the given Lagrangians. We denote by $K_U$ the canonical bundle of $U$.

\begin{theorem}[General wall-crossing formula]
	\label{th:wall_crossing}
Let $U$ be a Liouville domain with $c_1(U)=0\in H^2(U,\Z)$. Let 
$$\bL_i=(L_i,\rho_i)\subset U,\quad i=1,2,$$ 
be two exact Lagrangian submanifolds equipped with $\C^*$-local systems, such that the Floer cohomology of this pair computed in $U$ is non-vanishing:
$$
HF^*_U(\bL_1,\bL_2)\neq 0.
$$
Let $X$ be a compact monotone symplectic manifold and $D$ a smooth symplectic divisor which is Poincar\'e dual to a multiple of the anticanonical class: $PD[D]=dc_1(X)$, $d\in \Z_{>0}$. 

Suppose that there exists a Liouville embedding $U\subset X\setminus D$ such that $L_i$ becomes monotone in $X$, and the images of $H_1(L_i)$ in $H_1(X)$ are zero. Suppose either
\begin{enumerate}
\item $d = 1$ or $2$,
\item 
	the map $\bigoplus_{i} H_1(L_i,\Z)\to H_1(U,\Z)$ is surjective,

\item or more generally
there exists a trivialization $\zeta$ of $K_U$ such that $\zeta^d$ is homotopic to the natural trivialization of $K_{X \setminus D}^{\otimes d}$.
\end{enumerate}
If $W_{L_i}(\rho_i)\in \C$ is the value of the disk potential of $\bL_i$ computed in $X$ using the embedding $L_i\subset U\subset X$, then 
$$
W_{L_1}(\rho_1)=W_{L_2}(\rho_2).
$$
\end{theorem}
The hypotheses of this theorem may seem complicated at first glance because we have striven for the greatest generality that our method of proof allows, but the purpose of the theorem is to reduce to the question of matching disk potentials to the question of nonvanishing of $HF^*_U(\bL_1,\bL_2)$, which is a question that does not involve $X$. This latter question can often be answered using a technique due to Seidel \cite[Section~11]{SeiBook13}. Other variations of the hypotheses are possible; see Theorem \ref{thm:wall_cross_P2} and Lemma \ref{lem:Z_graded_implies_P2}.

The key tools in the proof are the \emph{relative Floer theory} of the pair $(X,D)$ which relates Floer theory in $X$ to Floer theory in $X \setminus D$, described in section \ref{sec:relative-floer}, and the locality properties of the Floer cohomology of compact Lagrangians, described in section \ref{sec:locality}. An alternative approach to this result that uses symplectic cohomology of $X \setminus D$ and a stretching argument is presented in the paper \cite{To17} by the second author.

With this result in hand, we proceed to study the geometry of Lagrangian mutations and apply Theorem \ref{th:wall_crossing} to them. The first part of this study concerns complex dimension two, which is the most studied case in the literature, while the second part considers higer-dimensional cases where new phenomena appear.

In complex dimension two, the examples are the del Pezzo surfaces, regarded as monotone symplectic manifolds. There is a notion of \emph{mutation configuration}, which is a pair $(L,D)$ of a Lagrangian torus $L$ and a Lagrangian disk $D$ attached to it; from such a configuration one can construct a mutated torus that we denote $\mu_D L$. There is a corresponding algebraic mutation defined on functions that is denoted $\mu_{[\bd D]^\perp}$. We show the following.

\begin{theorem}[cf.~Theorem \ref{th:wall_cross_4d}]
  	Let $X$ be a monotone del~Pezzo surface, and $(L,D)\subset X$ be a mutation configuration. Let $L'=\mu_D L$ be the mutated torus. Then
	\begin{equation}
	W_{L'}=\mu_{[\bd D]^\perp}W_{L}.
	\end{equation}
\end{theorem}
We actually prove a more refined result, Theorem \ref{th:lag_mut}. This involves a \emph{Lagrangian seed}, which is a Lagrangian torus with a collection of Lagrangian disks attached. Using results of \cite{STW15}, one finds that, when one mutates using one disk, one can follow the other disks through the process. This is very important because it shows that mutations can be \emph{iterated}. As a consequence we can show the existence of infinitely many essentially distinct monotone Lagrangian tori in all del~Pezzo surfaces.
\begin{theorem}[Vianna \cite{Vi16} + Corollary \ref{cor:bl2cp2}]
  If $X$ is a monotone del~Pezzo surface, then $X$ contains infinitely many monotone Lagrangian tori which are pairwise not Hamiltonian isotopic.
\end{theorem}
In all cases except $X = \Bl_2\C P^2$, the two-point blow up of the projective plane, this result was proved by Vianna \cite{Vi16}. The method of \cite{Vi16} uses the Newton polytope of the disk potential to distinguish tori, but this does not suffice for the case of $\Bl_2\C P^2$. The more refined information about disk potentials that we obtain allows us to reprove Vianna's theorem and extend it to this last remaining case.

Another application is to homological mirror symmetry of del Pezzo surfaces. We consider the cases $X = \Bl_k \C P^2$ with $k = 6,7,8$, note that $k =6$ is the cubic surface. We find a generator for a certain summand of the monotone Fukaya category, and we show that one of the other summands contains an infinite family of pairwise non-quasi-isomorphic objects supported on tori; see Propositions \ref{prop:hms_dp} and \ref{prop:pairwise_disjoint_family} in section \ref{sec:application_hms}. This answers \cite[Conjecture~B.2]{She15}. We note that the potential for the cubic surface was also computed by Fukaya-Oh-Ohta-Ono \cite[section 24]{FO311b}.

We then turn to complex dimension greater than two. In higher dimensions there is more than one type of mutation. There is a local model for the two dimensional mutation, which is a configuration in $\C^2 \setminus \{z_1z_2 = \epsilon\}$, and one can obtain an $n$-dimensional local model by taking the product with $(\C^*)^{n-2}$. However, one may also consider local models of the form
\begin{equation}
  \left(\C^k\setminus \{z_1z_2\cdots z_k = \epsilon\}\right)\times (\C^*)^{n-k}
\end{equation}
for $2 \leq k \leq n$. We find such configurations in toric manifolds of higher dimension, and prove the wall-crossing formula, see Theorem \ref{th:mut_higher_dim_toric}. With respect to an appropriately chosen coordinate system $(y_i)_{i=1}^n$, we find that the algebraic mutation takes the form (Equation \eqref{eq:wall_cross_simple})
\begin{equation}
\begin{array}{ll}
y_i\mapsto y_i,\quad  i=1,\ldots, k-1,\\
y_k\mapsto y_k(1+y_1+\ldots+y_{k-1}),\\
y_i\mapsto y_i,\quad i=k+1,\ldots,n,
\end{array}
\end{equation}
cf.~\cite{KP11,ACGK12}.
The complete statement of the wall-crossing formula in section~\ref{sec:higher_dim} also specifies how to choose the appropriate basis $(y_i)_{i=1}^n$.
Applying these results to $\C P^n$, we prove
\begin{theorem}[Corollary \ref{cor:cpn_tori}]
  For each $1\le k\le n$, $\C P^n$ contains a monotone Lagrangian torus whose potential is given by
  \begin{equation}
    W=\sum_{i=k}^n x_i^{-1}+(x_k)^k\cdot \left(\sum_{i=1}^k x_i^{-1}\right)^k\cdot \prod_{i=1}^n x_i.
  \end{equation}
\end{theorem}

A particularly interesting feature of the algebraic mutation for $k > 2$ is that it is \emph{not binomial}, whereas all mutations in complex dimension two are binomial. Recall that in the pictures of SYZ mirror symmetry developed by Gross-Siebert and Kontsevich-Soibelman, one envisions that the space of Lagrangian submanifolds is divided into \emph{chambers} by \emph{walls}, the latter being the loci where the Lagrangian torus bounds a Maslov index zero disk. Now take a monotone torus (say a toric fiber), and apply a higher-dimensional mutation with $k > 2$ to it. These two tori lie in two different chambers, and we expect that the two chambers are separated by infinitely many walls, that is, there is no path connecting them that crosses finitely many walls. This is because the non-binomial mutation is not expected to be expressible as a composition of finitely many binomial mutations. 

Having outlined our results, we shall present some further remarks to place these results in a broader context and connect to related work, and then we give a detailed outline of the contents of the paper.
 
\subsection{Remarks on Theorem \ref{th:wall_crossing}}

\begin{remark}
  We will recall later that $X\setminus D$ has a canonical Liouville structure which is implied in the requirement that $U\subset X\setminus D$ be a Liouville embedding. 
  We recall that an Liouville embedding $(U,\theta_U)\subset (M,\theta_M)$ is an inclusion for which  $\theta_M|_U$ differs from $\theta_U$ by an exact 1-form.
  Second, we explain later that the $L_i\subset X$ are automatically monotone, therefore the potentials of the $L_i$ are meaningful invariants; they are recalled in section~\ref{subsec:potential}.
\end{remark}

\begin{remark}
	\label{rmk:wall_cross_symp}A divisor $D$ appearing in the statement of Theorem~\ref{th:wall_crossing} is called a Donaldson divisor.  The role played by $D$ is auxiliary: it aids the proof. To apply Theorem~\ref{th:wall_crossing}, one may start with a given symplectic embedding of a Liouville domain $U$ in a compact Fano (or monotone symplectic) manifold $X$ and try to find a divisor $D\subset X$ away from $U$ which turns $U\subset X\setminus D$ into an inclusion of Liouville domains. In some cases this is possible, and we again refer to Subsection~\ref{subsec:donaldson} for further discussion.
\end{remark}

\subsection{Context: wall-crossing}
\label{subsec:intro_wc}
Let $U$ be a Liouville domain and 
$L_1,L_2\subset U$  exact Lagrangian submanifolds.
Denote $m_i=\dim H_1(L_i,\R)$ and fix integral bases for these homology groups. The spaces of local systems $\Hom(H_1(L_i,\Z),\C^*)$ may then be identified with $(\C^*)^{m_i}$. Suppose we are able to determine those pairs of local systems for which the Floer cohomology does not vanish: 
\begin{equation}
\label{eq:rho_pairs_hf_nonzero}
\left\{(\rho_1,\rho_2) \in (\C^*)^{m_1}\times (\C^*)^{m_2} \mid
HF^*_U((L_1,\rho_1),(L_2,\rho_2))\neq 0\right\}.
\end{equation}
For simplicity, assume that $m_1=m_2=m$ and that the above set is determined by a birational map
$$
\phi\co (\C^*)^{m}\dashrightarrow (\C^*)^{m},
$$
that is, $$HF^*((L_1,\rho),(L_2,\phi(\rho)))\neq 0\textit { for all }\rho\in (\C^*)^m\textit{ such that }\phi(\rho)\textit{ is defined}.$$
In this case we call $\phi$ the \emph{wall-crossing map}. This map, if it exists, is determined by $L_1,L_2$ and $U$.

Now suppose $U\subset X$ is an inclusion as in the setup of Theorem~\ref{th:wall_crossing}, so that $L_1,L_2\subset X$ become monotone.
Theorem~\ref{th:wall_crossing} can be rephrased to say that the Landau-Ginzburg potentials of the $L_i$ differ by action of  $\phi$:
\begin{equation}
\label{eq:W_coord_change}
W_{L_1}=W_{L_2}\circ \phi.
\end{equation}
Because $\phi$ is a birational map, the composition $W_{L_2}\circ \phi$ does not \emph{a priori} have to be a Laurent polynomial, but (\ref{eq:W_coord_change}) ensures that it actually is. This is a symplectic-geometric manifestation of the \emph{Laurent phenomenon} known in cluster algebra.

Now suppose that $L_1,L_2$ are Lagrangian (but not necessarily Hamiltonian) isotopic via an isotopy $L_t$. In this case the wall-crossing map $\phi$ is determined by Maslov index~0 holomorphic disks bounded by the $L_t$. See Auroux \cite{Au07,Au09} and Kontsevich and Soibelman \cite{KS06} for a discussion of the  basic 4-dimensional example, and Fukaya, Oh, Ohta and Ono \cite{FO3Book} for a general statement, briefly recollected in \cite[(5c)]{Sei08}. In main examples, there is a discrete set of parameter values $t$ such that $L_t$ bounds a holomorphic Maslov index~0 disk (in dimension~4, this is guaranteed by the virtual dimension of the moduli space), and the original term \emph{wall-crossing} refers to a moment $t_0$ when $L_{t_0}$ bounds such a disk.
 Identifying the precise way in which the holomorphic Maslov index~0 disks contribute to $\phi$ is quite tricky, due to the fact that such disks automatically come with all their multiple covers which are not transversally cut out by the $J$-holomorphic curve equation.

To avoid this difficulty, in concrete examples it is often easier to determine $\phi$ by computing the Floer cohomologies (\ref{eq:rho_pairs_hf_nonzero}) directly, rather than by studying  holomorphic Maslov index~0 disks.
This idea is due to Seidel \cite[Section~11]{SeiBook13}, and it served as an inspiration for Theorem~\ref{th:wall_crossing}.

\subsection{Context: Lagrangian mutations}
Constructing exact Lagrangian submanifolds in a Liouville domain $U$ is hard in general.
Because much of the geometry of $U$ is encoded in its Lagrangian skeleton, it is natural to look for exact Lagrangian submanifolds within a skeleton, and to seek for possible modifications of the skeleton which would reveal new Lagrangian submanifolds unseen by the previous skeleton. The ultimate hope is that there is a combinatorial way of encoding and book-keeping such modifications, preferably with interesting algebra behind it.

Now, whenever $U\subset X\setminus D$ is a Liouville inclusion, one obtains a monotone Lagrangian submanifold in $X$ for each exact one in $U$, and Theorem~\ref{th:wall_crossing} provides a way of computing their potentials, which can be used to prove that the obtained monotone Lagrangians in $X$ are not Hamiltonian isotopic.

This relatively new idea, called \emph{mutation of Lagrangian skeleta}, is by no means easy to realise in practice; so far it has been made precise for skeleta of certain four-dimensional Liouville domains by Shende, Treumann and Williams \cite{STW15}. The skeleta they considered consist of a 2-torus with a collection of 2-disks attached along their boundary to the torus. It is proved in \cite{STW15} that the considered mutations can be iterated indefinitely. The underlying combinatorics is of cluster-algebraic nature and was explored in a series of papers by Cruz Morales, Galkin and Usnich \cite{CM13,CMG13,GU12}. This story exhibits a beautiful interplay between Lagrangian mutations and mutations of the corresponding Landau-Ginzburg potentials. This interplay is made rigorous by Theorem~\ref{th:wall_crossing}, and in section~\ref{sec:mut_4d} we explore it in full detail.

Further in section~\ref{sec:higher_dim}, we provide examples of higher-dimensional mutations arising from toric geometry. To our knowledge, they are the first higher-dimensional examples of mutations. Unlike the 4-dimensional story, in higher dimensions we will discuss our examples of mutation only as a one-off procedure, meaning that we will not address the possibility of iterating it.

\subsection{Outline of the paper}
\label{sec:outline}

In section \ref{sec:relative-floer} we introduce relative Floer theory of a pair $(X,D)$. This theory relates Floer theory in $X$ to Floer theory in $X \setminus D$, and the disk potential appears as a curvature term. Since we desire to treat the case where $D$ represents a positive multiple of the anticanonical class rather than the anticanonical class itself, this section begins with an exposition of the theory of graded Lagrangians in manifolds with torsion $c_1$ (such as $X \setminus D$ when $[D] = dc_1(X)$). We then introduce the disk potential, and related it to relative Floer theory. The main results of this section are Theorem \ref{thm:wall_cross_P2} and Lemma \ref{lem:Z_graded_implies_P2} that allow one to match the potentials of different Lagrangian branes.

Section \ref{sec:locality} begins with what is essentially an observation, that when $j : U \to V$ is a Liouville embedding, the compact Fukaya category of $U$ embeds fully faithfully into the Fukaya category of $V$. This allows us to amplify the results of section \ref{sec:relative-floer} and prove Theorem \ref{th:wall_crossing}. This section ends with a discussion of the question of finding a Donaldson divisor in the complement of a given Lagrangian skeleton.

Section \ref{sec:mut_4d} is our study of mutations in complex dimension two. This section contains the definitions of Lagrangian and algebraic seeds and mutations, and reviews the algebraic theory of mutations. It shows that seeds can be mutated, and hence that the mutation process can be iterated. We show that the general wall-crossing formula applies in this case and allows us to relate the potentials across a mutation. The main results are encapsulated in Theorem \ref{th:lag_mut}. There is an exposition of how Lagrangian seeds arise from almost toric diagrams for del Pezzo surfaces, and a precise calculation for del Pezzo surfaces that results from this. This section ends with the applications to infinitely monotone tori and homological mirror symmetry described above.

Section \ref{sec:higher_dim} studies higher dimensional mutations. In the context of toric manifolds, we show that such mutations can be applied to the Lagrangian torus that corresponds to the barycentre of the moment polytope. As mentioned above, we obtain a non-binomial wall-crossing formula, see Theorem \ref{th:mut_higher_dim_toric}.

\subsection*{Acknowledgements}
We thank Paul Seidel for a suggestion that led to the proof using relative Floer theory; this greatly simplified the proof over a previous version. We thank Nick Sheridan for his explanations regarding gradings in the relative Fukaya category. 

We also thank the Mittag-Leffler Institute and the organizers of the Fall 2015 special semester on ``Symplectic Geometry and Topology'' where this project was initiated, and the Institute for Advanced Study where part of this work was completed.

JP was partially supported by NSF grant DMS-1522670.

DT was partially supported by the Simons
Foundation (grant \#385573, Simons Collaboration on Homological Mirror Symmetry), and by the Knut and Alice Wallenberg Foundation (project grant Geometry and Physics).

\section{Relative Floer theory and curvature}
\label{sec:relative-floer}

Our exposition draws heavily on \cite{She15,She13}, which serve as excellent references for this material.

\subsection{Manifolds with torsion $c_1$ and fractional gradings}
\label{sec:frac-gradings}

In this subsection we present a theory of gradings that works in the situation of manifolds with torsion first Chern class. This will furnish a natural setting in which Floer cohomology chain complexes admit fractional gradings. This is in a sense partially orthogonal to the theory in \cite{Sei00}: we have a \emph{torsion} $c_1$ and Floer cohomology admits a \emph{fractional} grading, whereas in Seidel's case $c_1$ is a \emph{divisible} class and Floer cohomology admits a grading by a \emph{torsion} group. The theory presented here is an elaboration of standard techniques, but we could not find a reference in the form we desire. See \cite{She15,She13} for the closely related theory of anchored Lagrangian submanifolds. 

\subsubsection{$d$-graded Lagrangian submanifolds}

Let $d$ be a positive integer, and let $U$ be an almost complex manifold of dimension $2n$ such that $d c_1(TU) = 0 \in H^2(U,\Z)$. Denote by $K_U = \det_\C T^*U$ the canonical bundle. Then $K_U^{\otimes d}$ is trivial, and we pick a nowhere vanishing section $\eta_U^d$. 

Let $L \subset U$ be a totally real submanifold of dimension $n$, and consider the orientiation line bundle $\det_\R TL$. There is a sequence of bundle maps
\begin{equation*}
  \xymatrix{
    (\det_\R TL)^{\otimes d} \ar[r] & (\det_\C TU)^{\otimes d} \ar[r]^-{\eta_U^d} & \C  \\
  }
\end{equation*}  
where the first map comes from the isomorphism $(\det_\R TL) \otimes_\R \C \cong \det_\C TU$ (here we use that $L$ is totally real), and the second map is pairing with $\eta_U^d$. Since $\eta_U^d$ is nowhere vanishing, the composition takes nonzero vectors to nonzero vectors. Thus by taking phases, we obtain a function $\varphi^d_L : (\det_\R TL)^{\otimes d} \setminus 0_L \to S^1$, which we call the \emph{$d$-phase function} of $L$. If $L$ is oriented, then $(\det_\R TL)^{\otimes d}$ comes with a preferred class of nonvanishing sections; the same is true if $d$ is even. In either case, by composing $\varphi^d_L$ with such a section we obtain a map $L \to S^1$. Hencefore we shall assume that either $L$ is oriented or $d$ is even, and we shall denote again by $\varphi_L^d : L \to S^1$ the resulting $d$-phase map.

Let $\exp(2\pi i \cdot) : \R \to S^1$ be the universal covering homomorphism with kernel $\Z$. 

\begin{defn}
\label{def:d-graded}
  Let $U$ be an almost complex manifold with $dc_1(TU) = 0$ and $\eta_U^d$ a trivialization of $K_U^{\otimes d}$. Let $L \subset U$ be a totally real submanifold. 
  \begin{itemize}
  \item If $d$ is even, a \emph{$d$-grading} of $L$ consists of a lift $\tilde{\varphi}_L^d\co L \to \R$ of the $d$-phase map $\varphi_L^d \co L \to S^1$ defined using $\eta_U^d$. That is, we require $\exp(2\pi i \tilde{\varphi}_L^d) = \varphi_L^d$.
  \item If $d$ is odd, a \emph{$d$-grading} of $L$ consists of a choice of orientation of $L$, together with a lift $\tilde{\varphi}_L^d\co L \to \R$ of the $d$-phase map $\varphi_L^d \co L \to S^1$ defined using $\eta_U^d$ and the chosen orientation of $L$.
  \end{itemize}
\end{defn}

A $d$-grading of $L$, if it exists, is not unique, but any two $d$-gradings differ by addition of a locally constant function $L \to \Z$.

We remark that the standard notion of grading on a Lagrangian submanifold (e.g., \cite[Section 11]{SeiBook08}) is the case $d = 2$ of this definition. 

\subsubsection{Fractional indices}

The purpose of the notion of $d$-gradings is that, if $L_0$ and $L_1$ are $d$-graded submanifolds of $U$, then to any transverse intersection point $x \in L_0 \cap L_1$ we can associate an absolute index $i(x)$, which is a rational number. More precisely, let $Z_d \subset \Q$ be the additive subgroup generated by $1$ and $2/d$: if $d$ is even then $Z_d = (2/d)\Z$, while if $d$ is odd then $Z_d = (1/d)\Z$; note that for $d = 1$ or $d = 2$, $Z_d = \Z$. We seek to define $i(x) \in Z_d$.

We first develop the linear theory (which we will later apply to the tangent space at $x$). To start with, we recall the linear theory of $2$-graded Lagrangians developed in \cite[Section 11]{SeiBook08}. Let $\Gr(V) \cong U(n)/O(n)$ denote the Lagrangian Grassmannian in the Hermitian vector space $V$. Pick an isomorphism $\eta^2_V \co \det_\C(V)^{\otimes 2} \to \C$, and denote by $\alpha_V \co \Gr(V) \to S^1$ the associated 2-phase function. A \emph{2-graded Lagrangian plane} is a pair $\Lambda^\# = (\Lambda, \alpha^\#)$, where $\Lambda \in \Gr(V)$ and $\alpha^\# \in \R$ satisfies $\exp(2\pi i \alpha^\#) = \alpha_V(\Lambda)$. To any pair $\Lambda_0^\#, \Lambda_1^\#$ of 2-graded Lagrangian planes there is an associated \emph{absolute index} $i(\Lambda_0^\#, \Lambda_1^\#) \in \Z$, defined at \cite[Eq.~(11.25)]{SeiBook08}.

There is a shift operation on 2-graded Lagrangian planes: for $\sigma \in \Z$, one defines
\begin{equation}
  S^\sigma\Lambda^\# = S^\sigma(\Lambda,\alpha^\#) = (\Lambda,\alpha^\# - \sigma). 
\end{equation}
The crucial identity is (cf. \cite[Eq.~(11.37)]{SeiBook08}):
\begin{equation}
\label{eq:crucial}
  i(S^{\sigma_0}\Lambda_0^\#,S^{\sigma_1}\Lambda_1^\#) = i(\Lambda_0^\#, \Lambda_1^\#) + \sigma_0 - \sigma_1.
\end{equation}

We now develop the parallel $d$-graded theory in the case where $d$ is even. Pick an isomorphism $\eta^d_V : \det_\C(V)^{\otimes d} \to \C$, and denote by $\beta_V : \Gr(V) \to S^1$ the associated $d$-phase function. Also choose an isomorphism $\eta^2_V : \det_\C(V)^{\otimes 2} \to \C$ such that $\eta^d_V = (\eta^2_V)^{d/2}$. A \emph{$d$-graded Lagrangian plane} is $\Lambda^\flat = (\Lambda, \beta^\flat)$ where $\Lambda \in \Gr(V)$ and $\beta^\flat \in \R$ satisfies $\exp(2\pi i \beta^\flat) = \beta_V(\Lambda)$. There is a shift operation on $d$-graded Lagrangian planes, defined for $\sigma \in (2/d) \Z$ by 
\begin{equation}
  S^\sigma\Lambda^\flat = S^\sigma(\Lambda,\beta^\flat) = (\Lambda, \beta^\flat - (d/2)\sigma) 
\end{equation}
This definition is consistent with the previous one, in the following sense. There is a function from $2$-graded Lagrangians to $d$-graded Lagrangians that takes $\Lambda^\# = (\Lambda, \alpha^\#)$ to $\Lambda^\flat = (\Lambda, (d/2)\alpha^\#)$; call this function $\tau_{d/2}$. Then for any $\sigma\in \Z$, we have
\begin{equation}
\label{eq:consistent}
  S^\sigma\tau_{d/2}\Lambda^\# = \tau_{d/2}S^\sigma \Lambda^\#.
\end{equation}

Going further, we can write any $d$-graded Lagrangian plane $\Lambda^\flat$ in the form
\begin{equation}
  \Lambda^\flat = S^\sigma\tau_{d/2}\Lambda^\#
\end{equation}
for some $2$-graded Lagrangian plane $\Lambda^\#$ and some $\sigma \in (2/d)\Z$, for indeed any two $d$-gradings on the same underlying Lagrangian plane differ by a shift. This fact, combined with \eqref{eq:crucial}, motivates the following definition.

\begin{defn}
  Let $\Lambda_0^\flat$ and $\Lambda_1^\flat$ be $d$-graded Lagrangian planes. For $i = 0,1$, write  $\Lambda_i^\flat = S^{\sigma_i}\tau_{d/2}\Lambda_i^\#$ for some $2$-graded Lagrangian $\Lambda_i^\#$ and $\sigma_i \in (2/d)\Z$. Define the \emph{absolute index} to be
  \begin{equation}
    \label{eq:frac-index-def}
    i(\Lambda_0^\flat,\Lambda_1^\flat) := i(\Lambda_0^\#,\Lambda_1^\#) + \sigma_0 - \sigma_1 \in Z_d,
  \end{equation}
  where $i(\Lambda_0^\#,\Lambda_1^\#)$ is the previously defined absolute index of $2$-graded Lagrangian planes. 
\end{defn}

It is precisely \eqref{eq:crucial} and \eqref{eq:consistent} that guarantee $i(\Lambda_0^\flat,\Lambda_1^\flat)$ does not depend how $\Lambda_i^\flat$ is represented as a shift of a $2$-graded Lagrangian. We must also address the dependence on the choice of $\eta^2_V$ such that $\eta^d_V = (\eta^2_V)^{d/2}$. The ambiguity in the choice of $\eta^2_V$ is multiplication by a $(d/2)$-th root of unity. Such a multiplication changes the notion of $2$-graded planes by a fractional shift. In terms of \eqref{eq:frac-index-def}, the effect is to add the same quantity to $\sigma_0$ and $\sigma_1$, while leaving $i(\Lambda_0^\#,\Lambda_1^\#)$ unchanged. Thus the absolute index is independent of the choice of $\eta^2_V$.

The theory in the case where $d$ is odd is similar to the case of $d$ even described above, with two difference: first, in addition to $\eta^d_V$, we choose an isomorphism $\eta_V : \det_\C(V) \to \C$ such that $\eta^d_V = (\eta_V)^d$, and second that we work with the oriented Lagrangian Grassmannian $\Gr^{+}(V) = U(n)/SO(n)$ throughout. Then we use $(\eta_V)^2$ to determine the notion of $2$-graded planes in the construction.

With the linear theory in hand, we can define the \emph{absolute index} of a pair of $d$-graded Lagrangian submanifolds. Namely let $U$ be a symplectic manifold with $dc_1(TU) = 0$ and a chosen trivialization $\eta_U^d$, and suppose $L_1$ and $L_2$ are $d$-graded in the sense of Definition \ref{def:d-graded}. Then at any transverse intersection point $x \in L_1\cap L_2$, we find that $T_xL_1$ and $T_xL_2$ are two Lagrangian subspaces of $T_xU$, carriyng precisely the data required to define the absolute index as above, and we define $i(x) \in Z_d \subset \Q$ accordingly.

\subsection{From monotone to $d$-graded}

Now we study how monotone Lagrangians in a monotone symplectic manifold $X$ become $d$-graded in the complement of a divisor $D$. Given a compact symplectic manifold $(X,\omega)$ such that $[\omega] = 2\tau c_1(X)$, $\tau > 0$, we consider a smooth divisor $D$ such that $[D] = dc_1(X)$ for some positive integer $d$. We note that then $[\omega] = \frac{2\tau}{d}[D]$.

  We equip $X \setminus D$ with a Liouville one-form $\alpha$ such that $d\alpha = \omega|_{X\setminus D}$. Recall that there is a notion of \emph{linking number} of $\alpha$ with $D$ \cite[Section 3.5]{She15}; it is computed as
  \begin{equation}
    \ell = - \lim_{\delta \to 0} \int_{S^1}  \gamma_\delta^*y^*\alpha
  \end{equation}
  where $y : D^2 \to X$ is an embedding of a disc that intersects $D$ positively transversally at the center of the disc and nowhere else, and $\gamma_\delta : S^1 \to D^2$ is a positively oriented circle of radius $\delta$. By pairing both sides of the equation $[\omega] = \frac{2\tau}{d}[D]$ with a surface $\Sigma \subset X$, we find that $\int_\Sigma \omega = \frac{2\tau}{d}[D]\cdot [\Sigma]$. On the other hand we can use Stokes' theorem to see $\int_\Sigma \omega = \ell [D] \cdot [\Sigma]$. Hence necessarily $\ell = \frac{2\tau}{d}$.

  In the case where $D$ has multiple components, there is a linking number for each of them; the results below remain valid as long as we assume this number is the same for each component.

\begin{proposition}
  \label{prop:monotone-to-d-graded}
  Suppose that $X$ is a monotone symplectic manifold and $L$ is a monotone Lagrangian submanifold such that the image of $H_1(L) \to H_1(X)$ is zero. Let $D$ be a divisor representing $dc_1(X)$ disjoint from $L$, and let $\eta^d$ be the natural trivialization of $K_{X\setminus D}^d$.  If $L$ is exact in $X \setminus D$, then $L$ is $d$-gradable with respect to $\eta^d$.
\end{proposition}

\begin{proof}
  Take any loop $\gamma$ in $L$, and choose a surface $\beta$ in $X$ with $\partial \beta = \gamma$. Then using Stokes' theorem applied to $d\alpha = \omega_{X\setminus D}$, we find that $\int_\beta \omega = \ell [D]\cdot \beta - \int_{\partial \beta} \alpha$. The second term vanishes because $L$ is exact, so $\int_\beta \omega = \ell [D]\cdot \beta = \frac{2\tau}{d} [D]\cdot \beta$. Now, by monotonicity, $\int_\beta \omega = \tau \mu(\beta)$. Thus $\mu(\beta) = \frac{2}{d}[D]\cdot \beta$. On the other hand, by our choice of $\eta^d$, the quantity $\frac{d}{2} \mu(\beta)$ is the sum of the degree of the $d$-phase function $\phi_L^d : L \to S^1$ on $\gamma$ and the intersection number $[D]\cdot \beta$. Thus the degree of $\phi_L^d$ on $\gamma$ is zero. Since this is true for all loops, $L$ is $d$-gradable.
\end{proof}

\subsubsection{The case $d=1$}
Having developed the theory of $d$-graded Lagrangian submanifolds, we now specialize to the most important special case used in this paper, which is $d = 1$. When $d = 1$, the torsion condition is $c_1(U) = 0$, and a Lagrangian $L$ admits a $1$-grading if and only if it is orientable and has vanishing Maslov class. The absolute index of an intersection point of two $1$-graded Lagrangians lies in $Z_1 = \Z$.

\begin{remark}
  The case $d = 1$ suffices for the main applications in this paper, where $U$ will be (a subdomain of) $X \setminus D$ where $D \subset X$ is an anticanonical divisor. However, the more general setting of $d$-graded Lagrangians seems to be the right one for the proof of the general wall-crossing formula using relative Floer theory, and we believe it will be useful in future applications.
\end{remark}

\subsubsection{Compatible trivializations}
We end this section with a definition that will be useful in our later arguments.

\begin{definition}
\label{def:compatible-triv}
  Let $j : U \to V$ be a an open embedding. Suppose $c_1(U) = 0$ and $dc_1(V) = 0$, and let $\zeta$ be a trivialization of $K_U$, and $\eta^d$ a trivialization of $K_V^{\otimes d}$. We say that $\zeta$ and $\eta^d$ are \emph{compatible trivializations} if the trivialization $\zeta^d$ of $K_U^{\otimes d}= K_V^{\otimes d}|_U$ obtained from $\zeta$ is homotopic to $\eta^d$. 
\end{definition}


  Let us identify the obstruction  to finding a compatible trivialization. Take trivializations $\zeta$ of $K_U$ and $\eta^d$ of $K_V^{\otimes d}$. Then $\zeta^d$ and $\eta^d|_U$ are two trivializations of the same line bundle, and so differ by a class $\alpha \in H^1(U,\Z)$. Changing $\zeta$ by $\beta \in H^1(U,\Z)$ changes this difference from $\alpha$ to $\alpha + d\cdot \beta$. Thus the obstruction lies in $H^1(U,\Z)/(d\cdot H^1(U,\Z))$. Observe that, by the long exact sequence in cohomology associated to the short exact sequence $0 \to \Z \to \Z \to \Z/d\Z \to 0$ of coefficient groups, the latter group embeds in $H^1(U,\Z/d\Z)$.  There is an analogue of Proposition~\ref{prop:monotone-to-d-graded} that gives a sufficient condition for vanishing of the obstruction that is easy to check in the applications.

  \begin{proposition}
  \label{prop:U-to-d-graded}
 Suppose that $X$ is a monotone symplectic manifold,  $D$ a divisor representing $dc_1(X)$, and $U\subset X\setminus D$ a Liouville subdomain such that $c_1(U)=0$ and the inclusion $H_1(U,\Z)\to H_1(X,\Z)$ is zero. Assume that $L_i \subset U$ are Lagrangian submanifolds that are exact in $U$ and monotone in $X$, and such that the map $\bigoplus_{i} H_1(L_i,\Z) \to H_1(U,\Z)$ is surjective.
 
 Let $\eta^d$ be the natural trivialization of $K_{X\setminus D}^d$. Then there is a trivialization $\zeta$ of $K_U$ is compatible with $\eta^d$ in the sense of Definition~\ref{def:compatible-triv}. Furthermore, any Lagrangian submanifold $L' \subset U$ that is exact in $U$ and monotone in $X$ has vanishing Maslov class with respect to $\zeta$.
\end{proposition}

  \begin{proof}
    Let $\gamma \in H_1(U,\Z)$ be the class of a loop. The hypothesis that $\bigoplus_{i} H_1(L_i,\Z) \to H_1(U,\Z)$ is surjective means we can regard $\gamma$ as a loop on $L_i$ for some $i$. We know from Proposition~\ref{prop:monotone-to-d-graded} that the degree on $\gamma$ of the $d$-phase function $\phi_{L_i}^d$ computed using $\eta^d$ is zero.

    Suppose now that $\zeta$ is a trivialization of $K_U$ and let $\psi_{L_i}$ be the phase function for $L_i$ computed using $\zeta$. Then the degree of $\psi_{L_i}$ on $\gamma$ is an integer, and the degree of $\psi_{L_i}^d$ on $\gamma$ is divisble by $d$. The value of the obstruction $b \in H^1(U,\Z/d\Z)$ on the loop $\gamma$ is the reduction modulo $d$ of the difference of degree of $\psi_{L_i}^d$ on $\gamma$ (which is divisible by $d$) and the degree of $\phi_{L_i}^d$ on $\gamma$ (which is zero). Thus the obstruction vanishes on $\gamma$. Since this is true for every $\gamma$, the obstruction vanishes, and we conclude that there is a compatible trivialization $\zeta$ of $K_U$. 

    Now let $L'\subset U$ be another Lagrangian that is exact in $U$ and monotone in $X$. Then the phase function for $L'$ computed with $\zeta$ is a $d$-th root of the $d$-phase function for $L'$ computed with $\eta^d$. Since the latter has degree zero for every loop $\gamma \subset L'$ by Proposition \ref{prop:monotone-to-d-graded}, the former does as well, and so $L'$ has vanishing Maslov class.
  \end{proof}

\subsection{Potentials}
\label{subsec:potential}
We shall now take a first pass at the disk potential;
this is a well-known circle of ideas, see for instance \cite{Au07}. Let $X$ be a positively monotone symplectic manifold with $[\omega] = 2\tau c_1(TX)$, and let $L$ be an oriented spin monotone Lagrangian with $\tau \mu(u) = \int_u \omega$ for all $u \in H_2(X,L)$. We equip $L$ with $\C^*$-local system; because $\C^*$ is abelian, this can be encoded by a homomorphism 
$$\rho \co H_1(L,\Z) \to \C^*.$$ Let $\cM(L,\beta)$ denote the moduli space of disks with boundary on $L$ of Maslov index 2 with a single marked point on the boundary of the disk representing the relative homology class $\beta \in H_2(X,L)$. The virtual dimension of $\cM(L,\beta)$ is $\mu(\beta) + n - 2 = n$, and the moduli space is regular for generic choice of almost-complex structure $J$. There is a map $\mathrm{ev}\co \cM(L,\beta) \to L$ that evaluates at the boundary marked point. We denote by $n_\beta$ the degree of this map. 

\begin{defn}
The \emph{disk potential} of $\bL = (L,\rho)$ is the number:
\begin{equation}
\label{eq:W_def_number}
  W(\bL) = W(L,\rho) = \sum_{\beta \mid \mu(\beta) = 2} n_\beta\cdot  \rho(\partial \beta).
\end{equation}
\end{defn}

In the statement of Theorem~\ref{th:wall_crossing} we have used the subscript $X$ to stress that the the potential is computed using disks inside $X$, but we are omitting it now.
In view of Subsection~\ref{subsec:intro_wc}, let us provide an equivalent viewpoint on the potential.
Denote $m=\rk H_1(L;\Z)/Torsion$ and choose a basis of this group:
\begin{equation}
\label{eq:basis_H}
\Z^m\xrightarrow{\cong} H_1(L,\Z)/\mathit{Torsion}.
\end{equation}

\begin{defn}
	\label{def:pot_LG}
The \emph{Landau-Ginzburg potential} of $L$ with respect to the basis (\ref{eq:basis_H}) is a Laurent polynomial
$$
W_L\co (\C^*)^m\to\C,\quad W_L\in\C[x_1^{\pm 1},\ldots,x_m^{\pm 1}]
$$
in formal variables $x_1,\ldots,x_m$  defined as follows:
\begin{equation}
\label{eq:W_def_poly}
 W_L=\sum_{\beta \mid \mu(\beta) = 2} n_\beta\cdot  \mathbf{x}^{\bd \beta}
 \end{equation}
where we consider $\bd\beta$ as an element of $\Z^m$ using (\ref{eq:basis_H}), and  denote 
$\mathbf{x}^l=x_1^{l_1}\ldots x_m^{l_m}$.
By Gromov compactness and the fact that $L$ is monotone, the above sum is finite, which implies that $W_L$ is a Laurent polynomial.
\end{defn}

The two definitions of the potential agree in the following sense. Let $\rho$ be a local system on $L$; using (\ref{eq:basis_H}), let us consider $\rho$ as a point in $(\C^*)^m$ by evaluating it on the basis elements of $H_1(L,Z)/\mathit{Torsion}$. Then
$$
W_L(\rho)=W(L,\rho)\in\C
$$
where the two sides are taken from (\ref{eq:W_def_poly}) and (\ref{eq:W_def_number}), respectively.

Recall that $GL(m;\Z)$ consists of integral matrices with determinant $\pm 1$.
If one changes the basis (\ref{eq:basis_H}) by a matrix
$
(a_{ij})\in GL(m;\Z),
$
the corresponding potentials differ by a change of co-ordinates given by the multiplicative action of $GL(m;\Z)$ on $(\C^*)^m$:
\begin{equation}
\label{eq:change_coord_W}
\begin{array}{c}
x_i\mapsto x_i'=\prod_{j=1}^m x_j^{a_{ij}}, \quad \textit{ so that}\\
W_L'(x_1,\ldots,x_m)=W_L(x_1',\ldots, x_m').
\end{array}
\end{equation}

 The proposition below is classical.
The reason is that, by monotonicity, $2$ is the lowest possible Maslov index a non-constant holomorphic disk with boundary on $L$ can have, and so bubbling is excluded during a deformation of the almost complex structure or a Hamiltonian motion of $L$.

\begin{prop}[Invariance of potential]
	\label{prop:potential_invariance}
	For an oriented spin monotone Lagrangian submanifold with a $\C^*$-local system  $\bL=(L,\rho)\subset X$, its disk potential $W(\bL)\in \C$ is invariant under the choice of $J$ and Hamiltonian isotopies of $L$.  
	
	Likewise the Landau-Ginzburg potential
	$W_L\in \C[x_1^{\pm 1},\ldots, x_m^{\pm 1}]$ up to the change of co-ordinates (\ref{eq:change_coord_W}) is invariant under Hamiltonian isotopies of $L$, and more generally of symplectomorphisms of $X$ applied to $L$.
\end{prop}

\begin{remark}
	\label{rem:potential}
	The potential also determines part of the \ai endomorphism algebra of $L$, if we look at  $L$ as an object of the monotone Fukaya category. This is proved for the case of toric fibers in toric manifolds \cite{FO310b}, and is expected to hold in general.

        Equip $L$ with a local system $\rho$ and consider the \ai structure maps $\mu^k$ on the self-Floer complex of $\bL=(L,\rho)$ $$CF^*(\bL,\bL)=\hom^*(\bL,\bL)$$ with its Morse grading. As an \ai algebra, this complex is only canonically $\Z/2$ graded, but it still  useful to remember the Morse grading keeping in mind that it is not preserved by the structure maps. Restricting to the degree 1 part of the complex, the contribution from Maslov index 2 disks to the \ai structure gives maps
	\begin{equation}
	\label{eq:mu_k_maslov_2}
	\mu^k\co (\hom^1(\bL,\bL))^{\otimes k}\to\hom^0(\bL,\bL),
	\end{equation}
	and there are no contributions from higher index disks in these degrees.
	Assuming $H_1(L,\Z)$ is torsion-free, the symmetrisations of these operations are equal to the corresponding iterated partial derivatives of $W_L$:
	$$
	\mu^k(\Sym(x_1\otimes \ldots\otimes x_k ))=\bd_{x_1\ldots x_k}|_{\rho}W_L
	$$
	where $x_i$ are elements (in any order, possibly with repetitions) of the chosen basis of $H_1(L,\Z)$.
	
	Moreover, a homological algebra argument shows that when $L$ is an $m$-torus, the maps (\ref{eq:mu_k_maslov_2}) determine the whole \ai algebra of $\bL$. Summing up, the potential of a Lagrangian torus remembers  the whole endomorphism \ai algebra $\hom^*(\bL,\bL)$ up to quasi-isomorphism, i.e.~is the only enumerative geometry invariant of $L$ relevant for the Fukaya category.
	
	When $L$ is a torus, a particularly basic  manifestation of the above discussion is the following fact: the self-Floer cohomology of $\bL$ is non-zero if and only if $\rho\in (\C^*)^m$ is a critical point of $W_L$. In the latter case, the critical value $W_L(\rho)$ must be an eigenvalue of the quantum multiplication by $c_1(X)$ in the small quantum cohomology $QH^*(X,\C)$. 
	
	The potential itself is a much finer invariant than its Fukaya-categorical shadow (\ref{eq:mu_k_maslov_2}). For example, if $\rho$ is a Morse critical point of $W_L$, then the quasi-isomorphism type of $\hom^*(\bL,\bL)$ is uniquely determined, by intrinsic formality of the Clifford algebra. Due to this, for instance, Vianna's monotone Lagrangian tori in $\C P^2$ are not distinguished as objects of the Fukaya category; but Proposition~\ref{prop:potential_invariance} can be used to prove that   they are not Hamiltonian isotopic. 
\end{remark}

\subsection{Disk potential and relative Floer theory}
Now given two monotone Lagrangians with local systems $\bL_i = (L_i, \rho_i), i = 1, 2$, there is a Floer cochain complex $CF(\bL_1,\bL_2)$. This carries a differential $\mu^1$ that counts strips that are rigid modulo translation, weighted by parallel transport of the local system along the boundary arcs. Due to the presence of disks with boundary on $L_i$, $\mu^1$ is not strictly speaking a differential, but rather satisfies the equation
\begin{equation}
  \mu^1(\mu^1(x)) = (W(\bL_2)-W(\bL_1))x, \quad \forall x \in CF(\bL_1,\bL_2).
\end{equation}
It is precisely a refinement of this equation to the relative setting that will establish the general form of the wall-crossing formula.

The precise context in which we will apply relative Floer theory is as follows:

\begin{hypothesis}
\label{hyp:main}
  We consider $(X,D)$ having the following properties:
\begin{itemize}
\item $X$ is a monotone symplectic manifold: $[\omega] = 2\tau c_1(X)$.
\item $D \subset X$ is a smooth symplectic divisor that represents $dc_1(X)$ cut out by a section $\eta^d \in \Gamma(X,K_X^{-d})$.
\end{itemize}
We consider Lagrangians having the following properties:
\begin{itemize}
\item $L$ is a monotone Lagrangian: $\int_\beta \omega = \tau \mu(\beta)$ for any $\beta \in H_2(X,L)$.
\item $L$ is contained in $X\setminus D$, which carries a Liouville structure making $L$ exact.
\item $L$ is orientable and $d$-graded with respect to $\eta^d|_{X\setminus D}$ (see section \ref{sec:frac-gradings}).
\end{itemize}
\end{hypothesis}

Recall that in the case $d = 1$, one has $c_1(X\setminus D) = 0$, and the condition that $L$ be $d$-gradable reduces to the condition that $L$ has vanishing Maslov class.

\begin{rem}
  An important case where these hypotheses are satisfied is the case of \emph{K\"{a}hler pairs} and \emph{anchored Lagrangian submanifolds} as studied by Sheridan \cite{She13}. 
\end{rem}

In this context of Hypothesis \ref{hyp:main}, there is a version of the monotone relative Fukaya category \cite[Definition 3.2]{She13}. The coefficient ring is the polynomial ring $\C[r]$ in one variable. The objects are $d$-graded exact Lagrangians in $X\setminus D$ that are monotone in $X$ as above. The morphism spaces and structure maps are the same as in the usual monotone Fukaya category, with the exception that each holomorphic disk is counted with an extra coefficient $r^{(u\cdot D)}$, where $u\cdot D$ is the intersection number of $u$ with the divisor $D$. There is also a technical condition on the almost complex structure that is not important for the present discussion.\footnote{A small difference between our monotone relative Fukaya category and Sheridan's version is that we take the monotone relative Fukaya category as a curved $A_\infty$ category, whereas Sheridan avoids the curvature by fixing a value of $W$ in advance.}

Now we come to the entire purpose of section \ref{sec:frac-gradings} and the assumption that our Lagrangians are $d$-graded, which is to control the degrees of the structure maps $\mu^k$. We denote by $CF_{X\setminus D}(L_1,L_2)$ the Floer cochains of the exact Lagrangians $L_1,L_2$ in the exact symplectic manifold $X\setminus D$, and we denote by $CF(L_1,L_2) = CF_{X\setminus D}(L_1,L_2) \otimes \C[r]$ the Floer cochains for the monotone relative theory.

\begin{lem}
  \label{lem:mu-k-degree}
  Let $X,D, L_i, i =0,\dots,k$ satisfy hypothesis \ref{hyp:main}. Give $CF_{X\setminus D}(L_i,L_j)$ the grading by $Z_d \subset \Q$ coming from the $d$-gradings of $L_i, L_j$. Give $CF(L_1,L_2) = CF_{X\setminus D}(L_1,L_2) \otimes \C[r]$ a $Z_d$-grading by declaring that $r$ has degree $2/d$. Then the structure map
  \begin{equation}
    \mu^k \co CF(L_{k-1},L_k) \otimes \cdots \otimes CF(L_0,L_1) \to CF(L_0,L_k)
  \end{equation}
  has degree $2- k$.
\end{lem}

\begin{proof}
  The basic idea is to work one disk at a time. Let us assume first that $d$ is even. Let $u \co S \to X$ be a map contributing to the coefficient of $x_0$ in $\mu^k(x_k,x_{k-1},\dots,x_1)$, where $S$ is a disk with $k+1$ boundary points removed. Let $P = u^{-1}(D)$ be the divisor in the domain where the map touches $D$ (this is a formal linear combination of interior points since the boundary of $S$ maps to $\cup L_i \subset X\setminus D$). Then over $S \setminus \supp P$, we have already chosen a section $\eta^d$ trivializing $u^*K_X^{-d}$. On the other hand, since $S$ is contractible it is possible to trivialize $u^*K_X^{-2}$ by some section $\zeta^2$. Then $(\zeta^2)^{d/2}$ is a trivialization of $u^*K_X^{-d}$, which differs from $\eta^d$ over $S\setminus \supp P$. The difference between these two homotopy classes of trivializations of $u^*K_X^{-d}$ is measured by a class in $H^1(S \setminus \supp P, \Z)$. An inspection shows that the class is one whose value on the loop parallel to the boundary of $S$ is $\deg P = (u \cdot D)$. 
Letting $i(x)$ denote the absolute indices computed with respect to $\eta^d$ and the given $d$-gradings, and letting $i'(x)$ denote the indices computed with respect to the auxiliary choice of $\zeta^2$, we then find
  \begin{equation}
    i(x_0) - \sum_{s = 1}^k i(x_s) + (2/d)(u\cdot D)= i'(x_0) - \sum_{s = 1}^k i'(x_s)    
  \end{equation}
Since the only maps that contribute to the $\mu^k$ operation are the ones where the right-hand side equals $2-k$, the result follows.
\end{proof}


We will treat the monotone relative Fukaya category as a curved $A_\infty$ category. This means that, in addition to the usual $A_\infty$ operations $\mu^1, \mu^2, \dots$, there is a curvature term $\mu^0 \in CF(L,L)$ for each object $L$. This operation is nothing but the count of Maslov index 2 disks with boundary on $L$, so it is another avatar of the disk potential. 
\begin{lem}
\label{lem:muzero}
  For an object $\bL = (L,\rho)$ of the monotone relative Fukaya category, 
  \begin{equation}
    \mu^0_\bL = W(\bL)1_\bL r^d.
  \end{equation}
  (Recall that $W(\bL) \in \C$ is the disk potential, $1_L \in CF(\bL,\bL)$ is the unit, $d$ is the number for which $[D] = dc_1(TX)$.)
\end{lem}
\begin{proof}
  By Lemma \ref{lem:mu-k-degree}, we know that $\mu^0_\bL \in CF(\bL,\bL)$ has degree $2$. Hence, in the expansion
  \begin{equation}
    \mu^0_\bL = \sum_i a_i r^i,
  \end{equation}
  $a_i \in CF_{X\setminus D}(\bL,\bL)$ represents the disks $u$ with boundary on $L$ such that $(u\cdot D) = i$, and $2 = \deg(a_i) + (2/d)i$. So $\deg(a_i) = 2 - (2/d)i$. Since $\deg(a_i)$ is a non-negative integer, we know $a_i = 0$ except for possibly $i = 0, d/2, d$, and $d/2$ is only possible if $d$ is even. Since $L$ is exact in $X\setminus D$, we have $a_0 = 0$. Since $L$ is orientable, $a_{d/2} = 0$, for any disk that contributed to $a_{d/2}$ would have Maslov index $1$. Thus in all cases $\mu^0_\bL = a_d r^d$, where $a_d$ is the class counting Maslov index 2 disks.
\end{proof}

Generally speaking, the other operations $\mu^s, s \geq 1$ have no particular constraint on what powers of $r$ can appear, but we can always write them as polynomials in $r$,
\begin{equation}
  \mu^s = \mu^s_0 + \sum_i \mu^s_i r^i,
\end{equation}
We have separated out the constant term $\mu^s_0$; this terms counts holomorphic curves whose intersection number with $D$ is zero, and that therefore are contained in $X\setminus D$.

The key point is to expand the curvature equation
\begin{equation}
\label{eq:curvature}
  \mu^1(\mu^1(x)) = \mu^2(\mu^0_{\bL_2},x) - \mu^2(x,\mu^0_{\bL_1}), \quad x \in CF(\bL_1,\bL_2),
\end{equation}
in powers of $r$. 
By Lemma \ref{lem:muzero}, the right-hand side of Equation \eqref{eq:curvature} simplifies to 
\begin{equation}
  \mu^2(\mu^0_{\bL_2},x) - \mu^2(x,\mu^0_{\bL_1}) = (W(\bL_2)-W(\bL_1))xr^d
\end{equation}
Taking Equation \eqref{eq:curvature} modulo $r$ says that 
\begin{equation}
  \mu^1_0(\mu^1_0(x)) = 0,
\end{equation}
and indeed, we know $\mu^1_0$ is the differential whose cohomology computes $HF_{X\setminus D}(\bL_1,\bL_2)$. 

Now let us take Equation \eqref{eq:curvature} modulo $r^{d+1}$. Define $P = \sum_{i=1}^d \mu^1_i r^i$, so that
\begin{equation}
  \mu^1 \equiv \mu^1_0 + P \pmod{r^{d+1}}
\end{equation}
Since $(\mu^1_0)^2 = 0$, we find that 
\begin{equation}
  \mu^1_0(P(x)) + P(\mu^1_0(x)) + P(P(x)) \equiv (W(\bL_2)-W(\bL_1))xr^d \pmod{r^{d+1}}
\end{equation}
We interpret this equation as follows: we have a chain complex $(CF(\bL_1,\bL_2)/(r^{d+1}), \mu^1_0)$. In this chain complex, $P$ is a homotopy operator showing that the map
\begin{equation}
  (W(\bL_2)-W(\bL_1))r^d\cdot\Id
\end{equation}
is chain homotopic to $P^2$. If we assume that $P^2 \equiv 0 \pmod{r^{d+1}}$, we then get that a scalar multiplication is chain homotopic to zero. This implies that either the scalar is zero or the cohomology is zero.

\begin{thm}[Preliminary wall-crossing formula]
	\label{thm:wall_cross_P2}
  Assume that $P^2 \equiv 0 \pmod{r^{d+1}}$.
  Then
  \begin{equation}
    HF_{X \setminus D}(\bL_1,\bL_2) \neq 0 \implies W(\bL_1) = W(\bL_2).    
  \end{equation}
\end{thm}

\begin{proof}
  The discussion preceding the theorem shows that under the hypothesis $P^2 = 0$, the map $(W(\bL_2)-W(\bL_1))r^d\cdot\Id$ is homotopic to zero, and hence it induces the zero map on the cohomology of the complex
  \begin{equation}
    (CF(\bL_1,\bL_2)/(r^{d+1}), \mu^1_0)
  \end{equation}
  where the coefficient ring is $\C[r]/(r^{d+1})$. But this complex is isomorphic to
  \begin{equation}
    (CF_{X\setminus D}(\bL_1,\bL_2)\otimes \C[r]/(r^{d+1}),\mu^1_0).
  \end{equation}
  where $CF_{X\setminus D}(\bL_1,\bL_2)$ has coefficient ring $\C$. Its cohomology is therefore nothing but
  \begin{equation}
    HF_{X\setminus D}(\bL_1,\bL_2) \otimes \C[r]/(r^{d+1}).
  \end{equation}
  Thus if $HF_{X\setminus D}(\bL_1,\bL_2) \neq 0$, there is a class $[x]$ such that $r^d [x] \neq 0$, but such that $(W(\bL_2)-W(\bL_1))r^d [x] = 0$. This shows $W(\bL_1) = W(\bL_2)$.
\end{proof}

Outside of the anticanonical case where the hypothesis $P^2\equiv 0 \pmod{r^{d+1}}$ holds automatically, the most natural way to verify this hypothesis is to show that $P$ is divisible by a sufficiently high power of $r$. In fact, there are general grading conditions under $P$ has only the top term proportional to $r^d$.

\begin{lem}
  \label{lem:Z_graded_implies_P2}
  The condition $P^2 \equiv 0 \pmod{r^{d+1}}$ is satisfied under any of the following hypotheses.
  \begin{enumerate}
  \item $d=1$, that is, $D$ is anticanonical.
  \item $CF_{X\setminus D}(\bL_1,\bL_2)$ is relatively $\Z$-graded, that is, the difference of degrees of any nonzero homogeneous elements lies in $\Z$.
  \item $\bL_1$ is isomorphic to a shift of $\bL_2$ as objects of the Fukaya category of $X \setminus D$.
  \item There is an open submanifold $U \subset X \setminus D$ containing $L_1$ and $L_2$ such that $c_1(U) = 0$, and $U$ and $X \setminus D$ admit compatible trivializations in the sense of Definition \ref{def:compatible-triv}.
  \end{enumerate}
\end{lem}

\begin{proof} 
  To see that $P^2 \equiv 0 \pmod{r^{d+1}}$ holds if $d = 1$, simply observe that $P$ is divisible by $r$, so $P^2$ is divisible by $r^2$.

  To see sufficiency of the second condtion, recall that $\mu^1$ has total degree 1 when $r$ is given degree $2/d$. 
  Thus the degree of $\mu^1_k$ is $1 - k(2/d)$. Because $CF_{X\setminus D}(\bL_1,\bL_2)$ is relatively $\Z$-graded, this number must be an integer. Because $L_1$ and $L_2$ are assumed orientable, and strips can only connect intersection points of opposite pairity, it must in fact be an odd integer. Clearly, $1-k(2/d)$ is an odd integer if and only if $k$ is divisible by $d$. Thus all powers of $r$ appearing in $P$ are divisible by $d$, and so $P^2$ is divisible by $r^{2d}$. 
  
  We now show the third condtion implies the second. Suppose $\bL_2 \cong \bL_1[\alpha]$, where $\alpha \in Z_d \subset \Q$. Then $CF_{X\setminus D}(\bL_1,\bL_2) \cong CF_{X\setminus D}(\bL_1,\bL_1)[\alpha]$, which is relatively $\Z$-graded because $CF_{X\setminus D}(\bL_1,\bL_1)$ is.

  We show that the fourth condition implies the second. Suppose $\zeta$ is a trivializatino of $K_U$ such that $\zeta_U^d$ and $\eta^d|_U$ are homotopic trivializations of $K_U^{\otimes d}$. Because $L_i$ is $d$-graded with respect to $\eta^d$ in $X\setminus D$, it is has Maslov class zero with respect to $\zeta$. Thus $CF_U(\bL_1,\bL_2)$ is $\Z$-graded. The absolute indices are the same when computed with respect to $\zeta_U$ and with respect to $\eta^d$, so $CF_{X\setminus D}(\bL_1,\bL_2)$ is $\Z$-graded as well.
\end{proof}

\section{Locality of compact branes}
\label{sec:locality}

\subsection{Locality of compact branes}
We now present a basic result that expresses a sense in which the Floer theory of compact branes is local with respect to inclusions of Liouville manifolds. The proof is a straightforward application of the maximum principle.

Given a Liouville domain $U$, we denote by $\Fc(U)$ the Fukaya category consisting of branes supported on compact Lagrangian submanifolds of the interior of $U$ (as opposed to infinitesimal or wrapped versions of the Fukaya category that admit certain noncompact branes). Given two Liouville manifolds $U, V$ and a Liouville embedding $j\co U \to V$, there is a pushforward functor 
\begin{equation}
  j_* \co \Fc(U) \to \Fc(V).
\end{equation}
This functor takes a brane supported on a Lagrangian $L \subset U$ to one supported on $j(L) \subset V$. There are also auxiliary brane structures, namely a local system and a spin structure on $L$, as well as a grading. The local system and spin structure are simply transferred by the map $j$. As explained in section \ref{sec:frac-gradings} above, the choice of grading on a Lagrangian is always relative to a choice of grading on the ambient symplectic manifold, so one needs to assume that the embedding $j\co U \to V$ admits compatible trivializations in the sense of Definition \ref{def:compatible-triv}. However, we can avoid this issue by reducing the grading to $\Z/2\Z$.

To define $j_*$ on morphism complexes, we use a particular type of complex structure, namely one which is cylindrical near the real hypersurface $Y = j(\partial U) \subset V$. Having chosen such a complex structure, consider the differential, or indeed any $A_\infty$ operation, where the boundary conditions are drawn from $\Fc(U)$. We claim that any holomorphic curve $C$ contributing to this operation in $V$ must in fact be contained in $U$. Indeed, by considering $C' = C \cap (V \setminus j(\interior(U)))$, we obtain a compact holomorphic curve $C'$ in $V \setminus \interior(U)$ whose boundary is contained in $Y$. By our exactness assumptions and choice of cylindrical complex structure near $Y$, the integrated maximum principle \cite[Lemma 7.2]{ASei10} then implies that $C'$ has non-positive area, so it must be constant, and hence contained in $Y$. Thus the original $C$ is contained in $U$.

Thus, with this choice of almost complex structure, we can arrange that the $A_\infty$ operations involving objects drawn from $\Fc(U)$ are the same whether they are computed in $U$ or in $V$. So we can take $j_*$ to be the identity on morphism complexes, and we can furthermore take all higher components of this $A_\infty$ functor to be trivial. 

\begin{thm}
  \label{thm:nonvanishing}
  Suppose $j\co U \to V$ is an embedding of Liouville manifolds, and either work with $\Z/2\Z$-gradings or suppse $U$ and $V$ carry compatible trivializations. Then the functor
  \begin{equation}
    j_* \co \Fc(U) \to \Fc(V)
  \end{equation}
  is cohomologically full and faithful. In fact, if $\bL_0, \bL_1 \in \Ob(\Fc(U))$ are two branes, one can arrange that $CF^*_U(\bL_0,\bL_1)$ is chain isomorphic to  $CF^*_V(j_*\bL_0,j_*\bL_1)$.
\end{thm}


In light of Theorem \ref{thm:wall_cross_P2}, one can derive relationships between potentials whenever we have a pair $(\bL_0,\bL_1)$ such that $HF_{X\setminus D}(\bL_1,\bL_2)$ is non-vanishing. By Theorem \ref{thm:nonvanishing}, it suffices to check this non-vanishing in an arbitrarily small Liouville subdomain $U$ of $X\setminus D$ containing both $\bL_0$ and $\bL_1$. This is what lets us derive consequences from analysis of local models.

\subsection{Proof of wall-crossing}
We are now ready to finish the proof of the wall-crossing formula.

\begin{proof}[Proof of  Theorem~\ref{th:wall_crossing}] Let $\bL_i = (L_i,\rho_i)$, $i=1,2$, $U, X, D$ be as in the statement of the theorem. By Proposition \ref{prop:monotone-to-d-graded}, the $L_i$ are $d$-gradable. From the hypothesis $HF^*_U(\bL_1,\bL_2) \neq 0$ and Theorem \ref{thm:nonvanishing}, we find that $HF^*_{X \setminus D}(\bL_1,\bL_2) \neq 0$. Then the result follows from the combination of Theorem \ref{thm:wall_cross_P2} and Lemma \ref{lem:Z_graded_implies_P2}: the case $d = 1$ uses part (1), the case $d = 2$ uses part (2), the condition that $H^1(U,\Z) = 0$ implies that $U$ and $X\setminus D$ carry compatible trivializations, and in that case we can apply part (4).
\end{proof}

\subsection{Divisors in the complement of a Lagrangian skeleton}
\label{subsec:donaldson}

Let $(X,\omega)$  be a monotone symplectic manifold. Recall that a Donaldson divisor $D\subset X$ is a real codimension 2 smooth symplectic submanifold whose homology class is Poincar\'e dual to $dc_1(X)\in H^2(X,\Z)$ for a positive integer $d\in\N$ called the degree. They are precisely the divisors appearing in Theorem~\ref{th:wall_crossing}. Donaldson proved \cite{Do96} that such hypersurfaces always exist, for large enough $d$, and his construction has several upgrades showing that $D$ can be chosen so as to satisfy various additional properties. Auroux, Gayet and Mohsen \cite{AGM01} showed that one can find $D$ disjoint from any (closed) Lagrangian submanifold; more generally, the Lagrangian submanifold may be immersed and have boundary; see \cite{Mo13,CW13}. Furthermore, Charest and Woodward \cite{CW13} proved that one can find a Donaldson divisor making a given {\it strongly rational} Lagrangian submanifold exact in its complement.

\begin{theorem}[Version of the Auroux-Gayet-Mohsen theorem, {\cite[Lemma 4.11]{CW13}}]
	\label{th:donaldson_hyp}
	Let $L\subset X$ be a monotone Lagrangian submanifold such that  $H_1(L,\R)\to H_1(X,\R)$ vanishes. Then there exists a Donaldson hypersurface $D$ of sufficiently high degree which disjoint from $L$, and such that  $L\subset X\setminus D$ is exact for a Liouville 1-form on $X\setminus D$.\qed
\end{theorem}

Note that \cite[Lemma 4.11]{CW13} speaks of  strongly rational Lagrangians; a monotone Lagrangian $L$ with the property that $H_1(L,\R)\to H_1(X,\R)$ vanishes is strongly rational. The statement carries over to the immersed case; the proof presented of the following extension of \cite{AGM01} is due to Denis Auroux.

\begin{theorem}
Let $N_1,N_2\subset X$ be two cleanly intersecting isotropic submanifolds of a symplectic manifold $X$. Then there is a Donaldson divisor $D\subset X$ disjoint from $N_1\cup N_2$.
\end{theorem}


  \begin{proof}
    First we recall the steps in the construction of a Donaldson divisor disjoint from a single smooth isotropic submanifold $N \subset X$. The starting point is a Hermitian line bundle $\cL \to X$ with curvature $-i \omega$. Because $N$ is isotropic, the Hermitian connection is flat restricted to $N$, and so some power $\cL^{\otimes k}|_N$ is topologically trivial. Then \cite[Lemma 2]{AGM01} we can find a nowhere vanishing section of $\cL^{\otimes k}|_N$ which has norm $1$ and is nearly covariantly constant. Then we choose a mesh of points in $N$, and sum together a collection sections peaked at those points.
    A key step is \cite[Lemma 4]{AGM01}, which says that the peaked sections at nearby points of the mesh have similar arguments, and so cannot cancel completely. This sum of peaked sections is then used as the input to Donaldson's transversalization process.

    In the case where $N = N_1 \cup N_2$ consists of two cleanly intersecting Lagrangian submanifolds, the proof that the sum of peaked sections is bounded away from $0$ along $N$ is different. It involves using the previous construction along the three isotropic submanifolds, $N_1$, $N_2$, and the intersection $N_{12} \coloneqq N_1\cap N_2$. First apply the previous argument to $N_{12}$. This produces a section $s_{12}$ that is everywhere bounded above by $1$, bounded below on $N_{12}$ by $c_1$, and which at small distance $d$ from $N_{12}$ is bounded above by $\exp(-r_0 k d^2)$ and below by $c_1 \exp(-r_1 k d^2)$, for positive constants $c_1, r_0, r_1$ satisfying bounds that, for large $k$, depend only on the dimension and the process used to construct the mesh. Because $N_1$ and $N_2$ intersect cleanly along the compact manifold $N_{12}$, there is a lower bound on their transversality angle, which means that there is a constant $c_2 > 0$ so that a pair points $p \in N_1$ and $q \in N_2$ at distance $d$ from $N_{12}$ are at least distance $c_2d$ apart. Pick $\delta > 0$ such that
    \begin{equation*}
      \exp(-r_0 c_2^2 \delta^2) < c_1 /10,
    \end{equation*}
    and define
    \begin{equation*}
      c_{3} = c_{1} \exp(-r_{1} \delta^{2})/10.
    \end{equation*}
    Next, sum together peaked sections, rescaled by $c_{3}$, along $N_{1}$ and $N_{2}$ to build sections $s_{1}$ and $s_{2}$, respectively, which are bounded above by $c_{3}$ and below by $c_{1}c_{3}$. We consider $s = s_{12} + s_{1} + s_{2}$ as our candidate section.

    Now \cite[Lemma 4]{AGM01} implies that the phases of $s_{12}$ and $s_{1}$ are similar along $N_{1}$, so $s_{12}$ and $s_{1}$ do not cancel out along $N_{1}$. We claim that $s_{12} + s_{1}$ and $s_{2}$ do not cancel out along $N_{1}$. Within distance $k^{-1/2}\delta$ of $N_{12}$, $s_{12}$ is bounded below by $c_{1}\exp(-r_{1} \delta^{2})$, while $s_{1}$ and $s_{2}$ are bounded above by one-tenth of this quantity by our choice of $c_{3}$, so the sum is bounded away from zero. At distances greater than $k^{-1/2}\delta$ from $N_{12}$, the two submanifolds are at least $c_{2}k^{-1/2}\delta$ apart from each other, so on $N_{1}$ we have that $s_{12}+s_{1}$ is bounded below by $c_{1}c_{3}$, while $s_{2}$ is bounded above by $c_{3}\exp(-r_{0}c_{2}^{2}\delta^{2})$; by our choice of $\delta$, the latter is at most one-tenth the former. Thus $s$ is indeed bounded away from zero on $N_{1}$. A similar argument shows that $s$ is bounded away from zero on $N_{2}$. Thus $s$ can indeed be used as the input to Donaldson's transveralization process to produced the desired divisor.
  \end{proof}


\begin{corollary}
	\label{cor:donaldon_mutation_config}
	Let $\cL\subset X$ be the union of a monotone Lagrangian two-torus $L$ and a Lagrangian disk attached to $L$ cleanly along its boundary. Then there exists a Donaldson divisor $D\subset X$ disjoint from $\cL$, and such that $L\subset X\setminus D$ is exact for the standard Liouville structure on $X\setminus D$.
\end{corollary}

\begin{proof}
	The monotonicity (and more generally, strong rationality) condition on $L$ implies that $\cL^{\otimes k}$ restricts to a trivial line-bundle-with-connection \cite{CW13} when $k$ is divisible by a fixed constant, so one can start the previous proof with a section of that line bundle which is strictly covariantly constant over $L$. If one proceeds as in the previous proof taking $N_1=L$ and $N_2$ to be the disk, the outcome is that $L$ becomes exact in the complement of the obtained divisor: this is checked analogously to \cite[Theorem 3.6]{CW13}.
\end{proof}

\section{Mutations of two-tori and Landau-Ginzburg seeds}
\label{sec:mut_4d}	
In this section we discuss the algebra and geometry behind the simplest type of wall-crossing: the one associated with mutations of two-dimensional Lagrangian tori.
We begin the section by reviewing the algebraic part of the story, which was explored in the papers of Cruz~Morales, Galkin and Usnich \cite{GU12,CMG13,CM13} and is a special case of the general theory of cluster algebra \cite{FZ02}. The algebraic package provides the notion of a Landau-Ginzburg seed and its mutations, and proves the Laurent phenomenon which guarantees that seeds can be mutated indefinitely.

Further in this section, we describe a geometric counterpart of the mentioned package which involves geometric mutations of Lagrangian tori, and is inspired by the recent work of Shende, Treumann and Williams \cite{STW15}. We explain that the wall-crossing formula establishes the link between the algebraic and geometric stories, and discuss some new applications of this newly established link.

\subsection{Mutations and seeds}
All integral vectors below are assumed to be primitive.
Let $(x,y)\in(\C^\x)^2$ be the standard co-ordinates. By a Laurent polynomial, we mean a finite sum of monomials $cx^{a}y^{b}$, $a,b\in\Z$, $c\in\C$.

\begin{definition}
Let $v=(v_1,v_2)\in \Z^2$ be an integral vector. In this section, we call  the following birational map  the \emph{wall-crossing map}, or the \emph{cluster transformation}:
\begin{equation}
	\label{eq:wall_cross_map}
	\mu_{v}\co (\C^\x)^2\dashrightarrow (\C^\x)^2,\quad
	\mu_{v}(x,y)= \left(x\left(1+x^{v_2}y^{-v_1}\right)^{-v_1},\   y\left(1+x^{v_2}y^{-v_1})\right)^{-v_2}\right).
\end{equation}
\end{definition}

\begin{definition}
	\label{def:mut_function}
	Let $W(x,y)\co (\C^\x)^2\to\C$ be a Laurent polynomial and $v\in\Z^2$. The \emph{mutation} of $W$ in the direction of $v$, or along $v$, is the function $\mu_vW$ (which not necessarily a Laurent polynomial) given by:
	\begin{equation}
		\label{eq:mut_functn}
		(\mu_vW)(x,y)=W(\mu_v(x,y))
	\end{equation}
\end{definition}

\begin{remark}
To mutate a Laurent polynomial $W$, one can replace each monomial of $W$ by the following expression:
\begin{equation}
	\label{eq:mutation_monomial}
	x^{u_1}y^{u_2}\mapsto x^{u_1}y^{u_2}\left(1+x^{v_2}y^{-v_1}\right)^{-u_1v_1-u_2v_2}
\end{equation}
Taking single monomial $x^{u_1}y^{v_2}$, we see that its mutation is  a Laurent polynomial if and only if the number $-u_1v_1-u_2v_2$ is non-negative. For more complicated Laurent polynomials, it is more tricky to understand under which mutations it stays Laurent.
\end{remark}

\begin{remark}
	\label{rem:sl2_acts}
There is a multiplicative action of $GL(2,\Z)$ on $(\C^\x)^2$, where a matrix 
$\left(
\begin{smallmatrix}
a&b\\
c&d
\end{smallmatrix}
\right)$
with determinant $\pm 1$ acts by the following automorphism,
compare with (\ref{eq:change_coord_W}):
\begin{equation}
\label{eq:sl2_acts}
(x,y)\mapsto(x^ay^b,x^cy^d)
\end{equation}
The wall-crossing maps respect this action. For example, the wall-crossing map along the vector $(0,1)$ is:
\begin{equation}
\label{eq:wall_cross_map_stand}
(x,y)\mapsto\left(x,y(1+x)^{-1}\right)
\end{equation}
For a $GL(2,\Z)$ transformation taking a given vector $v$ to $(0,1)$, the wall-crossing map (\ref{eq:wall_cross_map}) will be the composition of (\ref{eq:wall_cross_map_stand}) and (\ref{eq:sl2_acts}). 
\end{remark}

\begin{definition}
	\label{def:mut_trop}
	Let $u=(u_1,u_2)$ and $v=(v_1,v_2)$ be  vectors in $\Z^2$, where $v$ is primitive. The  \emph{(tropical) mutation} of $u$ along $v$ is given by
	\begin{equation}
		\label{eq:mut_trop}
		\mu_{v}u=u+\max(0,
			u_1v_2-u_2v_1
	)\cdot v\in\Z^2.
	\end{equation}
\end{definition}

\begin{remark}
Formula (\ref{eq:mut_trop}) is a tropicalisation of (\ref{eq:mutation_monomial}), provided that in (\ref{eq:mutation_monomial}) one uses $v^\perp=(-v_2,v_1)$ instead of $v$.
\end{remark}

\begin{definition}
	\label{def:LG_seed}
	A \emph{Landau-Ginzburg  seed} $(W,\{v_1,\ldots, v_k\})$
	is a tuple consisting of a Laurent polynomial $W(x,y)\in \C[x^{\pm1},y^{\pm 1}]$,
	and a collection of primitive integral vectors $v_i\in\Z^2$ called directions, such that  each mutation
	$
	\mu_{v_i}W
	$
	is also a Laurent polynomial. 
	
	The vectors $v_i$ are allowed to repeat, and in the case when a vector $v$ is found $p$ times in the collection $\{v_1,\ldots,v_k\}$, we require that
	$$
	\mu_v(\ldots(\mu_v W)))\quad\text{($p$ times)}
	$$
	is a Laurent polynomial.
\end{definition}

\begin{definition}
	\label{def:mut_LG}
	Let $(W,\{v_1,\ldots,v_k\})$ be an LG seed. Its \emph{mutation} in the $j$'th direction is the following tuple:
	\begin{equation}
		\label{eq:mutation_seed}
		\left(\mu_{v_j} W, \{\mu_{v_j} v_1,\ldots, \mu_{v_j} v_{j-1},\ -v_j,\  \mu_{v_j}v_{j+1},\ldots, \mu_{v_j}v_k\}\right).
	\end{equation}
We will denote the mutated tuple by $\mu_j(W,\{v_i\})$. 
See Section~\ref{subsec:proof_mut} for examples.	
\end{definition}

Cruz~Morales and Galkin \cite{CMG13} (see also the thesis of Cruz~Morales \cite{CM13}) proved the theorem below, which is
a version of the Laurent phenomenon known in cluster algebra.

\begin{theorem}[Laurent phenomenon for LG models, Galkin and Cruz~Morales]
	\label{th:Laurent_phen}
	Any mutation  of an LG seed is an LG seed on its own right.\qed 
\end{theorem}

The content of the theorem is that if
$\mu_{v_j}W$ is a Laurent polynomial
for all $j$, then $W$ will forever stay a Laurent polynomial under iterated mutations, provided that the list of vectors that one uses for mutation is being modified according to (\ref{eq:mutation_seed}).

Starting with a LG seed, we can therefore  mutate it indefinitely, having $k$ choices of mutation on each step. This way, one gets infinite collections of Laurent polynomials, and in many cases these collections will contain infinitely many Laurent polynomials up to the  $SL(2,\Z)$-action (\ref{eq:sl2_acts}). 

\begin{remark}
	\label{rem:one_dir}
	Mutating a seed $(W,\{v\})$ that has a single direction does not produce infinitely many Laurent polynomials up to the $GL(2,\Z)$ action. To see this, first apply mutation once to` get the seed $(\mu_vW,\{-v\})$. The composition of the wall-crossing maps $\mu_{-v}\circ \mu_v$ equals to the action (\ref{eq:sl2_acts}) of a $GL(2,\Z)$ matrix which is conjugate to:
	$$
	\left(
	\begin{matrix}
	1&0\\
	1&1
	\end{matrix}
	\right).
	$$
So twice-repeated mutation   brings the seed back to $(W,\{v\})$ up to the  $GL(2,\Z)$ action.
\end{remark}

\begin{remark}
LG seeds and their mutations are an instance of the general notions of cluster algebra \cite{FZ02}. At a first glance, there seems to be a slight difference since in cluster algebra, the number of directions used for mutation must coincide with the number of variables, while this is not a requirement for an LG seed. However, Gross, Hacking and Keel 
explain that LG seeds can modelled in the language of classical cluster algebra by introducing auxiliary variables. \end{remark}

\setlength{\extrarowheight}{0.3em}
\begin{table}[h]
	\begin{tabularx}{\textwidth}{c l Xl l}
		Del Pezzo surface & Potential & Directions\\	
		$\C P^2$& $x+y+\frac 1 {xy}$&$(1,1)$, $(-2, 1)$, $(1, -2)$
		\\
		
		$\C P^1\times \C P^1$&$x+y+\frac 1 x+\frac 1 y$ &$(1,1)$, $(1,-1)$,
		\newline  $(-1,1)$, $(-1,-1)$\\
		
		$\Bl_1\C P^2$&$x+ y +\frac 1 {xy}+ xy$ &
		$(-2,1)$, $(1,-2)$,
		\newline
		$(1,0)$, $(0,1)$ \\
		
		$\Bl_2\C P^2$&
		$(1+x+y)(1+\frac 1 {xy})-1$ &
		$(1,-1)$, $(-1,1)$,
		\newline
		$(-1,0)$, $(0,-1)$,
		\newline
		$(1,1)$
		\\
		$\Bl_3\C P^2$&$ (1+x)(1+y)(1+\frac 1 {xy})-2$ &$\pm (1,-1)$, $\pm(1,0)$ $\pm(0,1)$\\
		
		$\Bl_4\C P^2$&$(1+x+y)(1+\frac 1 x)(1+\frac 1 y)-3
		$ &$(-1,0)\times 2$, $(0,-1)\times 2$ \newline $(1,0)$,
		$(0,1)$,
		$(1,1)$ \\
		
		$\Bl_5\C P^2$& $(1+x)^2(1+y)^2/xy-4$ &$(1,0)\times 2$, $(0,1)\times 2$,
		\newline
		$(-1,0)\times 2$, $(0,-1)\times 2$
		\\
		
		$\Bl_6\C P^2$&$(1+x+y)^3/xy-6$ &$(1,0)\times 3$, $(0,1)\times 3$,
		\newline $(-1,-1)\times 3$\\
		
		$\Bl_7\C P^2$&$(1+x+y)^4/xy-12$ &
		$(-1,0)\times 4$, $(0,-1)\times 4$,
		\newline $(1,1)\times 2$
		\\
		
		$\Bl_8\C P^2$&$(1+x+y)^6/
		xy^2-60$ &
		$(-1,0)\times 6$, $(0,-1)\times 3$,
		\newline
		$(1,1)\times 2$
		\\
	\end{tabularx}
	\caption{Del Pezzo seeds.}
	\label{tab:dp}
\end{table}	

\subsection{Del Pezzo seeds}
Recall that the list of del Pezzo surfaces consists of blowups $\Bl_k\C P^2$ at $0\le k\le 8$ points, and additionally $\C P^1\times\C P^1$.
Galkin and Usnich \cite{GU12} have written down LG potentials associated to del Pezzo surfaces. In Table~\ref{tab:dp}, we reproduce those potentials and  complement them by collections of integral vectors $v_i$ that, together with those potentials, comprise LG seeds. We call them  del~Pezzo seeds.

The first four del~Pezzo surfaces in the table are toric, and the function is the classical toric potential in those cases. In the non-toric cases, the associated potentials can be traced back to the work of Hori and Vafa \cite{HV00}.

A remark about notation in Table~\ref{tab:dp} is due: when the list of directions contains a repeated vector, e.g.~$(-1,0)$ twice, we denote this by $(-1,0)\times 2$. The list of directions for the seed associated with $\Bl_k \C P^2$ contains $k+3$ directions. 

\subsection{Lagrangian seeds and the mutation theorem}
We  move on to the geometric part of the story.

\begin{definition}
\label{def:Lag_seed}
A \emph{Lagrangian seed} $(L,\{D_i\})$ in a symplectic 4-manifold $X$ consists of
a monotone Lagrangian torus $L\subset X$, and
a collection of embedded Lagrangian disks $D_i\subset X$ with boundary on $L$, which satisfy the following conditions. Here we denote $D_i^\mathrm{o}=D_i\setminus\bd D_i$.

\begin{itemize}
\item each $D_i$ is attached to $L$ cleanly along its boundary, i.e.~transversely in the directions complementary to the tangent lines $T(\bd D_i)$,
\item $D_i^\mathrm{o}\cap L=\emptyset$,
\item 
$D_i^\mathrm{o}\cap D_j^\mathrm{o}=\emptyset$, $i\neq j$,
\item the curves $\bd D_i\subset L$ have minimal pairwise intersections, i.e.~ there is a diffeomorphism $L\to T^2$ taking each $\bd D_i$ to a geodesic of the flat metric. 
\end{itemize}
\end{definition}

\begin{figure}[h]
	\includegraphics[]{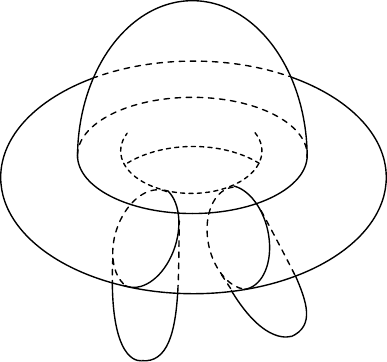}
\caption{An example of a Lagrangian seed consisting of a Lagrangian torus and three attached Lagrangian disks.}
\label{fig:Lag_skeleton}
\end{figure}	

\begin{remark}
One can think of a Lagrangian seed as a Lagrangian embedding of the $CW$-complex shown obtained in the following way, see Figure~\ref{fig:Lag_skeleton}. Fix a 2-torus,  pick a collection of flat geodesic curves in some homology classes $v_i$ (in case when $v_i=v_j$, we pick distinct parallel geodesics), and attach a 2-disk along each curve.
The only data needed here is the collection of homology classes $v_i$, which obviously become $[\bd D_i]$ after the construction is performed.
Definition~\ref{def:Lag_seed} describes a nice Lagrangian embedding of this $CW$-complex into $X$.
Such skeleta have been studied by Shende, Treumann and Williams \cite{STW15}, and we will soon turn to their results. 
\end{remark}

Recall that since $L$ is monotone, its potential $W_L$  is a Laurent polynomial, and is an invariant of $L$. Therefore, fixing a basis for $H_1(L,\Z)$, one can associate to any Lagrangian seed the following tuple:
$$
(L,\{D_i\}) \leadsto (W_L,\{[\bd D_i]^\perp\})
$$
which has the right format for being an LG~seed (see Definition~\ref{def:LG_seed}) in the sense that it consists of a Laurent polynomial and a list of vectors in $\Z^2$. 
Our next claim will say that this is indeed an LG~seed, meaning that the mutations $\mu_{[\bd D_i]^\perp }W_L$ are also Laurent  polynomials. Recall that
$$
(v_1,v_2)^\perp=(-v_2,v_1).
$$
We point out that the classes $[\bd D_i]^\perp$  \emph{not} $[\bd D_i]$ produce what turns out to be an LG~seed, in the algebraic conventions that we set up above.

\begin{theorem}[Mutation theorem]
	\label{th:lag_mut}
Let $(L,\{D_i\}_{i=1}^k)$ be a Lagrangian seed in a monotone symplectic manifold $X$. Then $(W_L,\{[\bd D_i]^\perp\})$ is an LG~seed.

Moreover, for any $j\in\{1,\ldots, k\}$ there is another Lagrangian seed $(L',\{D_i'\})$ in $X$ whose associated LG seed is the mutation, see Definition~\ref{def:mut_LG}, of the former LG seed in the $j$th direction:
$$
(W_{L'},\{[\bd D_i']^\perp\})=\mu_j(W_{L},\{[\bd D_i]^\perp\})
$$ 
for some choice of bases for $H_1(L,\Z)$ and $H_1(L',\Z)$. 
\end{theorem} 
 
A proof of this theorem will be given later. While the proof  uses Corollary~\ref{cor:donaldon_mutation_config},
in all existing examples the Lagrangian seed naturally sits in the complement of an anticanonical divisor, and the proof can be carried through without referring to Corollary~\ref{cor:donaldon_mutation_config}.

\begin{definition}
	\label{def:mut_lag_seed}	
	The new Lagrangian seed $(L',\{D_i'\})$ appearing in Theorem~\ref{th:lag_mut} has an explicit geometric construction explained below, and is called the (geometric) \emph{mutation} of the Lagrangian seed $(L,\{D_i\})$ along the disk $D_j$; the seed 
$(L',\{D_i'\})$ can also be denoted by 
$$
\mu_j(L,\{D_i\}) \text{, \quad or\quad } (\mu_{D_j}L,\{\mu_{D_j}D_i\}).
$$
The definition of the mutated Lagrangian torus $\mu_{D_j}L$ is given below in Definition~\ref{def:mutation}; the construction of the mutated Lagrangian disks $\mu_{D_j}D_i$ with boundary on $\mu_{D_j}L$ is provided by \cite{STW15} and is discussed later.
\end{definition}

In the rest of the section, we will explain Definition~\ref{def:mut_lag_seed} and prove Theorem~\ref{th:lag_mut} following the outline below:
\begin{itemize}
	\item We explain how to mutate a torus $L$ along a single Lagrangian disk $D$ keeping it monotone. 
	The new torus is denoted by $\mu_DL$.
	
	\item We review the result of \cite{STW15} which essentially proves that any additional Lagrangian disks on $L$ can be carried over to the mutated torus, and provide an alternative proof. This will complete Definition~\ref{def:mut_lag_seed}.
	
	\item We prove that the potentials of $L$ and $\mu_D L$ differ by mutation along $[\bd D]$ in the sense of Definition~\ref{def:mut_function}, provided that a compatible Donaldson divisor exists.
\end{itemize}

The statement from the last item 
is the wall-crossing formula 
which has been predicted by Auroux~\cite{Au07}, and in a slightly different context by Kontsevich and Soibelman~\cite{KS06}.
We will derive it from our general wall-crossing formula---Theorem~\ref{th:wall_crossing}---using a \emph{local wall-crossing} result of Seidel~\cite{SeiBook13}.

\subsection{Mutation configurations}

Let $X$ be a symplectic  4-manifold. 

\begin{definition}
A \emph{mutation configuration} is a Lagrangian seed (Definition~\ref{def:Lag_seed}) containing a single Lagrangian disk.
In other words, a mutation configuration $(L,D)\subset X$ consists of
an Lagrangian torus $L\subset X$ and an embedded Lagrangian disk $D\subset X$ with boundary on $L$. We require that $D$ attaches to $L$ cleanly, its boundary is non-contractible in $L$, and $D$ does not intersect $L$ away from its boundary. 
\end{definition}

We shall use the notation $Op$ for a small open neighbourhood of a set.

\begin{lemma}[Weinstein neighbourhood theorem for mutation configurations]
\label{lem:nbhood_for_T_D}
Let $(L,D)\subset X$ and $(L_0,D_0)\subset X_0$ be two mutation configurations in symplectic 4-manifolds $X,X_0$. Then there exist neighbourhoods $Op(L\cup D)$ and $Op(L_0\cup D_0)$ such there is a symplectomorphism between them taking $L,D$ resp.~to $L_0,D_0$.
\end{lemma}

\begin{proof}
The homology classes  $[\bd D]\in H_1(L,\Z)$ and $[\bd D_0]\in H_1(L_0,\Z)$ are primitive because the disk boundaries are embedded.
So there exists a diffeomorphism $L\to L_0$ taking $\bd D\to \bd D_0$, and it has a lift to a symplectic bundle map $TX|_L\to TX_0|_{L_0}$ taking the subbundle $TL$ to $TL_0$, and also taking $TD|_{\bd D}$ to $TD_0|_{\bd D_0}$. As in the proof of the Weinstein neighbourhood theorem, such a diffeomorphism extends to a symplectomorphism $\psi$ between neighbourhoods $Op(L)\to Op(L_0)$ taking $L$ to $L_0$ and $D\cap Op(L)$ to a Lagrangian tangent to $D_0\cap Op(L_0)$ along the curve $\bd D_0$. After applying a Hamiltonian isotopy, we may assume $\psi$ sends $D\cap Op(L)$ precisely to $D_0\cap Op(L_0)$; compare~\cite[Lemma~3.3]{CEL10}. We may extend $\psi$ smoothly to a diffeomorphism $\Psi\co Op(L\cup D)\to Op(L_0\cup D_0)$ which takes  $D$ to $D_0$ and whose differential restricts to a symplectic bundle map $d\Psi|_D\co TX|_D\to TX_0|_{D_0}$. By the relative Moser theorem, see e.g.~\cite[Theorem~7.4]{CdS01}, there is a homotopy between $\Psi$ and a symplectomorphism $Op(L\cup D)\to Op(L_0\cup D_0)$, and this homotopy is constant on $Op(L)\cup D$; in particular the resulting symplectomorphism sends $L,D$ resp.~to $L_0,D_0$.
\end{proof}

\subsection{The model neighbourhood}
\label{subsec:model_nbhood}
We will now study the symplectic geometry of the neighbourhoods of $L\cup D$ appearing in the last lemma.
We will see that these neighbourhoods have a natural structure of the Liouville domain whose completion coincides with the Liouville completion of the following model Liouville domain: 
$$M=\C^2\setminus\{xy=1\},$$
with the symplectic form being the restriction of the standard form on $\C^2$. 

Recall that $M$ contains
two exact tori which are distinct up to compactly supported Hamiltonian isotopy, which
are  called the Clifford and the Chekanov torus. Let us recall their construction following 
 \cite{Au07,Au09}. Consider the projection $$\pi\co M\to \C\setminus \{1\},\quad \pi(x,y)=xy,\quad x,y\in \C.$$
The fibres of $\pi$ are affine quadrics, smooth except for $\pi^{-1}(0)$. Take any simple closed curve 
$$\gamma\subset \C\setminus\{0,1\}$$
and a parameter $m\in \R$, and define Lagrangian tori
$$
T_{\gamma,m}=\{x,y\in \C^2:\  \pi(x,y)\in\gamma,\  |x|-|y|=m\}\subset M.
$$
The following lemma is well known, but we prove it for completeness.

\begin{lemma}
	\label{lem:exactness_in_M_restrictns}
	Let $\theta$ be a 1-form on $M$, $d\theta=\omega$. There is a constant $A\in \R$ (depending on $\theta$) such that a torus $T_{\gamma,m}$ is exact with respect to $\theta$ only if:
	\begin{itemize}
		\item $\gamma$ encloses $1\in\C$;
		\item $m=0$;
		\item the disk in $\C$ bounded by $\gamma$ has area $A$. 
	\end{itemize}
\end{lemma}

\begin{proof}
	Suppose $\gamma$ does not enclose the point $1\in \C$. Then the map
	\begin{equation}
	\label{eq:incl_T_M}
	H_1(T_{\gamma,m},\Z)\to H_1(M,\Z)\cong \Z
	\end{equation}
	vanishes and there are elements of
	$H_2(M,T_{\gamma,m},\Z)$ whose $\omega$-areas are $m$ and $A_\gamma$, where $A_\gamma$ is the area bounded by $\gamma$. Exactness implies $m=0$ and $A_\gamma=0$, but the latter is impossible.

	Suppose $\gamma$ encloses the point $1\in \C$, then the map (\ref{eq:incl_T_M}) 
	has 1-dimensional kernel. The area of an element of $$H_2(M,T_{\gamma,m},\Z)$$ bounding a generator of the kernel of (\ref{eq:incl_T_M}) is $\pm m$, so we must have $m=0$ for an exact torus.
	
	Now consider two tori $T_{\gamma,0}$ and $T_{\gamma',0}$, where $\gamma,\gamma'$ enclose $1\in \C$.
	Pick elements $a\in H_1(T_{\gamma,0})$,  $a'\in H_1(T_{\gamma;, 0})$ whose images under the inclusion map (\ref{eq:incl_T_M}) (and its analogue for $T_{\gamma',0}$) is the same generator of $H_1(M,\Z)$. Then there is an element of
	$$H_2(M,T_{\gamma,0}\cup T_{\gamma',0},\Z)$$
	with boundary $a\sqcup (-a')$ whose area is $\pm(A_\gamma-A_{\gamma'})$. This must vanish for exact tori, so the areas $A_\gamma$ must be the same for all of them.
\end{proof}

The converse of Lemma~\ref{lem:exactness_in_M_restrictns} is also true, moreover the primitive 1-form can always be chosen so as to be Liouville.

\begin{lemma}
	\label{lem:exactness_in_M}
	For each $A\in \R_+$, there is a Liouville 1-form $\theta$ on $M$ such that all Lagrangian tori satisfying the conditions of Lemma~\ref{lem:exactness_in_M_restrictns} are exact.
\end{lemma}

\begin{proof}
We can view $M$ as the result of a Weinstein 2-handle attachment to $T^*T^2$, where we take some torus satisfying the conditions of Lemma~\ref{lem:exactness_in_M_restrictns} to be the zero-section of $T^*T^2$. The standard Weinstein structure on $T^*T^2$ extends to one on $M$, keeping the chosen torus exact. All other tori satisfying the conditions of Lemma~\ref{lem:exactness_in_M_restrictns} are exact by for homological reasons, see the proof of Lemma~\ref{lem:exactness_in_M_restrictns}.
\end{proof}

\begin{lemma}
\label{lem:isotopy_in_M}
Two Lagrangian tori $T_{\gamma,0}$, $T_{\gamma',0}$ can be mapped to each other by a compactly supported Hamiltonian isotopy inside $M$ if and only if $\gamma$ and $\gamma'$ enclose disks of the same area inside $\C$, and $\gamma$ and $\gamma'$ are smoothly isotopic in $\C\setminus\{0,1\}$.
\end{lemma}

\begin{proof}
The last two conditions imply that $\gamma,\gamma'$ are Hamiltonian isotopic inside $\C\setminus \{0,1\}$; this isotopy lifts to a Hamiltonian  isotopy between the tori if $m=m'$.
\end{proof}	

Now fix, once and for all, a Liouville 1-form $\theta_M$ and a corresponding area value $A\in \R_+$ as in
Lemmas~\ref{lem:exactness_in_M_restrictns} and~\ref{lem:exactness_in_M}; the value of $A$ will not matter.
It follows from Lemma~\ref{lem:isotopy_in_M} that there are precisely two classes of exact tori among the $T_{\gamma,0}$, up to compactly supported Hamiltonian isotopy:

\begin{itemize}
\item Clifford-type tori: $\gamma\subset \C$ encloses a disk of area $A$ which contains the points $0$ and $1$;
\item Chekanov-type tori: $\gamma\subset \C$ encloses a disk of area $A$ which contains the point $1$ but does not contain the point $0$. See Figure~\ref{fig:Cliff_Chek}~(left).
\end{itemize}

\begin{figure}
\includegraphics{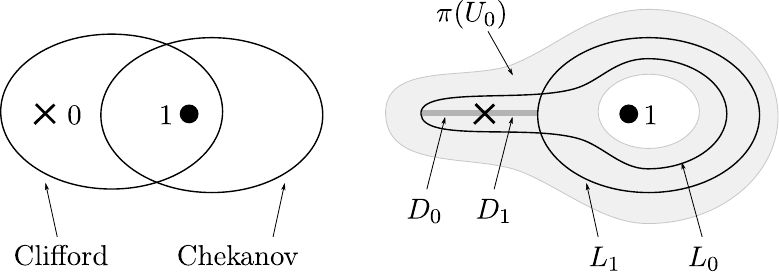}
\caption{Left: curves defining a Clifford-type and a Chekanov-type torus. Right: same tori after a Hamiltonian isotopy, adjusted in a such a way that $L_1$ lies in a small neighbourhood of $L_0\cup D_0$ and vice versa. The projections of $D_0,D_1$ are also shown.}
\label{fig:Cliff_Chek}
\end{figure}

Next, we claim that Clifford and Chekanov type tori bound Lagrangian disks, and therefore form mutation configurations. Consider the torus $T_{\gamma,0}$ of either class, then the desired Lagrangian disk is given by
$$
\{x,y\in \C:\ \pi(x,y)\in \delta, \ |x|=|y|\},
$$
where $\delta$ is any simple path connecting a point of $\gamma$ with the origin, avoiding $\gamma$ and $\{0,1\}$ in its interior; see Figure~\ref{fig:Cliff_Chek} (right). 

Finally, note that a Chekanov-type torus, together with a Lagrangian disk it bounds, can be placed in an arbitrarily small neighbourhood of the union of any Clifford-type torus and a Lagrangian disk with boundary on it, as shown in Figure~\ref{fig:Cliff_Chek} (right).
The lemma below summarises the  properties we have discussed.

\begin{lemma}
\label{lem:config_W}
Fix an exact Clifford-type torus $L_0$ and a Lagrangian disk $D_0$ so that $(L_0,D_0)\subset M$ so that $(L_0,D_0)$ is a mutation configuration. Then:
\begin{enumerate}
\item[(a)]
any neighbourhood of $L_0\cup D_0$ contains another mutation configuration $(L_1, D_1)$, where $L_1$ is an exact Chekanov-type torus;
\item[(b)]
there is an arbitrarily small neighbourhood $U_0$ of $L_0\cup D_0$ such that $(\theta_M)|_{U_0}$ is Liouville, and such that the completion of $U_0$ is isomorphic to the completion of $M$.\qed

\end{enumerate}
\end{lemma}

\subsection{Seidel's local wall-crossing}
Finally, we recall an important computation due to Seidel~\cite[Proposition~11.8]{SeiBook13} which proves the \emph{local} wall-crossing formula for the pair $L_0,L_1$. (We are using the notation from the previous subsection.)

\begin{lemma}[Local wall-crossing]
	\label{lem:seidel_local}
Let $L_0,L_1\subset M$ be a Clifford-type resp.~a Chekanov-type torus in $M$, as above; then there exist bases of
$H_1(L_i,\Z)\cong \Z^2$ with the following property. Consider arbitrary local systems 
$$
\rho_{i}\in (\C^\x)^2,\quad i=0,1
$$
and denote $\bL_i=(L_i,\rho_{i})$. Then 
$$
HF^*_M(\bL_0,\bL_1)\neq 0
$$
if and only if
$$
\rho_{1}=\mu_{[\bd D_0]^\perp}(\rho_{0})\in (\C^\x)^2
$$
where $D_0$ is the Lagrangian disk described in the previous subsection, and $\mu_{[\bd D_0]}$ is defined in (\ref{eq:wall_cross_map}).\qed
\end{lemma}

The lemma is proven  by computing the relevant holomorphic strips explicitly. Recall that under the Lefschetz fibration $\pi\co M\to \C\setminus\{1\}$, the $L_i$ project to circles, and there are two obvious holomorphic strips between those circles. The strip containing $0$ lifts to {\it two} different holomorphic strips in $M$ between the $L_0,L_1$, while the other strip lifts to a single strip in $M$.  Lemma~\ref{lem:seidel_local} follows by looking at the boundary homology classes of these three strips.

\begin{remark}
	\label{rem:std_Lag_isotopy}
Let us describe the choice of bases which makes Lemma~\ref{lem:seidel_local} hold. Consider the smooth isotopy from $L_0$ to $L_1$ obtained by lifting both $L_i$ $\pi$-fibrewise from the set $|x|-|y|=0$ to $|x|-|y|=\epsilon$, and subsequently isotoping them to each other by an isotopy lifted from one between the defining curves in $\C$. Any two bases for $H_1(L_i)$ related by this isotopy will be suitable for Lemma~\ref{lem:seidel_local}. In particular, one can pick a basis where
$(1,0)$ is the fibre class; note that this class equals $\pm[\bd D_0]$ depending on the choice of orientation. If we choose $[\bd D_0]=(-1,0)$, then $\mu_{[\bd D_0]^\perp}$  is precisely given by the formula appearing in \cite[Proposition~11.8]{SeiBook13}.
\end{remark}

\subsection{Mutation of Lagrangian tori}
We now combine the Weinstein neighbourhood theorem for mutation configurations with our knowledge about the model space $M$.

\begin{lemma}
\label{lem:mutation}
Let $(L,D)\subset X$ be a mutation configuration in a symplectic manifold. Then there exists a neighbourhood $U\subset X$ of $L\cup D$, and another  mutation configuration $(L',D')\subset U\subset X$ with following the property. There is a symplectomorphism $\phi\co U\to U_0$, where $U_0$ is a sufficiently small neighbourhood from Lemma~\ref{lem:config_W}, such that $\phi$ takes $L,D,L', D'$ resp.~to $L_0,D_0,L_1, D_1$. Moreover,
\begin{itemize}
\item[(a)] If $X$ and $L$ are monotone or exact, then $L'$ is also monotone or exact, respectively.
\item[(b)] If $X$ is Liouville and $L$ is exact, then $U\subset X$ can be arranged to be a Liouville embedding. 
\end{itemize}
\end{lemma}

\begin{proof}
By Lemma~\ref{lem:nbhood_for_T_D}, we can find a symplectomorphism $\psi\co Op(L\cup D)\to Op(L_0\cup D_0)\subset M$, where $L_0\subset M$ is an exact torus of Clifford class. Then, by Lemma~\ref{lem:config_W}, we can find a smaller neighbourhood $U_0\subset Op(L_0\cup D_0)$ which is a Liouville subdomain of $M$ and contains another mutation configuration $(L_1,\tilde D_1)$ where $L_1$ is exact and of Chekanov type. We define $U=\psi^{-1}(U_0)$, $K=\psi^{-1}(L_1)$ and $\tilde D=\psi^{-1}(\tilde D_1)$. Then $\psi$ is the desired symplectomorphism, up to checking properties (a) and (b). 

We note that (a) is equivalent to the fact that the area maps
$$\omega\co H_2(X,L,\Z)\to\R,\quad H_2(X,K,\Z)\to \R$$
are equal. This is proved by restricting to $U$ and using an argument similar to the one in the proof of Lemma~\ref{lem:exactness_in_M_restrictns}.	

To prove (b), recall that $U_0\subset M$ has already been chosen to be a Liouville domain on its own. Checking that $U\subset X$ is a Liouville subdomain means checking that
$$[\theta_X|_U-\psi^*(\theta_M|_{U_0})]=0\in H^1(U).$$
This follows from the fact that $L$ is exact with respect to both 1-forms in the above expression, and that $H^1(U)\to H^1(L)$ is surjective.
\end{proof}

\begin{definition}
\label{def:mutation}
Let $(L,D)\subset X$ be a mutation configuration.
We say that the mutation configuration $(L',D')\subset X$ from Lemma~\ref{lem:mutation} is obtained from $(L,D)$ by {\it mutation  along} $D$, or {\it in the direction of} $D$. It is defined uniquely up to Hamiltonian isotopy.
If $L$ is a monotone (or exact) Lagrangian torus, then so is $L'$; we call $L'$ the \emph{mutated} torus.
We denote 
$$L'=\mu_{D}L,\quad D'=\mu_D D.$$
\end{definition}

Mutating along a single Lagrangian disk cannot give more than two different tori, as expressed by the following lemma (compare with the algebraic Remark~\ref{rem:one_dir}).

\begin{lemma}[Reverse mutation]
\label{lem:reverse_mut}
Two consecutive mutations of $(L,D)$ along $D$, and then of $(L',D')$ along $D'$ give a configuration which is Hamiltonian isotopic to the original $(L,D)$.
\end{lemma}
\begin{proof}
This follows from the fact that the roles of the Clifford and the Chekanov classes in Lemma~\ref{lem:mutation} may be swapped, by a non-compactly symplectomorphism of $M$.
\end{proof}

\begin{remark}
So far, we have not used the Liouville properties stated in Lemma~\ref{lem:mutation}, but they will be soon be used  in the proof of the wall-crossing formula.
\end{remark}

\subsection{Wall-crossing formula}
We are ready to prove the core statement of Theorem~\ref{th:lag_mut}, the wall-crossing formula.

\begin{theorem}[Wall-crossing formula]
	\label{th:wall_cross_4d}
	Let $X$ be a monotone del~Pezzo surface, and $(L,D)\subset X$ be a mutation configuration admitting a compatible Donaldson divisor. Let $L'=\mu_D L$ be the mutated torus. For any basis of $H_1(L,\Z)$, there exists a basis of $H_1(L',\Z)$ for which
	\begin{equation}
	\label{eq:wall_cross_4dim}
	W_{L',X}=\mu_{[\bd D]^\perp}W_{L,X}.
	\end{equation}
\end{theorem}

\begin{proof}
Let $U$ be a neighbourhood of $L\cup D$ as in Lemma~\ref{lem:mutation},
and $\Sigma$ a Donaldson divisor provided by Corollary~\ref{cor:donaldon_mutation_config}.
Because $L\subset U$ is exact inside both $U$ and $X\setminus\Sigma$, and $H^1(U)\to H^1(L)$ is surjective, it follows that $U\subset X\setminus \Sigma$ is a Liouville embedding (compare the proof of Lemma~\ref{lem:mutation}).
Recall that $L'\subset U$ is also exact, by construction.
In view of Proposition~\ref{prop:U-to-d-graded}, we can apply Theorem~\ref{th:wall_crossing} to
$$
L,L'\subset U\subset X\setminus \Sigma.
$$
Together with 
 Lemma~\ref{lem:seidel_local} this yields the result.
\end{proof}

\subsection{Almost toric fibrations and Lagrangian seeds}
Almost toric fibrations on symplectic 4-manifolds were introduced by Symington \cite{Sym03}, and have recently been used by Vianna~\cite{Vi13,Vi14,Vi16} to construct infinitely many monotone Lagrangian tori in del~Pezzo surfaces.
We will assume that the reader is familiar with this notion and the related terminology. 

First, we wish to explain how to construct Lagrangian disks with boundary on Lagrangian tori with the help of toric and almost toric fibrations.

\begin{lemma}[Constructing Lagrangian seeds]
\label{lem:constr_lag_disks}
Let $\pi\co X\to B$ be an almost toric fibration over a base $B$, and $L=\pi^{-1}(b)$ be a monotone torus, for some point $b\in B$. 
Suppose there is a line segment in $B$ starting at $b$, going in the direction of a primitive integral vector $v\in T_bB$, and with endpoint on either:

\begin{enumerate}
	\item[(i)] a vertex of $\bd B$ (i.e.~a point whose $\pi$-preimage is a single point), or
	\item[(ii)] a nodal point of $B$, whose monodromy line also goes in the direction of $v$.
\end{enumerate}
Then there is a Lagrangian disk $D\subset X$ which projects onto that line segment, such that $(L,D)$ is a mutation configuration, and
$$
[\bd D]^\perp=v
$$
under the canonical identification between $H_1(L,\Z)$ and the integral lattice of $T_bB$. Repeating this construction for different segments results in a Lagrangian seed.

Conversely, if $(L,D)$ is a mutation configuration, then a neighbourhood of $L\cup D$ has an almost toric fibration over an open disk with one nodal fibre for which $L$ is the fibre over a point which lies on the monodromy line of the nodal point, and $D$ projects to the line segment as shown in Figure~\ref{fig:cp2}.
\end{lemma}

\begin{proof}
Denote the line segment by $l$. To construct the disk, pick a curve in $L=\pi^{-1}(b)$ in the homology class $v^\perp$. One can parallel transport the curve over $l$ in such a way that it traces a Lagrangian cylinder projecting to $l$. As $l$ approaches the node or the corner we can arrange that the cycle being transported collapses to a single point, and the total Lagrangian submanifold that was swept becomes a Lagrangian disk. In case (i) the collapse is automatic because the preimage of the vertex is a single point, and in case (ii) it is guaranteed by the condition on the monodromy line. In both cases,  one can verify that the collapse results in a smooth disk using the local form of the singularities. 

The last statement follows from the Weinstein neighbourhood theorem for $L\cup D$ (see Lemma~\ref{lem:nbhood_for_T_D}) and the well-known explicit construction of the desired almost toric fibration on $M$, see section~\ref{subsec:model_nbhood} and \cite{Au07,Au09}. 
\end{proof}	

\begin{remark}
Our convention for identifying $T_bB$ with $H_1(L,\Z)$ is such that if there is a {\it holomorphic} disk on $L$ which in the base diagram leaves the point $b$ in the direction $v$, then its boundary homology class equals $v$.
\end{remark}

\begin{figure}[h]
	\includegraphics[]{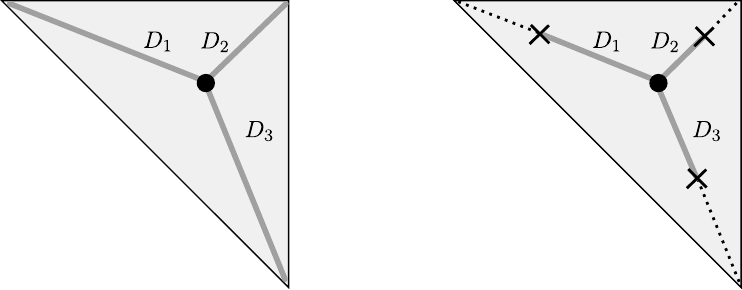}
\caption{The toric and an almost toric fibration on $\C P^2$ giving rise to the same Lagrangian seed. The line segments for Lemma~\ref{lem:constr_lag_disks} are shown in dark gray. The nodes are shown with a cross.}	
\label{fig:cp2}
\end{figure}

\begin{example}
	\label{ex:cp2}
Consider the standard toric fibration on $\C P^2$ over the triangle shown in Figure~\ref{fig:cp2}; the monotone torus projects to the barycentre of the trangle. We have chosen the triangle so that the superpotential equals $x+y+1/xy$ in the associated basis. The primitive integral directions pointing from the barycentre towards the three vertices are:
\begin{equation}
\label{eq:bdies_seed_cp2}
(1,1),\ (-2,1),\ (1,-2).
\end{equation}
Consequently, for the standard Clifford torus $L\subset \C P^2$ there is a Lagrangian seed $(L,\{D_1,D_2,D_3\})$ such that the classes $[\bd D_i]^\perp$ equal (\ref{eq:bdies_seed_cp2}). This matches with Table~\ref{tab:dp}.

Alternatively, there is an almost toric fibration on $\C P^2$ obtained by a procedure called {\it smoothing the corners} of the base. This fibration has three nodes whose monodromy lines have  directions (\ref{eq:bdies_seed_cp2}). The monodromy lines intersect at the barycentre, so one may apply Lemma~\ref{lem:constr_lag_disks}(ii) to get the same Lagrangian seed.

A simpler example is  the Clifford torus in $\R^4$ with potential $x+y$; it bounds a Lagrangian disk with boundary homology class $(1,1)$.
\end{example}

\begin{proposition}
	\label{prop:del_pezzo_seed}
For each del~Pezzo surface $X$, there exists a monotone Lagrangian torus $L\subset X$ included in a Lagrangian seed $(L,\{D_i\})$, such that the classes $[\bd D_i]^\perp\in \Z^2$ are the vectors shown in Table~\ref{tab:dp}, for some basis of $H_1(L,\Z)$.
\end{proposition}

\begin{proof}
For toric del~Pezzos, the vectors from Table~\ref{tab:dp} are precisely the directions pointing from the barycentre of the polytope to its vertices, which implies Proposition~\ref{prop:del_pezzo_seed} by Lemma~\ref{lem:constr_lag_disks}(i).

For a non-toric del~Pezzo surface $X$, Vianna~\cite{Vi16} constructed almost toric fibrations on $X$ that have as many nodes as there are vectors in the corresponding entry from Table~\ref{tab:dp}, with the property that all monodromy lines of these nodes intersect at the point corresponding to the monotone fibre. Using Lemma~\ref{lem:constr_lag_disks}(ii) one gets a Lagrangian seed with the desired number of Lagrangian disks. 
It is an easy exercise to show, starting with one of the fibrations Vianna provides, to bring the boundary homology classes of the Lagrangian disks to the ones listed in Table~\ref{tab:dp}, either by mutating the almost toric fibrations in the sense of~\cite{Vi16}, or by mutating the Lagrangian seed using Theorem~\ref{th:stw} below. (These two ways are equivalent.)

Finally, the Lagrangian seeds constructed this way belong to the complement of an anticanonical divisor (the preimage of the boundary of the almost toric fibration) where the torus becomes exact. This divisor is automatically compatible.
\end{proof}	

An important concept in the theory of almost toric fibrations is the notion of {\it nodal slide}; it was the main geometric tool in Vianna's construction of infinitely many Lagrangian tori. The next lemma explains that nodal slide is essentially equivalent to the mutation of Lagrangian tori from Definition~\ref{def:mutation}. 
\begin{figure}[h]
	\includegraphics[]{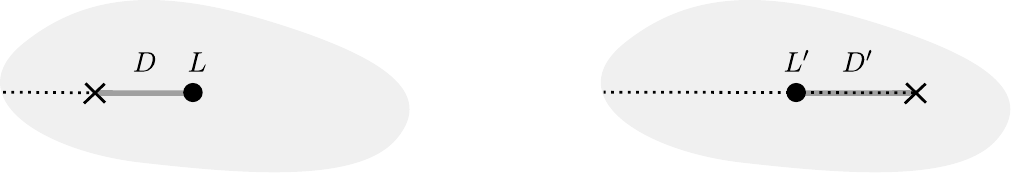}
	\caption{Mutation via nodal slide. The Lagrangian disks project to the dark gray segments. The monodromy line is dotted; the node is shown with a cross.}
	\label{fig:slide}
\end{figure}

\begin{lemma}[Mutation via nodal slide]
\label{prop:mut_and_nodal_slide}
In the setup of Lemma~\ref{lem:constr_lag_disks}(ii), let us slide the node considered therein past the point $b$. Denote by $\pi'\co X\to B'$ the resulting new almost toric fibration, and denote $L'=(\pi')^{-1}(b)$;  see Figure~\ref{fig:slide}.
The Lagrangian torus
$L'$ is also monotone and satisfies:
$$
L'=\mu_D L
$$ 
up to Hamiltonian isotopy. 
\end{lemma}

\begin{proof}
Let $U\subset X$ be the $\pi$-preimage of a neighbourhood of the segment $l$. Then $U$ is a Weinstein neighbourhood of $L\cup D$, see Lemma~\ref{lem:nbhood_for_T_D}. By definition, the nodal slide is modelled on changing the almost toric fibration on the model space $M$ (see subsection~\ref{subsec:model_nbhood}) so that the given fibre over $b$ switches from being a Clifford-type torus to being a Chekanov-type torus.
\end{proof}

\subsection{Mutating Lagrangian disks}
Below is a slight refinement of a theorem of Shende, Treumann and Williams~\cite{STW15}.

\begin{theorem}[Shende-Treumann-Williams]
	\label{th:stw}
Suppose $(L,\{D_i\}_{i=1}^k)$ is a Lagrangian seed in the sense of Definition~\ref{def:Lag_seed}, and $L'=\mu_{D_j} L$ is the mutated torus in the sense of  Definition~\ref{def:mutation}. Then $L'$ also bounds $k$ Lagrangian disks denoted by $$D_i'=\mu_{D_j}D_i,\quad i=1,\ldots,k$$ such that
$
(L',\{D_i'\}_{i=1}^k) 
$
constitute a Lagrangian seed. Moreover, for any basis of $H_1(L,\Z)$ there exists a basis of $H_1(L',\Z)$ such that if we denote:
$$
v_i=[\bd D_i]^\perp\in H_1(L,\Z)= \Z^2,
\quad 
v_i'=[\bd D_i']^\perp\in H_1(L',\Z)= \Z^2,
$$
then:
$$
\begin{cases}
v_j'=-v_j,\\
v_i'=\mu_{v_j}v_i,\ i\neq j
\end{cases}
$$
where the latter mutation is understood as in Definition~\ref{def:mut_trop}.
The disk $D_j'$ coincides with the disk $\mu_{D_j}D_j$ from Definition~\ref{def:mutation}, and the above basis of $H_1(L',\Z)$ agrees with the one required for Lemma~\ref{lem:seidel_local} and Theorem~\ref{th:wall_cross_4d}. Moreover, if the former Lagrangian seed admits a compatible Donaldson divisor, so does the latter.
\qed
\end{theorem}

Observe that the vectors $v_i'$ are in agreement with the mutation of LG seeds, see Definition~\ref{def:mut_lag_seed}. 
As noted in \cite{STW15}, the above theorem fails for Lagrangian surfaces of higher genus (unlike the previous discussion in this section, which can be generalised to higher genus Lagrangians). 

There is one detail of Theorem~\ref{th:stw} which is not mentioned in \cite{STW15}: the fact that the choices of bases that make Theorem~\ref{th:stw} hold are the same as the ones making Lemma~\ref{lem:seidel_local} hold, if we identify $(L, D_j)$, $(L',D_j')$ with their local models $(L_0,D_0)$ and $(L_1,D_1)\subset M$ as in Lemma~\ref{lem:mutation} and Definition~\ref{def:mutation}. (Because the choice of basis in Theorem~\ref{th:wall_cross_4d} is derived from Lemma~\ref{lem:mutation}, consistency with it also follows.)
We provide an alternative proof of Theorem~\ref{th:stw} where this missing detail becomes  transparent.

\begin{figure}[h]
\includegraphics[]{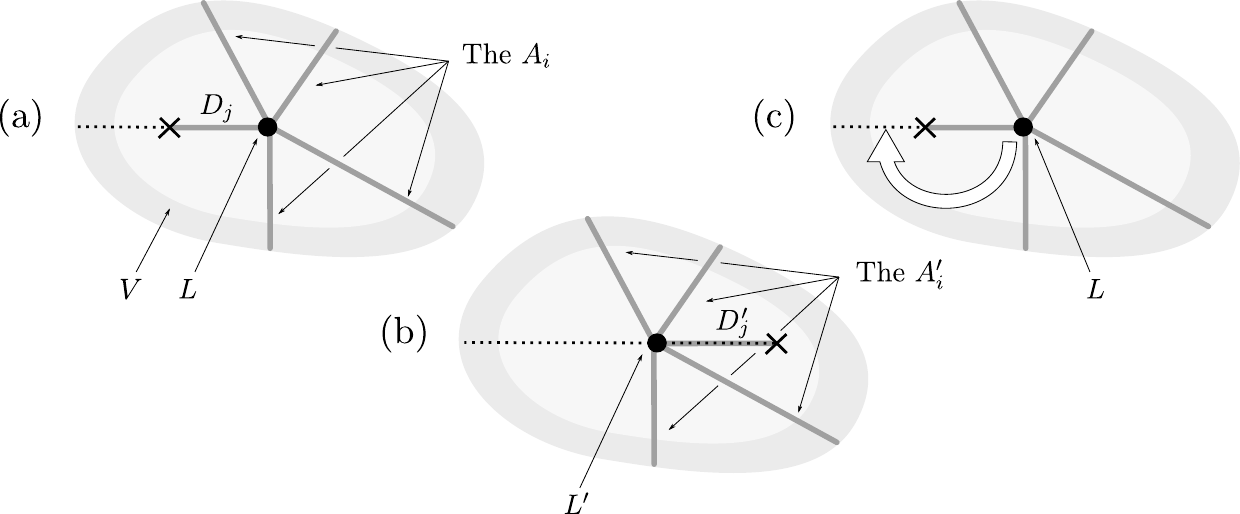}
\caption{A nodal slide perspective on mutating a Lagrangian seed.}
\label{fig:slide_manydisks}
\end{figure}

\begin{proof}[Proof of Theorem~\ref{th:stw}]
Let $U$ be a neighbourhood of $L\cup D_j$ where $D_j$ is the disk chosen for mutation. We mentioned that there is an almost toric fibration $\pi\co U\to B$ over a disk $B$ with one nodal fibre, see Figure~\ref{fig:slide} and the proof of Proposition~\ref{prop:mut_and_nodal_slide}. In this model, $L=\pi^{-1}(b)$ and $D_j$ projects to the line segment $l$. 

We now fit the other disks into this picture. One can assume that the other disks $D_i$ intersect $U$ only near their boundary (i.e.~near $L$), therefore $D_i\cap U=A_i$ is a Lagrangian annulus, $i\neq j$. One can arrange that in the almost toric fibration on $U$, the $A_i$ project to  straight line segments as shown in Figure~\ref{fig:slide_manydisks}(a), whose slopes are equal to the boundary homology classes $v_i=[\bd D_i]^\perp$ under the identification $H_1(L,\Z)=T_bB$. The annuli themselves are obtained by parallel transport of an appropriate circle $\gamma_i$ inside the toric fibre over the point where the line segment meets $\bd B$, as in the proof of Lemma~\ref{lem:constr_lag_disks}.

Pick a subset $V\subset U$ which is the preimage of the annulus shown by darker shade in Figure~\ref{fig:slide_manydisks}(a); then $V$ is a collar of $\bd U$. 
Now perform the nodal slide to get the new fibration $\pi'\co U\to B'$, see Figure~\ref{fig:slide_manydisks}(b).
The nodal slide can be performed so that the two fibrations match on $V$ identically: $\pi|_V\equiv \pi'|_V$. Define the Lagrangian annulus $A_i'$ to be transport of the circle $\gamma_i$ as above along the same line segment, but this time using the fibration $\pi'$ instead of $\pi$. We define
$$
D_i'=(D_i\setminus A_i)\cup A_i'.
$$
The sets $D_i$ are again smooth Lagrangian disks because the $A_i$ and $A_i'$  coincide in $V$, since $\pi$ and $\pi'$ match on $V$. 
The boundaries of new disks $D_i'$ are on $L'$; these are the desired disks.

It remains to compute the homology classes $v_i'=[\bd D_i']^\perp$.
First we must specify an identification between $H_1(L,\Z)$ with $H_1(L',\Z)$; it  comes from a specific isotopy between $L$ and $L'$ inside $U$ which we now describe. First, before performing the nodal slide, we isotop $L$ following the path shown by the bold white arrow in Figure~\ref{fig:slide_manydisks}(c). The endpoint of the path is on the same horizontal line, and the path itself goes around the node in the lower half-plane. Clearly, this gives a smooth isotopy from $L$ to another torus. We then compose this isotopy with an isotopy which moves the point (representing the torus) within the horizontal monodromy line to the right while {\it simultaneously} sliding the node to the right as well, so that we eventually end up with $L'$ without crossing the node. One can check that this isotopy is the same one as described in Remark~\ref{rem:std_Lag_isotopy} using the Lefschetz fibration setup.

The annuli $A_i$ represented by the rays in the {\it lower} half space in Figure~\ref{fig:slide_manydisks}(a) can be smoothly deformed through the process of this isotopy, just by deforming the rays making their common endpoint follow the move that we described above. (We allow the rays to curve, but keep their part inside $V$ fixed.) Given the choice of bases, it implies that
$v_i'=v_i$.  In the same basis we have that $v_j=(-1,0)$, which means that the max-term in the tropical mutation formula (\ref{eq:mut_trop}) becomes
$$
\max(0,(v_i)_2)
$$
where $(v_i)_2$ is the second co-ordinate of the vector $v_i$. 
Recall that we were working with rays in the lower half-space, meaning $(v_i)_2<0$, so (\ref{eq:mut_trop}) translates to $\mu_{v_j}v_i=v_i$. This agrees with our computation  $v_i'=v_i$.

Now take an annulus $A_i$ represented by a ray in the {\it upper} half space, that is, $(v_i)_2>0$. If we deform the ray making their common endpoint  follow the move described above, it will intersect the node once. This means that $v_i'$ differs from $v_i$ by the monodromy around the node, which is the Dehn twist around the vanishing cycle $v_j$. In agreement with this, the tropical mutation (\ref{eq:mut_trop}) has the effect of the same Dehn twist provided that 
$\max(0,(v_i)_2)=(v_i)_2>0$.
\end{proof}

\subsection{Proof of the  mutation theorem}
\label{subsec:proof_mut}
We put together  the previous discussion to prove Theorem~\ref{th:lag_mut}.

To exhibit the desired mutated Lagrangian seed 
$(L',\{D_i'\})$, we take the torus $L'=\mu_{D_j}L$ from Definition~\ref{def:mutation}, and the Lagrangian disks $\{D_i'\}=\{\mu_{D_j}D_i\}$ from Theorem~\ref{th:stw}.
The potential $W_{L'}$ is computed by Theorem~\ref{th:wall_cross_4d}. The Donaldson divisor compatible with $(L,\{D_i\})$ is also compatible with $(L',\{D_i'\})$.
 
 It remains to note that $(W_L,\{[\bd D_i]^\perp\})$ is indeed an LG seed. For this, we need to check that $\mu_{[\bd D_j]^\perp} (W_L)$ are Laurent polynomials for all $j$. This follows from the fact that they are the LG~potentials of the mutated tori, which are all monotone.
\qed

\begin{example}
Consider  the Lagrangian seed $(L,\{D\})$ in $\R^4$ where $L$ is the Clifford torus and $D$ is the Lagrangian disk with boundary homology class $(1,1)$, see Example~\ref{ex:cp2}. The corresponding LG~seed is
$$
(x+y,\ \{(1,1)\}).
$$
The mutated torus $\mu_D L$ is the well~known monotone Chekanov torus in $\R^4$, which bounds a single family of holomorphic Maslov index~2 disks. And indeed, we compute:
$$
\mu_{(1,1)}\co x\mapsto x(1+xy^{-1})^{-1} ,\ y\mapsto y(1+xy^{-1})^{-1},
$$
so
$$
\mu_{(1,1)}(x+y)=(x+y)(1+xy^{-1})^{-1}=y
$$
which is a monomial. The LG~seed associated with the Chekanov torus is therefore
$$
(y,\ \{(-1,-1)\}).
$$
Its mutation gives back the previous LG~seed for the Clifford torus up to the $SL(2,\Z)$-action, see Remark~\ref{rem:one_dir}.
\end{example}

\begin{example}
	Consider  the Lagrangian seed $(L,\{D_1,D_2,D_3\})$ in $\C P^2$ where $L$ is the Clifford torus, see Example~\ref{ex:cp2} or Proposition~\ref{prop:del_pezzo_seed}. The corresponding LG~seed is
	$$
	(x+y+1/xy,\ \{(1,1),\, (-2,1),\, (1,-2)\}).
	$$
Mutating it along the vector $(1,1)$, we obtain the following LG~seed:
$$
\left(y+\frac{(x+y)^2}{xy^3},\ \{(-1,-1),\, (-2,1),\, (4,1)\}\right).
$$	
The corresponding monotone torus  in $\C P^2$ is the Chekanov torus, see e.g.~\cite{Au07}.
 The potential appearing in the seed above is known: it matches \cite[Formula~(5.5)]{Au07} under the change of variables $x=uw$, $y=u$. The above three boundary homology classes of Lagrangian disks on the Chekanov torus can be read off from Vianna's almost toric fibration \cite[Figure~1]{Vi13}, see Lemma~\ref{lem:constr_lag_disks}. This LG~seed corresponds to the Markov triple $(1,1,2)$, and further mutations recover the seeds corresponding to all Markov triples.
 \end{example}

\subsection{Lagrangian tori in del Pezzo surfaces}

We are ready to show that the del~Pezzo LG~models from Table~\ref{th:lag_mut}
are all geometrically realised by Lagrangian seeds in the sense of Definition~\ref{def:Lag_seed}. 

\begin{theorem}
\label{th:lag_del_pezzo}
For each del~Pezzo LG seed from Table~\ref{tab:dp}, there exists a Lagrangian seed $(L,\{D_i\})$ in the corresponding del Pezzo surface $X$ which realises that LG seed, in the sense that $(W_L,\{[\bd D_i]^\perp\})$ equals that LG~seed.
\end{theorem}

Given the previous discussion, the remaining statement here is the computation of the Landau-Ginzburg potentials of the tori from Proposition~\ref{prop:del_pezzo_seed}.
For a proof, we need the following proposition.

\begin{proposition}
\label{prop:upper_bound}
Let $\{v_i\}$, $v_i\in \Z^2$, be one of the 10 collections of vectors appearing in Table~\ref{tab:dp}.
Consider the following subring, called the upper bound: 
$$
\mathcal{U}=\left\{W\in\C[x^{\pm 1},y^{\pm 1}]: (W,\{v_i\})\text{ is an LG seed}\, \right\}\subset \C[x^{\pm 1},y^{\pm 1}].
$$
Then 
$$
\mathcal{U}=\C[W_0],
$$
where $W_0$ is the corresponding potential from Table~\ref{tab:dp}. In other words, any $W\in \mathcal{U}$ has the form $W=f(W_0)$, where $f$ is a polynomial in one variable.
\end{proposition}

\begin{remark}
Rephrasing the definition, $\mathcal{U}$ consists of Laurent polynomials $W$ such that $\mu_{v_i}(W)$ (see Definition~\ref{def:mut_function}) is also a Laurent polynomial, for all $i$.
\end{remark}

The upper bound  in the context of LG seeds was introduced by Cruz~Morales and Galkin \cite{CMG13}. Proposition~\ref{prop:upper_bound} appeared as \cite[Conjecture~4.1]{CMG13} but seems to be relatively well-understood by experts; we are grateful to Mark~Gross and Denis~Auroux for pointing this out. We explain it briefly.

\begin{proof}[Sketch of proof of Proposition~\ref{prop:upper_bound}]
Given the vectors $\{v_i\}$, Gross, Hacking and Keel \cite[Definition~2.8, Remark~2.9]{GHK15} explain how to construct the $\mathcal{A}$-cluster variety $X^\vee$ whose ring of regular functions $\C[X^\vee]$ equals the upper bound $\mathcal{U}$. It turns out that when $\{v_i\}$ is one of the ten collections from Table~\ref{tab:dp}, there is a proper map $X^\vee\to \C$, therefore $\C[X^\vee]$ is a polynomial ring in one variable. (In fact, that proper map is an elliptic fibration.) 

Finally, one can verify that $W_0\in \mathcal{U}$, and that $W_0$ does not have the form $f(W)$ for another Laurent polynomial $W$. Therefore $W_0$ is a generator of $\C[X^\vee]$.
\end{proof}	

\begin{remark}
For a better understanding of how the above proof works, it is useful to recall some details about the cluster variety $\check X$.
Suppose we are given an LG seed in which the vectors $v_i$ are pairwise distinct.	
To construct the mirror $X^\vee$ up to a possible complex codimension~2 error (which does not matter for our arguments), one starts by taking one initial copy of $(\C^*)^2$, and one additional copy of $(\C^*)^2$ for each vector $v_i$. Then one glues each additional copy of $(\C^*)^2$ to the original one via the wall-crossing map $\mu_{v_i}$ from (\ref{eq:wall_cross_map}). On a diagram, this looks as follows:
$$
{\small
	\xymatrixcolsep{2em}
	\xymatrixrowsep{2em}
\xymatrix{
	&(\C^*)^2\ar@{-->}^{\mu_{v_1}}[d]&
\\
	&(\C^*)^2&
\\
	(\C^*)^2\ar@{-->}^{\mu_{v_2}}[ur]&&(\C^*)^2\ar@{-->}_{\mu_{v_3}}[ul]
}
}
$$
For the purpose of the argument, one can simply define $X^\vee$ to be the scheme obtained by the above gluing. Because $W_0$ stays a Laurent polynomial under mutations along the $v_i$, it defines a regular function on $X^\vee$ denoted by the same letter:
$$
W_0\co X^\vee\to \C.
$$ 
A special property of the 10~del~Pezzo seeds is that the fibres of this function are compact (in fact, $W_0$ becomes a proper elliptic fibration); this is the map mentioned in the previous proof. The  point is that the fibres of $W_0$ on the original copy of $(\C^*)^2$ are elliptic curves with a number of punctures, and these punctures get filled in after gluing in the additional copies of $(\C^*)^2$. This can be verified by a direct computation.
It follows that $\C[X^\vee]$ is a polynomial ring in one variable. Moreover, it is obvious that any element of the upper bound $\mathcal{U}$ extends to a regular function on $X^\vee$.

When some vectors among the $\{v_i\}$ coincide (which is the case for most del Pezzos in Table~\ref{tab:dp}), the gluing construction of $X^\vee$ has to be modified as explained in \cite{GHK15}. The need for this can be bypassed as after a finite sequence of mutations, each of the 10 del~Pezzo seeds can be brought to a form where all vectors are different. The fact that upper bounds are preserved by mutations is well known, see e.g.~\cite{CMG13}.
\end{remark}

\begin{proof}[Proof of Theorem~\ref{th:lag_del_pezzo}]
Fix a del Pezzo surface $X$. By Proposition~\ref{prop:del_pezzo_seed}, there exists a Lagrangian seed $(L,\{D_i\})$ with $v_i=[\bd D_i]^\perp$ as in Table~\ref{tab:dp}.
It remains to show that the Landau-Ginzburg potential $W_L$ is equal to the corresponding Laurent polynomial $W_0$ from Table~\ref{tab:dp}.
When $X$ is toric, Table~\ref{tab:dp} shows the standard toric potential which has been computed by Cho and~Oh \cite{CO06}. For the cubic surface, it shows the potential computed by Fukaya, Oh, Ohta and Ono \cite{FO311b} and our argument will provide an alternative computation.
Let $X$ denote any non-toric del Pezzo surface.

By Theorem~\ref{th:lag_mut},  $W_L\in \mathcal{U}$ where $\mathcal{U}$ is the upper bound from Proposition~\ref{prop:upper_bound}. Therefore $W_L=f(W_0)$, for some one-variable polynomial $f(u)$ with integral coefficients. We need to show that $f(u)=u$. 

Vianna proved \cite{Vi14,Vi16} that unless $X=\Bl_k\C P^2$ for $k=1,2$, there exist almost toric fibrations on $X$ with {\it triangular} affine base. (Since $\Bl_k\C P^2$ is toric for $k=1,2$, we do not need to consider these cases.) Let $(L',\{D_i'\})$ be the Lagrangian seed obtained, using Lemma~\ref{lem:constr_lag_disks}, from an almost toric fibration on $X$ with a triangular affine base. Because all almost toric fibrations constructed by Vianna differ by a sequence of mutations, the Lagrangian seeds $(L',\{D_i'\})$ and $(L,\{D_i\})$ differ by a sequence of mutations.

Let $W_0'$, $W_{L'}$ be the Laurent polynomials obtained by applying the corresponding sequence of mutations to $W_0$ resp.~$W_L$, so that $W_{L'}=f(W_0')$. Observe that $W_{L'}$ is  the Landau-Ginzburg potential of the monotone torus $L'$.

\begin{remark}
The base being triangular is equivalent to the fact that there are exactly three distinct vectors $v_1,v_2,v_3$ among the $\{[\bd D_i']\}^\perp$ (they may appear with repetitions). For $\Bl_k\C P^2$, $k\in\{6,7,8\}$, the del Pezzo seed from Table~\ref{tab:dp} already comes from a triangular affine base, so for them we can take $(L',\{D_i'\})=(L,\{D_i\})$.
However when $k\in\{4,5\}$, the del~Pezzo seeds from Table~\ref{tab:dp} do not come from a triangular affine base, so one has to mutate them.
\end{remark}

By construction, $L'$ is the monotone fibre of an almost toric fibration on $X$  with triangular base. For such fibres, Vianna proved \cite{Vi14,Vi16} using SFT techniques that the Newton polytope of $W_{L'}$ equals the appropriately scaled affine dual to the triangle spanned by $v_1,v_2,v_3$; we denote this Newton polytope by $\mathcal{P}$. Moreover, he proved that the coefficients by the three monomials of $W_{L'}$ corresponding to the vertices  of  $\mathcal{P}$ are all equal to 1. The Newton polytope $\mathcal{P}$ turns out to be minimal, i.e.~$\lambda \mathcal{P}$ is not an affine translation of another integral triangle, for $0<\lambda<1$.

Assume however that $W_{L'}=f(W_0)$ and $f$ is a polynomial which is not linear. Then
$\mathcal{P}$ is equal to a non-trivial scaling of the integral Newton polytope of $W_0$, which contradicts the minimality. Therefore $f(u)=ku$ for some $k$; but then $k=1$ by the mentioned result about the monomial coefficients.
\end{proof}	

\begin{remark}
We recall that the technique of \cite{Vi14,Vi16} does not extend to a full computation of the potential beyond the three monomial coefficients, so an argument using the wall-crossing formula is necessary for a full computation. 
Recall that the arguments of  \cite{Vi14,Vi16}  do not apply to a monotone fibre $L\subset X$ of an almost toric fibration on $X$ whose base is not triangular.
\end{remark}

\subsection{Application: infinitely many tori}
We  provide two  applications of Theorem~\ref{th:lag_mut} to Lagrangian tori in del~Pezzo surfaces.
For the first application, recall a theorem of Vianna.

\begin{theorem}[\cite{Vi16}]
	If $X$ is a del Pezzo surface (equipped with a  monotone symplectic form) other than $\Bl_2\C P^2$, then $X$ contains infinitely many monotone Lagrangian tori which are pairwise not~Hamiltonian isotopic.\qed
\end{theorem}
 Although Vianna constructed infinitely many tori in $\Bl_2\C P^2$ which were expected to be all different, it was impossible to prove that they are different without knowing their Landau-Ginzburg potentials. As mentioned in the previous proof and remark, Vianna's estimates on the LG potential  only work in the case of triangular base; this is the reason why $\Bl_2\C P^2$ is not covered by the theorem above.

 Because Theorem~\ref{th:lag_del_pezzo} together with Theorem~\ref{th:lag_mut} provides a complete computation of the LG~potentials all the tori,
we can fill in the missing case in Vianna's theorem.

\begin{corollary}
  \label{cor:bl2cp2}
	The del Pezzo surface
$\Bl_2\C P^2$ (equipped with a  monotone symplectic form) contains infinitely many monotone Lagrangian tori which are pairwise not~Hamiltonian isotopic.
\end{corollary}

\begin{proof}
A look at Table~\ref{tab:dp} shows that the Newton polytope of any of the 10 potentials  is dual to an appropriate scaling of the polygon spanned by the vectors $\{v_i\}$, or equivalently to the polygon which is the affine base of the almost toric fibration giving rise (via Lemma~\ref{lem:constr_lag_disks} and Proposition~\ref{prop:del_pezzo_seed}) to the corresponding Lagrangian seed. The same property is true for arbitrary mutations of those Lagrangian seeds, by the wall-crossing formula and the known way in which mutation of an LG~potential acts on its Newton polytope \cite[Section~3]{ACGK12}.

 It therefore suffices to know that $\Bl_2\C P^2$ admits infinitely many almost toric fibrations with pairwise distinct affine bases, up to the action of $SL(2,\Z)$, which are obtained by mutations from the standard fibration appearing in Proposition~\ref{prop:del_pezzo_seed}. But this has already been shown by Vianna: he proved \cite[Figure~7]{Vi16} that there are at least as many distinct almost toric fibrations on $\Bl_2\C P^2$ as there are Markov triples of the form $(1,a,b)$, which means an infinite number.
\end{proof}

\begin{remark}
For a del~Pezzo Lagrangian seed, consider its (infinite) mutation graph whose vertices correspond to all possible mutated seeds, and 
the valency of a vertex equals the number of \emph{distinct} vectors in the corresponding LG seed.  One expects that each vertex of the mutation graph, up to a finite group quotient of the graph, corresponds to a different monotone torus in the del~Pezzo surface.

The statement is known for $\C P^2$ because the vertices of the infinite trivalent mutation graph are in bijection with ordered Markov triples. The vertices whose underlying \emph{unordered} Markov triples are different correspond to non-isotopic tori in $\C P^2$ by \cite{Vi14}: Markov triples determine the Newton polytopes of the superpotentials of the tori. In this case the finite group in question is  $S_3=\Z/(3)\oplus \Z/(2)$.

 For other del Pezzo surfaces, we are unaware of a complete answer to this question, compare the discussion in \cite[Section~1.2.3]{STW15}.
\end{remark}

\subsection{Application: Fukaya categories and HMS}
\label{sec:application_hms}
We now provide the second application  of the computation of the  potentials of tori in del Pezzo~surfaces. It is related to Sheridan's proof~\cite{She13} of Homological Mirror Symmetry for projective Fano hypersurfaces. We will be interested in  surfaces $X=\Bl_k\C P^2$ where $k=6,7,8$; note that $\Bl_6\C P^2$ is the cubic surface. Our aim is in particular to give an answer to \cite[Conjecture~B.2]{She15}

Recall that for a monotone symplectic manifold $X$ and a number $w\in\C$, one defines the monotone Fukaya $\F(X)_w$ whose objects are weakly unobstructed Lagrangian branes $\bL$ whose obstruction number equals
$$
\mu^0_\bL=w.
$$
This Fukaya category is non-trivial only if $w$ is an eigenvalue of the operator on $QH^*(X,\C)$ acting by the quantum multiplication by the first Chern class:
$$
-*c_1(X)\co QH^*(X)\to QH^*(X).
$$
For the cubic surface $X$, this operator has two eigenvalues:
$$
w=-6,\ w=21.
$$
Moreover, the eigenvalue $w=21$ is simple, see \cite{She15}. One can check that for $\Bl_k\C P^2$, $k=7,8$, there are simple eigenvalues $69$ and $669$, respectively. 

\begin{prop}
	\label{prop:hms_dp}
Consider the monotone torus in $X=\Bl_k\C P^2$, $k=6,7,8$, whose potential $W$ is shown in the corresponding entry of Table~\ref{tab:dp}, see
 Theorem~\ref{th:lag_del_pezzo}. 
 The point $(1,1)$ is a Morse critical point of $W$ with critical value $$w\coloneqq W(1,1)=
 \begin{cases}
21,& k=6,\\
69,&k=7,\\
669,&k=8.
 \end{cases}$$  
 Consider the brane $\bL$ which is the above torus with the trivial local system, so that $\mu^0_\bL=w$. It follows that $HF^*(\bL,\bL)\cong Cl_2$, and $\bL$ split-generates $\F(X)_w$.
\end{prop}

When $k=6$, this answers Conjecture~B.2 from Sheridan's paper \cite{She15}. In that paper, he also sketched the construction of the relevant torus and gave a heuristic computation of the value of its superpotential at the point $(1,1)$.

\begin{proof}
The statement about the critical point is verified by an explicit computation, and the Floer-theoretic implications are well known given the fact that $w$ is a simple eigenvalue, see \cite{She15} for details.	
\end{proof}	

Let $W$ be as in Lemma~\ref{prop:hms_dp} for $k=6,7$ or $8$. Then $W$ has one-dimensional critical locus
$$
1+x+y=0
$$
with critical value resp.~$-6,-12$ or $-60$. These are the non-simple eigenvalues of the quantum multiplication for the first Chern class for $\Bl_k\C P^2$.

\begin{proposition}
  \label{prop:pairwise_disjoint_family}
Let $X=\Bl_k\C P^2$ where $k=6,7,8$ and $w=-6,-12,-60$ respectively. Then $\F(X)_w$ contains an infinite family of pairwise non-quasi-isomorphic toric objects parametrised by the locus
$$
\{x,y\in\C:x+y+1=0\},
$$ which are the branes $\bL$ supported on the  monotone Lagrangian torus provided by Theorem~\ref{th:lag_del_pezzo}, equipped with local systems belonging to the above locus.
\end{proposition}

\begin{proof}
	Because $W$ has order of vanishing $3,4$ resp.~$6$ over $\{x+y+1=0\}$, see Table~\ref{tab:dp}, its second partial derivatives all vanish on this locus. It follows that for any Lagrangian brane $\bL$ as above, its Floer cohomology is isomorphic  to the exterior algebra in two variables:
	$$
	HF^*(\bL,\bL)\cong \Lambda^*(\C^2)\cong H^*(T^2,\C).
	$$
So $\bL$ is a toric object.	
For $\bL,\bL'$ corresponding to two different points of $\{x+y+1=0\}$, we have:
$$
HF^*(\bL,\bL')=0.
$$
This is true because $\bL,\bL'$  are supported on the same torus $L$ but correspond to different local systems on it. In this case the Morse part of Oh's spectral sequence ensures that $\dim HF^*(\bL,\bL')<4$ hence $\bL$, $\bL'$ are not quasi-isomorphic.
\end{proof}

\section{Higher-dimensional toric mutations}
\label{sec:higher_dim}

\subsection{Monotone polytopes}
Let 
$$M=\Z^n$$ 
be a lattice and $N=M^\vee\cong \Z^n$ be its dual. 
We denote by $M_\R=\R^n$ the vector space spanned by $M$, and use the similar notation for $N_\R$.
Let $\Delta\subset M_\R$ be a \emph{lattice} polytope, i.e.~a polytope with vertices in $M$. Codimension~1 faces of $\Delta$ are called  \emph{facets}. Following \cite[Definition~3.1]{MD11} and  \cite{EnPo09}, we recall the definition of  monotone  polytopes: they are the bases of toric fibrations on toric Fano  manifolds, or~monotone toric symplectic manifolds. We point out that we are considering \emph{moment} polytopes which are dual to Fano polytopes commonly used in algebraic geometry.

\begin{defn}
$\Delta$ is called \emph{simple} if exactly $n$ of its facets meet at each vertex. It is called \emph{smooth} if for each vertex $v$, the primitive integral normals to the facets meeting at $v$  form an integral basis of $M$. Equivalently, one can demand that the primitive integral vectors pointing along the edges containing $v$ form an integral basis.
\end{defn}

\begin{defn}
	\label{def:monot_1}
$\Delta$ is called \emph{monotone} if it is simple, smooth, contains $0$ in its interior and is given by the equations
\begin{equation}
\label{eq:delta}
\Delta=
\{x\in\R^n:\langle \eta_j,x \rangle\le 1,\ j=1,\ldots,m\}
\end{equation}
where $\{\eta_j\}_{j=1,\ldots,m}$ is the set of primitive integral outward-pointing normals to the facets of $\Delta$. It is more convenient to view the $\eta_j$ as elements of the dual lattice: $$\eta_j\in N,$$ 
so that the brackets $\langle -,-\rangle$ from (\ref{eq:delta}) become the pairing.
\end{defn}

For example, the polytope given by $\{x_i\ge -1,\ \sum_{i=1}^n x_i\le 1\}$ is monotone; it is a moment polytope of $\C P^n$.
Below is an equivalent definition of monotonicity; the equivalence was proved in  \cite{MD11,EnPo09}.

\begin{defn}
	\label{def:monot_corner}
$\Delta$ is called \emph{monotone} if it is simple, smooth, contains $0$ in its interior and for every vertex $v\in M$ of $\Delta$, there exists an $SL(n,\Z)$-transformation after applying which we obtain: 
$$
v=(-1,\ldots,-1)
$$
and the facets meeting $v$ belong to the planes
$$
\{x_i=-1\},\quad i=1,\ldots,n.
$$
\end{defn}

We note that $SL(n,\Z)=Aut(M)$.
Let $\Delta$ be a monotone polytope, and
$X$ be the toric Fano manifold whose moment polytope is $\Delta$. We consider $X$ as a monotone symplectic $2n$-manifold, and denote the moment map by 
$$p\co X\to\Delta.$$ 

It is proven in \cite[Section~3]{MD11} that $0\in\Delta$ is the unique interior lattice point of $\Delta$, and in \cite[Section~1.4]{EnPo09} that 
$$
L=p^{-1}(0)\subset X
$$
is a monotone Lagrangian torus, which we call the \emph{monotone torus fibre}. The origin of $\Delta$ is sometimes called its \emph{baricentre}.

We remind that the condition of monotonicity for polytopes is preserved by $SL(n,\Z)$-action. Polytopes differing by such action correspond to symplectomorphic manifolds, and the origin always corresponds to the same monotone torus fibre, see Figure~\ref{fig:cp2_forms}.

\begin{figure}[h]
	\includegraphics[]{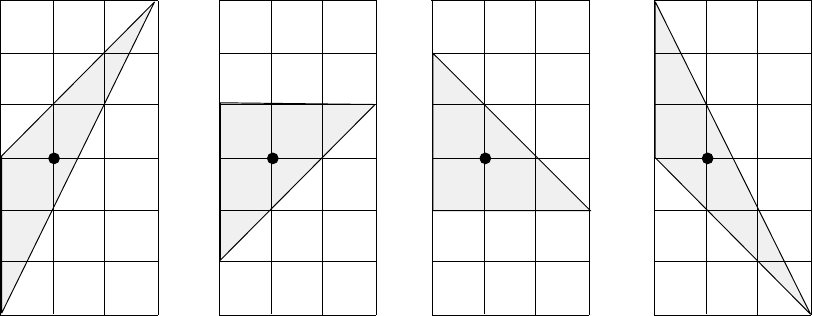}
	\caption{Several moment polytopes of $\C P^2$, differing by $SL(2,\Z)$-action.}
	\label{fig:cp2_forms}
\end{figure}

\subsection{Toric potential and mutation configurations}
Let $\Delta\subset M_\R$ be a monotone polytope, $X$  the corresponding toric manifold, and $L=p^{-1}(0)$ the monotone toric fibre.
 There is a canonical isomorphism $$N=H_1(L,\Z),$$ 
which, fixing the standard basis of $N\cong \Z^n$, provides a basis of $H_1(L,\Z)$.
By a theorem of Cho and~Oh~\cite{CO06},
the Landau-Ginzburg potential of $L$ (see Definition~\ref{def:pot_LG}) is given in this basis by:
 \begin{equation}
 \label{eq:toric_pot}
 W_L=\sum_{j=1}^m \mathbf{x}^{\eta_j}
\end{equation}
 where the sum is taken over the normal vectors $\eta_j\in N$ from Definition~\ref{def:monot_1}.
We are using the following standard notation:
$$
\mathbf{x}^{u}=x_1^{u_1}\ldots x_n^{u_n},\quad u=(u_1,\ldots,u_n)\in N,\quad u_i\in\Z.
$$
Here we consider $W_L$ as a function naturally defined on the algebraic torus $(\C^*)^n$ associated with the lattice $N$: $(\C^*)^n=\C_N/N$.

\begin{defn}
	\label{def:mut_config}	
Let $F\subset \Delta$ be a face of dimension 
$$
\dim F\le n-2
$$
(i.e.~any face except a facet) and 
$$w\in F\cap M$$ 
be an \emph{interior} lattice point of $F$. We say that $(F,w)$ is a \emph{mutation configuration}.
\end{defn}

\begin{remark}
A vertex of $\Delta$ is considered to be an interior lattice point of itself, so it always provides a mutation configuration $F=\{w\}$. However, higher-dimensional faces do not necessarily contain interior lattice points.
\end{remark}

For a face $F$ of $\Delta$, we shall denote
$$
\Pi_F=\mathrm{Span}_{M_\R}\{F,\{0\}\}\subset M_\R,
$$
which is a linear subspace of dimension $\dim F+1$.
One can consider its annihilator in the dual space: 
$\Pi_F^\perp\subset N_\R$. Observe that $N\cap \Pi_F^\perp$ is a full rank lattice in $\Pi_F^\perp$, so one can speak of an integral basis of $\Pi_F^\perp$. See Figure~\ref{fig:std_corner}.

\begin{figure}[h]
	\includegraphics[]{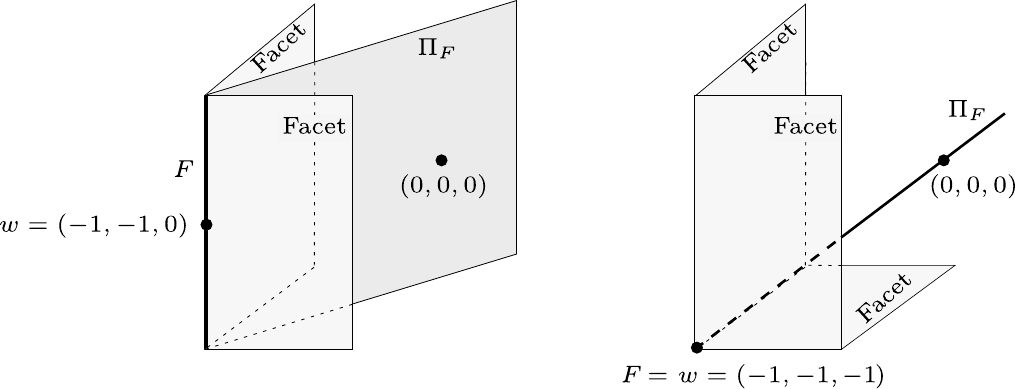}
	\caption{Pieces of polytopes exhibiting a mutation configuration in standard form. Left: a mutation configuration where $F$ is an edge. Right: a mutation configuration where $F=\{w\}$ is a vertex.}
	\label{fig:std_corner}
\end{figure}

It is useful to record the next simple lemma.

\begin{lemma}
	\label{lem:std_w}
Suppose $(F,w)$ is a mutation configuration for a monotone polytope $\Delta$, and $\dim F=n-k$. Then there is an $SL(n,\Z)$-transformation after applying which, $F$ belongs to the subspace
	$$
	\{x_1=\ldots=x_k=-1\},
	$$
the point $w$ becomes 
$$
w=(\underbrace{-1,\ldots,-1}_{k \textit{ times}},0,\ldots,0),
$$	
and in a neighbourhood of $w$, the polytope $\Delta$ becomes  given by the equations:
$$
\{x_1\ge -1,\ldots, x_k\ge -1\}.
$$
Once  $\Delta$ is brought to this form, the following is an integral basis of $\Pi_F^\perp$:
\begin{equation}
\begin{array}{ll}
u_1=&(-1,0,\ldots,0,1,0,\ldots,0),\\
u_{2}=&(0,-1,\ldots,0,1,0,\ldots,0),\\
\vdots\\
u_{k-1}=&(0,0,\ldots,-1,1,0,\ldots,0)
\end{array}
\end{equation}
\end{lemma}

\begin{proof}
If follows from the monotonicity condition (Definition~\ref{def:monot_corner}) that there exists an $SL(n,\Z)$-transformation taking $F$ to the subspace above, and $\Delta$ to the form specified above in a neighbourhood of any interior point of $F$. Such transformation takes $w$ to a point of the form
$$
(-1,\ldots,-1,r_{k+1},\ldots,r_n),\quad r_i\in \Z.
$$
Next consider an $SL(n,\Z)$-transformation $\phi$ taking 
$$
(-1,\ldots,-1,r_{k+1},\ldots,r_n)\mapsto (-1,\ldots,-1,0,\ldots,0),
$$
thus bringing $w$ to the desired point,
which additionally preserves the linear subspace $\R^k\times \{0\}$ and its orthogonal complement, and acts by the identity on the orthogonal complement. Then $\phi$ preserves the affine plane
$\{x_1=\ldots=x_k=-1\}$, so $F$ stays within this plane under $\phi$. Also, $\phi$ preserves the normals to the facets meeting at $F$, so $\Delta$  stays in the form specified in the statement.
The clause about the basis of $\Pi_F^\perp$ is obvious.
\end{proof}

\begin{definition}
	\label{dfn:std_form}
Once $\Delta$ is brought to the form specified in Lemma~\ref{lem:std_w}, we say that $\Delta$ is in \emph{standard form} with respect to $(F,w)$.
\end{definition}

\subsection{Toric mutation}

\begin{theorem}[Higher-dimensional toric mutation]
	\label{th:mut_higher_dim_toric}
	Suppose $\Delta$ is a  monotone polytope and $(F,w)$ is its mutation configuration where $\dim F=n-k$. Then the corresponding toric manifold $X$ contains a monotone Lagrangian torus denoted by
	$$
	L'=\mu_{F,w}L
	$$
	whose LG~potential (see Definition~\ref{def:pot_LG}), written in some basis of $H_1(L',\Z)$, is equal to the function obtained from $W_L$ by substituting each monomial in (\ref{eq:toric_pot}) according to the rule below:
\begin{equation}
\label{eq:wc_map_higher}
\mathbf{x}^v\mapsto \mathbf{x}^v(1+\mathbf{x}^{u_1}+\ldots+\mathbf{x}^{u_{k-1}})^{-\langle v,w\rangle}.
\end{equation}
Here $v\in N$, $\langle v,w\rangle\in \Z$ is the pairing (recall that $w\in M$), and $u_1,\ldots, u_{k-1}$ is an integral basis of  $\Pi_F^\perp$.
\end{theorem}

There is a birational map $(\C^*)^n\dashrightarrow(\C^*)^n$ whose action on monomials is given by (\ref{eq:wc_map_higher}), and this map is called the \emph{wall-crossing map}. So the statement of the theorem says that $W_L$, $W_{L'}$ differ by the action of the wall-crossing map.

Similarly to the discussion in Section~\ref{sec:mut_4d}, a Laurent polynomial does not a~priori remain a Laurent polynomial under the substitutions (\ref{eq:wc_map_higher}). But in the setting of Theorem~\ref{th:mut_higher_dim_toric}, $W_L$ has to remain a Laurent polynomial under the specified mutation because it becomes the LG~potential of another monotone Lagrangian torus.

\begin{example}
	Consider a moment polytope of a toric del~Pezzo surface, and let  $F=w=(u_1,u_2)$ be its vertex. In the statement of Theorem~\ref{th:mut_higher_dim_toric}, this means $k=n=2$. Next, $\Pi_F$ is the line spanned by $w$ and the origin, so we can take the vector 
	$(u_2,-u_1)$ to be the basis of the 1-dimensional space $\Pi_F^\perp$. Hence (\ref{eq:wc_map_higher}) becomes the wall-crossing map (\ref{eq:mutation_monomial}) studied in Section~\ref{sec:mut_4d}. So in dimension~four, Theorem~\ref{th:mut_higher_dim_toric} is a particular case of Theorem~\ref{th:lag_del_pezzo}.
\end{example}

\begin{remark}
The case $k=1$ (i.e.~taking $F$ to be a facet) is disallowed in Theorem~\ref{th:mut_higher_dim_toric}, but it is also natural to put (\ref{eq:wc_map_higher}) to be the identity map for $k=1$.
\end{remark}

\begin{example}
	\label{ex:wc_std}
Suppose $\Delta$ is in standard form with respect to $(F,w)$, see Definition~\ref{dfn:std_form}. For the choice of a basis of $\Pi_F^\perp$ from Lemma~\ref{lem:std_w}, the wall-crossing map is given by:
	\begin{equation}
	\label{eq:wc_std_corner}
	\begin{array}{ll}
		x_i\mapsto x_i x_k(x_1^{-1}+\ldots+x_k^{-1}),& 
		i=1,\ldots,k,\\
		x_i\mapsto x_i,&
		i=k+1,\ldots,n.
	\end{array}
	\end{equation}		
\end{example}

\begin{example}
	Let $\Delta$ be the standard moment polytope of $\C P^n$:
	$$
	\{x_i\ge -1,\ \textstyle\sum_i x_i\le 1\}
	$$
	Let $F$ be the $(n-k)$-dimensional face of $\Delta$ which belongs to the plane
	$$
	\{x_1=\ldots=x_k=-1\},
	$$
	and consider its interior lattice point 
	$$w=(\underbrace{-1,\ldots,-1}_{k \textit{ times}},0,\ldots,0).$$ 
	Here $k$ is any integer between $2$ and $n$.
	The toric potential (\ref{eq:toric_pot}) equals:
	$$
	W_L=x_1^{-1}+\ldots+x_{n}^{-1}+x_1\ldots x_n.
	$$ Using Theorem~\ref{th:mut_higher_dim_toric} and the wall-crossing computed in (\ref{eq:wc_std_corner}), the potential of the torus mutated along $(F,w)$ equals:
	\begin{equation}
	\label{eq:pot_mut_cpn}
	W_{\mu_{F,w}L}=\textstyle\sum_{i=k}^n x_i^{-1}+(x_k)^k\cdot (\textstyle\sum_{i=1}^k x_i^{-1})^k\cdot \textstyle\prod_{i=1}^n x_i
	\end{equation}
	When $k=1$, the above formula simply yields $W_L$.
\end{example}

\begin{corollary}
  \label{cor:cpn_tori}
	For each $1\le k\le n$, $\C P^n$ contains a monotone Lagrangian torus whose potential is given by (\ref{eq:pot_mut_cpn}).\qed
\end{corollary}

In the rest of the section we shall prove Theorem~\ref{th:mut_higher_dim_toric}, and finish off by an example illustrating that  it is essential to consider \emph{interior} lattice points in the definition of a mutation configuration.

\begin{lemma}
	\label{lem:one_basis_enough}
If Theorem~\ref{th:mut_higher_dim_toric} holds for a single choice of a basis $u_1,\ldots,u_{k-1}$ of $\Pi_F^\perp$, then it holds for all such choices.
\end{lemma}

\begin{proof}
After an $SL(n,\Z)$-transformation, one can assume that  $\Pi_F$ is spanned by the last $n-k+1$ basis vectors in $M$, $\Pi_F^\perp$ is spanned by the first $k-1$ basis vectors in $N$, and $w\in M$ is the $k$th basis vector. Choosing the basis of $\Pi_F^\perp$ to consist of the first $k-1$ basis vectors, (\ref{eq:wc_map_higher}) becomes:
\begin{equation}
\label{eq:wall_cross_simple}
\begin{array}{ll}
y_i\mapsto y_i,\quad  i=1,\ldots, k-1,\\
y_k\mapsto y_k(1+y_1+\ldots+y_{k-1}),\\
y_i\mapsto y_i,\quad i=k+1,\ldots,n.
\end{array}
\end{equation}
In this presentation it is clear that $SL(k-1,\Z)$ acting on the first $k-1$ co-ordinates (both on the source and the target) commutes with (\ref{eq:wall_cross_simple}), and the claim follows.
\end{proof}

\subsection{Local model}
In order to define toric mutation and prove Theorem~\ref{th:mut_higher_dim_toric}, we shall use a higher-dimensional version of the local model discussed in Section~\ref{sec:mut_4d}.
The technical implementation of this discussion will be somewhat harder than the 4-dimensional; this is mostly to do with constructing a neighbourhood of a mutation configuration which is convex with respect to the standard $J$.

We use the following symplectic manifold as the model: 
\begin{equation}
\label{eq:M_higher}
M=\C^k\setminus\{z_1\ldots z_k=\epsilon \}.
\end{equation}
We begin with
the symplectic form $\omega_{std}$ which is the restriction of the standard one from $\C^n$. The form $\omega_{std}$ is \emph{not} Liouville on $M$ because is does not `blow up' along $\{z_1\ldots z_k=\epsilon\}$. There exists a Liouville structure on $M$ and we will use it later, but right now we shall work with $\omega_{std}$.

Consider the holomorphic fibration
\begin{equation}
\label{eq:pi}
\pi\co M\to \C\setminus \{\epsilon\},\quad \pi(z_1,\ldots,z_k)=z_1\ldots z_k,\quad z_i\in \C.
\end{equation}
The number $\epsilon\in\R_{>0}$ shall be fixed throughout the discussion.
Proceeding as in Section~\ref{sec:mut_4d}, for any simple closed curve 
$\gamma\subset \C\setminus\{0,\epsilon\}$ which encloses the point $\{\epsilon\}$, one defines a Lagrangian torus
\begin{equation}
\label{eq:T_gamma_higher}
T_{\gamma}=\{(x_1,\ldots,x_k)\in \C^k:\  \pi(x_1,\ldots,x_k)\in\gamma,\  |x_1|=\ldots=|x_k|\}\subset M.
\end{equation}
(In the language of Section~\ref{sec:mut_4d}, this torus would have been called $T_{\gamma,0}$. Also, in Section~\ref{sec:mut_4d} we had $\epsilon=1$.)
Once again, the tori $T_\gamma$ are divided into two Hamiltonian isotopy classes, depending on whether or not $\gamma$ encloses the origin. We call the tori of the first type \emph{Clifford-type} tori, and tori of the second type \emph{Chekanov-type}.

The role of the Lagrangian disk appearing in Lemma~\ref{lem:config_W}
is now played by a singular Lagrangian $D_\delta$ which is topologically a cone over $T^{n-1}$. This singular Lagrangian $D_\delta$ is attached to the torus $T_{\gamma}$ along an $(n-1)$-dimensional subtorus of $T_\gamma$. In a formula:
$$
D_\delta=\{(x_1,\ldots,x_k)\in \C^k:\ \pi(x_1,\ldots,x_k)\in \delta, \ |x_1|=\ldots=|x_k|\},
$$
where $\delta\subset \C\setminus\{1\}$
is any simple path connecting a point of $\gamma$ with the origin. 
Different choices of $\delta$ produce Hamiltonian isotopic configurations $T_\gamma\cup D_\delta$.

Ultimately, we are interested in the standard choice of a Clifford-type torus with an attached Lagrangian cone:
\begin{equation}
\label{eq:def_T_D}
T=T_\gamma,\quad D=D_\delta
\end{equation}
where
$$
\gamma=\{z:|z|=1\},\quad \delta=[-1,0].
$$
One expects that $T\cup D\subset M$ has a neighbourhood $U$ such that:
\begin{itemize}
	\item $U$ is a Liouville domain with respect to $\omega_{std}$,
	\item the standard complex structure $J$ is cylindrical near $\bd U$.
\end{itemize}
In the next subsection, we are going to establish a  weaker version of this statement sufficient for the purpose of defining toric mutations and proving the wall-crossing formula.

\begin{remark}
	\label{rmk:nearby_chek}
Any given neighbourhood $U$ of $T\cup D$ contains  a (Hamiltonian isotopic copy of) Chekanov-type torus. Indeed,	
by an appropriate choice of $\gamma$, a Chekanov-type torus $T_\gamma$ can be arranged to lie arbitrarily close to the standard configuration $T\cup D$.
\end{remark}

\begin{remark}
\label{rmk:no_wstein}
It is slightly easier to show that $M$ has a Liouville structure (with a symplectic form other than $\omega_{std}$) whose skeleton is $T\cup D$. We are going to prove this as an intermediate step, but here is an explanation why we eventually need to show the existence of a neighbourhood $U$ of $T\cup D$ which is convex (or Liouville) with respect to $\omega_{std}$. Recall that in dimension~4, we proved Lemma~\ref{lem:nbhood_for_T_D} which is a version
of the Weinstein neighbourhood theorem for a mutation configuration. In the present situuation, we will avoid using and proving a version of the Weinstein neighbourhood theorem for the Lagrangian skeleton $T\cup D$, as our toric models (studied later) will be immediately symplectomorphic to neighbourhoods of the form $(U,\omega_{std})$---see Lemma~\ref{lem:toric_T_D} below.

The Weinstein neighbourhood theorem for the Lagrangian CW complex formed by the union of an $n$-torus and an attached Lagrangian cone over an $(n-1)$-subtorus has a subtlety when $n>1$. Such neighbourhoods have a symplectic invariant, namely the Legendrian isotopy class of the Legendrian $(n-1)$-torus in $S^{2n-1}$ which is the link of the conical singularity. 
The correct version of the Weinstein neighbourhood theorem must require that
the link invariants coincide for the two  configurations in question. While it is possible to prove this version of the theorem, we bypass the need for it.
\end{remark}

\begin{remark}
We announced that $T\cup D\subset U$ will be used as a local model for defining toric mutations. Let us make it precise: this is the model which shall be used for Theorem~\ref{th:mut_higher_dim_toric} in the case $k=n$. When $k<n$, it has to be additionally multiplied by $(\C^*)^{n-k}$ with the trivial Lagrangian $(n-k)$-torus therein.
\end{remark}

We conclude this subsection by an obvious lemma summarising the relevance of our local model to toric geometry.

\begin{lemma}
	\label{lem:toric_T_D}
Consider the standard toric fibration $p\co \C^k\to (\R_{\ge -1})^k$, and  consider the line segment
$$I=\{(l,\ldots, l):-1\le l\le 0\}\subset (\R_{\ge -1})^k.$$
(Compare with Figure~\ref{fig:std_corner}, right.) Then, for the standard $T$ and $D$ from (\ref{eq:def_T_D}), we have
$$
T=p^{-1}(0,\ldots,0),\quad p(D)=I.
$$
\end{lemma}	

\begin{proof}
The fibration $p$ is given by
$$
(z_1,\ldots,z_k)\mapsto (|z_1|^2-1,\ldots,|z_k|^2-1).
$$
The torus $T$ is given by the equations
$$
\{|z_1|=\ldots=|z_k|=1\},
$$
so it coincides with $p^{-1}(0)$. The Lagrangian cone $D$ is explicitly given by
$$
\{(re^{i\phi_1},re^{i(\phi_2-\phi_1)},\ldots,re^{i(\phi_{k-1}-\phi_{k-2})},re^{i(\pi-\phi_{k-1})}):0\le r\le 1,\ \phi_j\in[0,2\pi]\}
$$
and we see that  $p(D)=I$.
\end{proof}	

\subsection{A convex neighbourhood}
The aim of this subsection is to prove a weaker form (Proposition~\ref{prop:higher_conf_nbhd}) of the statement that there are neighbourhoods $(U,\omega_{std})$ of $T\cup D$ which are Liouville domains, and such that the standard $J$ is convex.
 
Recall that $M=\C^k\setminus\{z_1\ldots z_k=\epsilon\}$.
Consider the function
\begin{equation}
\label{eq:Phi_0}
\Phi_0=\sum_{i=1}^k|z_k|^2
\end{equation}
on $\C^k$, restricted to $M$. Then
$$
\omega_{std}=dd^c\Phi_0.
$$
Recall that for any function $f$, $d^cf=df\circ J$ where $J$ is the standard complex structure on $\C^k$. 

The existence of a Liouville structure on $M$ is well known, see e.g.~\cite[Section~4]{Au17}; let us remind it. One starts with a K\"ahler potential $\tilde\Phi$ on $M$ blowing up near the `boundary-at-infinity' of $M$: for example the most standard choice could be
$$
\tilde \Phi=\textstyle\sum_i |z_i^2|+(\log(|z_1\ldots z_k-\epsilon|))^2.
$$
One takes the following symplectic form on $M$:
$$
\omega=dd^c\tilde\Phi,
$$
then its has the Liouville form  $d^c\tilde\Phi$. Recall that different choices of K\"ahler potentials give symplectomorphic  symplectic structures on $M$.

It is possible, but not very easy, to compute the skeleton of $M$ with respect to above Liouvile structure. It will turn out that the skeleton is of the form $T_\gamma\cup D_\delta$. The vanishing locus of $d^c\tilde\Phi$ is composed of two isotropic $(k-1)$-tori within $T_\gamma$, and the origin as an isolated point. To compute the whole skeleton, one has to determine all flowlines of $d^c\tilde\Phi$ which do not escape to infinity.

We now will modify the potential $\tilde\Phi$ to make the determination of the skeleton easier.
Denote 
$$
F=(|z_1|-|z_2|)^2+(|z_2|-|z_3|)^2\ldots+(|z_{k-1}|-|z_k|)^2+(|z_k|-|z_1|)^2.
$$

\begin{figure}[h]
\includegraphics[]{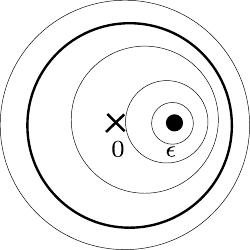}
\caption{The level sets of a function $g\co \C\setminus\{\epsilon\}\to\R$. The minimum is achieved at the unit circle (bold).}
\label{fig:kahler}
\end{figure}
Next, let 
$$
g(z)\co\C\setminus\{\epsilon\}\to \R
$$
be a subharmonic function whose critical locus coincides with the circle $|z|=1$ where it achieves its minimum,  which tends to $+\infty$ as $z\to\epsilon$ or $z\to\infty$, and which is conjugation-invariant. Figure~\ref{fig:kahler} shows how the level sets of $g$ should look like.
Set 
$$
G=g(z_1\ldots z_k)
$$
and 
\begin{equation}
\label{eq:Phi}
\Phi=F+G\co M\to\R.
\end{equation}
This is the K\"ahler potential that we are going to use for determining the skeleton of $M$; we use the symplectic form $\omega=dd^c\Phi$ on $M$.

Observe that
$$
\mathrm{Crit}\, F=\{(z_1,\ldots,z_k): |z_1|=\ldots=|z_k|\}.
$$
Next, by restricting to the fibres of $\pi$, we see
that $\mathrm{Crit}\,\Phi\subset \mathrm{Crit}\, F$.
A further computation shows that
$$
\mathrm{Crit}\,\Phi=\{0\}\cup\{(z_1,\ldots,z_k):|z_1|=\ldots=|z_k|=1\}=\{0\}\sqcup T.
$$
We remind that the holomorphic map $\pi$ was introduced in (\ref{eq:pi}), and the torus $T$ in (\ref{eq:def_T_D}).
\begin{corollary}
It holds that $T$ is $d^c\Phi$-exact.
\end{corollary}
\begin{proof}
	Indeed, $d^c\Phi=d\Phi\circ J$ vanishes on $T$.
\end{proof}	

The skeleton of $M$ is the union of $\mathrm{Crit}\,\Phi$ with all flowlines of the Liouville field $v$, $\iota_v\omega=d^c\Phi$, which do not escape to infinity, meaning that they  converge to $\{0\}$ at time $+\infty$ and to a point in $T$ at time $-\infty$. 
For further analysis, let us point that both $\Phi$ and $\Phi_0$ are invariant under:
\begin{itemize}
	\item the $T^{k-1}$ action generated by the $S^1$-actions of the form $(z_i\mapsto z_ie^{i\phi},\, z_j\mapsto z_je^{-i\phi})$,
	\item the $\Z/k\Z$-action cyclically permuting the co-ordinates,
	\item the complex conjugation of all co-ordinates simultaneously.
\end{itemize}
\begin{lemma}
The union of all non-constant non-escaping trajectories of the Liouville vector field with respect to the Liouville form $d^c\Phi$, where $\Phi$ is given by (\ref{eq:Phi}), is precisely 
$$
D=\{(z_1,\ldots,z_k): |z_1|=\ldots=|z_k|, \ z_1\ldots z_k\in[-1,0]\subset \C\}.
$$
The skeleton of $M$ is therefore
the standard configuration $T\cup D$.
\end{lemma}

\begin{proof}	
Let us check that the Liouville field is tangent to $D$. Using the invariance of $\Phi$ under the $T^{k-1}$-action, it suffices to check the tangency at the real points
$$
(r,\ldots,r)\in D,\quad 0\le r\le 1.
$$ 
By the $\Z/k\Z$-invariance,
the Liouville field belongs to the diagonal complex line $\{(\zeta,\ldots,\zeta):\zeta\in\C\}$. By conjugation-invariance,
the Liouville field is pointing in the direction $(\bd_r,\ldots,\bd_r)$ which is tangent to $D$. Therefore $D$ is composed of non-escaping trajectories.
It is not hard to check that there are no other non-escaping trajectories; we leave it to the reader.
\end{proof}

\begin{remark}
One can check that within the plane $(\zeta,\ldots,\zeta)$, the Liouville field is the $\pi$-lift of the Liouville field of the function $g$, which is simply the planar gradient of $g$. Conjugation-invariance implies that $(-1,0)$ is a flowline of $g$.
\end{remark}

\begin{corollary}
\label{cor:higher_skel}
The skeleton of $M$ with respect to the K\"ahler potential $\Phi$ (\ref{eq:Phi}) is $T\cup D$.\qed 	
\end{corollary}

It follows  that $D$ is Lagrangian with respect to $\omega=dd^c\Phi$. Next we show  all Clifford and Chekanov-type tori $T_\gamma$ introduced in (\ref{eq:T_gamma_higher})
are Lagrangian with respect to $\omega$ as well as $\omega_{std}$.

\begin{lemma}
	\label{lem:lag_t}
Every torus $T_\gamma\subset M$ is Lagrangian with respect to $\omega=dd^c\Phi$. 
\end{lemma}

\begin{proof}
To begin with, we claim that the restriction of $d^c\Phi$  to the 
$(k-1)$-torus
$$
S=\{(z_1,\ldots,z_k):|z_1|=\ldots=|z_k|,\ z_1\ldots z_k=p\}
$$
for any fixed $p\in \C$, vanishes. The torus lies in the fibre of $\pi$ which is complex, therefore, 
$$d^c\Phi|_S=d^c(\Phi|_{fibre})|_S=d^c(F|_{fibre})|_S.$$
Recall that $dF=0$ at every point of $S$. So $d^cF=dF\circ J$ also vanishes at every point of $S$.

By definition, $T_\gamma$ is a union of an $S^1$-family of $(k-1)$-tori of the above form, fibering over $\gamma$. It follows from what we just proved and the $T^{k-1}$-invariance of $d^c\Phi$ that there is a quotient map $p\co T_\gamma\to S^1$ (which can be seen as the projection by $\pi$ onto $\gamma$) such that $d^c\Phi|_T$ is the $p$-pullback of some 1-form on $S^1$. Because all 1-forms on $S^1$ are closed, $dd^c\Phi$ vanishes on $T$.
\end{proof}

We proceed to the main result of the subsection.

\begin{proposition}
	\label{prop:higher_conf_nbhd}
Let $T,D\subset M$ be the standard configuration from (\ref{eq:def_T_D}). For any  sufficiently small neighbourhood $V$ of $T\cup D$ there exists another neighbourhood $U$ (which is arbitrarily small, when $V$ is arbitrarily small), and a symplectomorphism
$$
\phi\co(U,\omega_{std})\to (V,dd^c\Phi)
$$
such that 
$$
\phi(T)=T,\quad \phi(D)=D,
$$
and $\phi$ takes Clifford- or~Chekanov-type tori to Clifford resp.~Chekanov-type tori, i.e.~$\phi(T_\gamma)=T_{\gamma'}$ where $\gamma$ and $\gamma'$ both enclose (resp.~do not enclose) the origin. Here $\Phi$ is taken from (\ref{eq:Phi}).
\end{proposition}

\begin{proof}[Proof of Proposition~\ref{prop:higher_conf_nbhd}]
Recall that $\Phi_0=\sum |x_i|^2$ denotes the standard K\"ahler potential on $\C^n$ so that $dd^c\Phi_0=\omega_{std}$. 
Fix a neighbourhood $V$ of $T\cup D$.
Consider an linearly interpolating family of symplectic forms on $V$,
$$
\omega_t=tdd^c\Phi+(1-t)dd^c\Phi_0,\quad t\in[0,1].
$$
We wish to apply Moser's trick to this family intended to provide a symplectomorphism 
$$
\psi\co
(V,dd^c\Phi)\xrightarrow{\cong}(U,\omega_{std}),
$$
the inverse to the desired map $\phi$.
Under Moser's trick, $U$ would be defined as the image of $V$ under the time-1 flow of  a vector field $v_t$. The main point is to make sure that the flow defining $\psi$ preserves $T$ and $D$; if it does, then $U$ will also be a neighbourhood of $T\cup D$.

Because $\bd \omega_t/\bd t=d(-d^c\Phi+d^c\Phi_0)$,
a naive attempt to use Moser's trick is to define the vector field $v_t$ by
$$
\iota_{v_t}\omega_t
=-d^c\Phi+d^c\Phi_0.
$$ 
This vector field is tangent to $D$ but not to $T$, so it is not readily suitable. 
Before we fix this issue, let us explain why $v_t$ is tangent to $D$. (At the origin which is singular point of $D$, this simply means that $v_t$ vanishes.)
Since $D$ is Lagrangian with respect to $\omega_t$ for all $t$, the fact that $v_t$ is tangent to $D$ is equivalent to the fact that $(-d^c\Phi+d^c\Phi_0)|_D=0$. 
It is obvious that $d^c\Phi_0|_D=0$, and one checks that $d^c\Phi|_D=d^cG|_D=0$.

We claim that there exists a \emph{closed} 1-form $\vartheta$ on $V$
such that
$$
(-d^c\Phi+d^c\Phi_0)|_{T\cup D}=\vartheta|_{T\cup D}.
$$
To see this, first note that since $T$ is Lagrangian with respect to $dd^c(-\Phi+\Phi_0)$, the form $d^c(-\Phi+\Phi_0)|_T$ is closed. Its cohomology class lies in the image of
\begin{equation}
\label{eq:incl_nbhd}
\R\cong H^1(V,\R)\to H^1(T,\R)
\end{equation}
because $d^c(-\Phi+\Phi_0)|_T$ extends as a closed (in fact, vanishing) form to $T\cup D$, and $T\cup D$ is homotopy equivalent to $V$.
We take $\vartheta$ to be a preimage of $-d^c\Phi+d^c\Phi_0$ under (\ref{eq:incl_nbhd}). Additionally, because $-d^c\Phi+d^c\Phi_0$ vanished on $D$, we can arrange that $\vartheta$ vanishes on $D$ as well.
As a final adjustment, we can arrange $\vartheta$ to be $T^{k-1}$-and $\Z/k\Z$-invariant by averaging, since $-d^c\Phi+d^c\Phi_0$ is
already invariant under these actions.

Next, we define the adjusted vector field $v_t$ by 
\begin{equation}
\label{eq:def_v_t}
\iota_{v_t}\omega_t
=-d^c\Phi+d^c\Phi_0-\vartheta.
\end{equation}
It is now tangent to $T\cup D$ as the right hand side of the above equation vanishes on $T\cup D$.
Its time time-1 flow takes $\omega$ to $\omega_{std}$ and maps a neighbourhood of $T\cup D$ to another neighbourhood of the same set, so provides the desired symplectomorphism $\psi$, and we take $\phi$ to be its inverse.

It remains to verify that the flow of $v_t$ takes Clifford- or Chekanov-type tori to Clifford- resp.~Chekanov-type tori; then the time-1 flow $\psi$ will have the same property. All ingredients in the defining equation (\ref{eq:def_v_t}) for $v_t$ are $T^{k-1}$- and $\Z/k\Z$-invariant; consequently, so is $v_t$. It follows from the $\Z/k\Z$-invariance that the set
$$
\{(z_1,\ldots,z_k):|z_1|=\ldots=|z_k|\}
$$ 
is preserved by the flow of $v_t$. Next, it follows by $T^{k-1}$-invariance that the $(k-1)$-tori of the form
$$
\{(z_1,\ldots,z_k):|z_1|=\ldots=|z_k|,\ z_1\ldots z_k=p\}
$$
are taken under the flow of $v_t$ to the tori of the same form. Since every torus $T_\gamma$ is a union of such $(k-1)$-tori fibering over $\gamma$, the flow of $v_t$  takes any such torus to a torus $T_{\gamma'}$ for an isotopic curve $\gamma'$. Because $v_t$ vanishes at $0\in\C^k$, the isotopy of curves from $\gamma$ to $\gamma'$  never crosses the point $0\in\C$.
\end{proof}

\subsection{Local wall-crossing}
The final ingredient for Theorem~\ref{th:mut_higher_dim_toric} is  a higher-dimensional version of the local wall-crossing formula. 
We continue using $T,D$ from (\ref{eq:def_T_D}) and $\Phi$ from (\ref{eq:Phi}). We begin with a simple lemma.

\begin{lemma}
	\label{lem:conv_J}
There exists an arbitrarily small neighbourhood $V$ of $T\cup D$ such that $(V,d^c\Phi)$ is a Liouville domain, and $T$ is exact.
Moreover:
\begin{itemize}
	\item 
	all other Clifford- and Chekanov-type tori $T_\gamma\subset V$ are exact, provided that $\gamma$ encloses a disk of area $\pi$ (same area as the unit circle defining the torus $T$);
	\item the standard complex structure $J$ is convex at the boundary of $V$.
\end{itemize}
\begin{proof}
	We already know that $T$ is exact and all tori $T_\gamma$ are Lagrangian. A torus $T_\gamma$ from the first clause admits a Lagrangian (non-Hamiltonian) isotopy to $T$ with zero total flux, therefore it is also exact. The second clause is automatic from the definition of the Liouville form as $d^c\Phi$.
\end{proof}

\end{lemma}

\begin{lemma}[Local wall-crossing]
	\label{lem:seidel_local_higher}
	Let $L_0,L_1\subset (V,d^c\Phi)$ be an exact Clifford-type resp.~Chekanov-type torus as in the previous lemma. 
	
	Recall that $L_0$ is (Hamiltonian isotopic to) the fibre of the standard toric fibration $p\co \C^n\to (\R_{\ge -1})^k$, see Lemma~\ref{lem:toric_T_D}. 
	Fix a basis of $H_1(L_0,\Z)$ induced by the standard basis of $\R^k$.
	
	 There exists a basis of
	$H_1(L_1,\Z)\cong \Z^k$ with the following property. 
	If 
	$$
	\rho_{i}\in (\C^\x)^k,\quad i=0,1
	$$
	are arbitrary local systems on the $L_i$ and $\bL_i=(L_i,\rho_i)$ then 
	$$
	HF^*_V(\bL_0,\bL_1)\neq 0
	$$
	if and only if
	$$
	\rho_{1}=\mu  (\rho_{0})
	$$
	where the wall-crossing map $\mu\co (\C^*)^k\dashrightarrow(\C^*)^k$ is given by
\begin{equation}
\label{eq:wall_cross_higher_corner}
x_i\mapsto x_i x_k(x_1^{-1}+\ldots+x_k^{-1}),\quad
i=1,\ldots,k.	
\end{equation}
\end{lemma}

\begin{remark}
	This lemma is essentially equivalent to Theorem~\ref{th:mut_higher_dim_toric} where we take $\Delta$ to be $(\R_{\ge -1})^k$, and $F=\{w\}=\{(-1,\ldots,-1)\}$. (This polytope $\Delta$ is not compact, but the compactness requirement for Theorem~\ref{th:mut_higher_dim_toric} is in fact not crucial.)  See also Example~\ref{ex:wc_std}.
\end{remark}

\begin{proof}
	We provide a sketch of proof and leave the details to the reader, since this is a relatively straightforward analogue of the 4-dimensional case from \cite[Proposition~11.8]{SeiBook13}, compare Lemma~\ref{lem:seidel_local}. The computation heavily relies on the map
	$\pi\co M\to \C\setminus\{\epsilon\}$,
	$\pi(z_1,\ldots,z_k)=z_1,\ldots z_k$ holomorphic with respect to the standard $J$. We denote the restriction of $\pi$ to the neighbourhood $V$ by the same symbol. 
Since the standard $J$ is cylindrical on the Liouville domain $(V,d^c\Phi)$ by  Lemma~\ref{lem:conv_J}, we can use it to compute the Floer cohomology in question.

	It is more convenient to perform  the computation in different bases of $H_1(L_i,\Z)$ than those appearing in the statement. The more convenient bases are the ones  introduced in the proof of Lemma~\ref{lem:one_basis_enough}; let us explain their gemetric meaning. For each of the two Lagrangians $L_0$, $L_1$, we choose a basis so that the first $(k-1)$ elements are a basis for the homology of the $(k-1)$-subtorus which is the fibre of $\pi$. We claim that for a suitable choice of the last elements of the bases, the statement of Lemma~\ref{lem:seidel_local_higher} holds with the following wall-crossing map taken instead of (\ref{eq:wall_cross_higher_corner}):
\begin{equation}
\label{eq:wall_cross_corner_higher_seidel}
\begin{array}{ll}
y_i\mapsto y_i,\quad  i=1,\ldots, k-1,\\
y_k\mapsto y_k(1+y_1+\ldots+y_{k-1}).
\end{array}
\end{equation}
This is precisely as in (\ref{eq:wall_cross_simple}), except that we do not have the extra `trivial' co-ordinates $y_{k+1},\ldots, y_n$.
The two formulas (\ref{eq:wall_cross_corner_higher_seidel}) and (\ref{eq:wall_cross_simple}) are intertwined by an $SL(n,\Z)$-transformation taking the standard basis from the statement of this lemma (which is also the basis used in Example~\ref{ex:wc_std}) to the one we are currently using (the one appearing in the proof of Lemma~\ref{lem:one_basis_enough}). Explicitly, this change of bases is expressed multiplicatively by its action on monomials as follows:
$$
\begin{array}{l}
y_i=x_k/x_i,\ i=1,\ldots,k-1,\\
y_k=x_k.
\end{array}
$$

A proof of (\ref{eq:wall_cross_corner_higher_seidel}) is very similar to \cite[Proposition~11.8]{SeiBook13}. The holomorphic strips between $L_0,L_1$ can be divided into two groups:
\begin{itemize}
	\item Fibre strips, projecting to one of the two intersection points between the curves defining $L_0$, $L_1$. They are responsible for the part of the wall-crossing map $y_i\mapsto y_i$, $i=1,\ldots,k-1$;
	\item Horizontal strips, projecting to one of the two  planar strips in $\C$ between the curves defining $L_0$, $L_1$. One of the two planar strips does not contain the origin and has a unique holomorphic lift; the other strip contains the origin and has $k$ holomorphic lifts, due to the fact that $\pi$ is singular over the origin. 
\end{itemize}
By analysing the differences between the boundary homology classes of those $k$ lifts of the strip, one establishes the last part of (\ref{eq:wall_cross_corner_higher_seidel}).
\end{proof}	

\subsection{Proof of the toric mutation theorem}

\begin{proof}[Proof of Theorem~\ref{th:mut_higher_dim_toric}]

For simplicity, first assume that $k=n$, so that $F=\{w\}$ is a vertex of $\Delta$.
Also assume that $\Delta$ is in standard form with respect to $(F,w)$, see Definition~\ref{dfn:std_form}.
Consider the line segment $I\subset \Delta$ connecting 
the vertex $w=(-1,\ldots,-1)$  with the origin.
By standard toric geometry, the fibration $p$ restricted to a neighbourhood of $I$ in the base,
is fibrewise symplectomorphic to the standard fibration
$$
p\co \C^k\to (\R_{\ge -1})^k
$$
from Lemma~\ref{lem:toric_T_D} restricted to a neighbourhood
$$
W\subset \C^k
$$
of the set 
$$
\{(z_1,\ldots,z_k):0\le|z_1|=\ldots=|z_k|\le 1\}.
$$
In this model, $L$ is identified with $T$.
By a slight abuse of notation, we shall also see the same model as sitting inside $X$, i.e.~shall write:
$$
W\subset X.
$$
By  Lemma~\ref{lem:toric_T_D}, the standard Lagrangian cone $D$ from (\ref{eq:def_T_D}) lives in this neighbourhood. We denote by the same letter the same Lagrangian cone seen inside $X$ and attached to the monotone fibre $L$.

Let $\Sigma\subset X$ a smoothing of the toric boundary divisor; then $\Sigma$ is a smooth anticanonical divisor, and $L$ is exact  in the complement $X\setminus \Sigma$. 
We can choose the smooth divisor $\Sigma\subset X$ in such a way that in the model neighbourhood $W$, the intersection $\Sigma\cap W$ is identified with
$$
\{z_1\ldots z_k=\epsilon\} \subset \C^k.
$$
By construction, $T\cup D$ has a neighbourhood $U'$ avoiding $\{z_1\ldots z_k=\epsilon\}$ and sitting within $W$. We can find a smaller neighbourhood $U$ of $T\cup D$ within $U'$ such that
\begin{equation}
\label{eq:U_Liouv}
(U,\omega_{std})\cong (V,dd^c\Phi)
\end{equation}
for a Liouville domain $(V,d^c\Phi)$, using Proposition~\ref{prop:higher_conf_nbhd}. 

We define the mutated torus
$$
L'=\mu_{F,v}L\coloneqq T_\gamma,
$$
where $T_\gamma\subset (U,\omega_{std})$ is any Chekanov-type torus sitting inside $U$ which is $d^c\Phi$-exact, see Lemma~\ref{lem:conv_J}.
By construction, we have a family of inclusions
$$
L,L'\subset U\subset X\setminus \Sigma.
$$
By (\ref{eq:U_Liouv}), $U$ is a Liouville domain;
by Lemma~\ref{lem:conv_J}, $L,L'$ are exact in $U$.
The inclusion
$U\subset X\setminus \Sigma
$
is an inclusion of Liouville domains because $L\subset X\setminus \Sigma$ is also exact and  $H^1(L)\to H^1(U)$ is surjective, compare with the proof of Theorem~\ref{th:lag_mut}.
Now Theorem~\ref{th:mut_higher_dim_toric} follows from Theorem~\ref{th:wall_crossing} and Lemma~\ref{lem:one_basis_enough}.

Finally, consider the case $k<n$. We again assume $\Delta$ is in standard form with respect to $(F,w)$. By standard toric geometry, the restriction of the fibration on $X$ onto the segment connecting $w$ to the origin contains a copy of  
$$W\times \left(\{z\in \C:1-\epsilon<|z|<1+\epsilon\}\right)^{n-k}$$
where $W$ is the $k$-dimensional model introduced above. The monotone torus $L\subset X$ is identified the product of the standard Clifford torus $T$  and the standard exact $(n-k)$-torus
$$
\{|z_i|=1,\ i=k+1,\ldots,n\}\subset \left(\{z\in \C:1-\epsilon<|z|<1+\epsilon\}\right)^{n-k}
$$
The mutated torus $\mu_{F,p}L$ is defined to be the product of a Chekanov-type torus and the same standard $(n-k)$-torus above. The rest of the proof is analogous and is left to the reader; the new factors correspond to the identity terms in the wall-crossing formula (\ref{eq:wc_std_corner}): $x_i\mapsto x_i$, $i=k+1,\ldots,n$.
\end{proof}	

\subsection{Non-mutability}

We finish by an example showing that the condition that $w\in F$ be an \emph{interior} lattice point in Definition~\ref{def:mut_config} and Theorem~\ref{th:mut_higher_dim_toric} is important. 

We begin with the following combinatorial observation. 
Suppose $\Delta$ is a monotone polytope defining a toric Fano $n$-manifold $X$. If $F$ is a facet of $\Delta$, $w\in F$ is its interior lattice point, $\eta'$ is the outward-pointing normal to a different facet $F'$ of $\Delta$, then 
$$\langle\eta', w\rangle\le 0.$$
Indeed, the equality $\langle\eta', w\rangle= 0$ is achieved at the intersection $w\in F\cap F'$, but then $w$ would not be an interior point of $F$. The converse is also true: if $w\in F$ is not an interior point, then there is an $\eta'$ such that $\langle\eta', w\rangle= 1$.

Now suppose $F$ is a facet of $\Delta$ and $w\in F$ is its lattice point, not necessarily an interior one. 
Consider the polytope $$\Gamma\times [-1,1]\subset \R^{n+1}$$ 
defining the manifold $X\times \C P^1$, its codimension~2 face 
$$G=F\times\{1\},$$ and the point
$$q=w\times\{1\}.$$
Then $q$ is an interior point of $G$ if and only if $w$ is an interior point of $F$.

\begin{proposition}
In the above setup, regardless of whether or not $q$ is an interior point of $G$, consider $(G,q)$ as a mutation configuration to formally define the associated wall-crossing transformation by (\ref{eq:wc_map_higher}). The toric potential of $X\times \C P^1$ remains a Laurent polynomial under this wall-crossing if and only if $q$ is an interior point of $G$. 

Therefore,  the statement of Theorem~\ref{th:mut_higher_dim_toric} does not generally hold without the requirement that the lattice point from the definition of a mutation configuration is an interior point of its face.
\end{proposition}

\begin{proof}
To begin with, we return to the polytope $\Delta$ defining $X$. Applying an $SL(n,\Z)$-transformation, we can arrange
that the facet $F$  belongs to the hyperplane $\{x_1=-1\}$, and that
$w=(-1,0,\ldots,0)$. Then $w$ is \emph{not} an interior point of $F$ (equivalently, $q$ is not an interior point of $G$) if and only if there exists a facet $F'\neq F$ of $\Delta$ whose outward-pointing normal $\eta'$ has negative first co-ordinate:
$$
\eta'_1\le -1.
$$
This follows from the previous discussion.

Then up to a reordering of the co-ordinates, $\Gamma=\Delta\times [-1,1]$ is in standard form with respect to $(G,q)$, in the sense of Definition~\ref{dfn:std_form} but ignoring the fact that $q$ need not be an interior point.
This allows us to write down the wall-crossing map with respect to $(G,q)$ using (\ref{eq:wc_std_corner}):
$$
\begin{array}{l}
x_1\mapsto x_1z(x_1^{-1}+z^{-1})=x_1(1+z/x_1),
\\
z\mapsto z^2(x_1^{-1}+z^{-1})=z(1+z/x_1),\\
x_i\mapsto x_i,\ i\neq 1
\end{array}
$$ 
We are using the co-ordinates $x_1,\ldots,x_n,z$ where $z$ corresponds to the extra $\R$-factor from the definition of $\Gamma=\Delta\times [-1,1]$.
In these co-ordinates, the toric potential (\ref{eq:toric_pot}) of $X\times\C P^1$ equals:
$$
W_L(x_1,x_2,\ldots,x_n,z)=
x_1^{-1}+\sum_{\eta'}\mathbf{x}^{\eta'}+z+z^{-1}.
$$
Here $x_1^{-1}$ is the term corresponding to the facet $F\times [-1,1]$ of $\Gamma$, the summation is taken over the normals $\eta'$ to the facets of $\Delta$ other than $F$, and the remaining two terms $z+z^{-1}$ come from the two ``horizontal'' facets of $\Gamma$. 

The sum of the three terms
$$
x_1^{-1}+z^{-1}+z
$$
transforms into a Laurent polynomial under the above wall-crossing, therefore the remaining summands of $W_L$, namely:
$
\sum_{\eta'}\mathbf{x}^{\eta'}
$
must also transform into a Laurent polynomial. On the other hand, expanding co-ordinatewise, they transform into the following expression:
$$
\sum_{\eta'}x_1(1+z/x_1)^{\eta'_1}x_2^{\eta'_2}\ldots x_n^{\eta'_n}
$$
Clearly, this is a Laurent polynomial if and only if $\eta'_1\ge 0$ for all $\eta'$ appearing in the sum. Otherwise this expression, considered as a meromorphic function, has a pole over $\{z=-x_1\}$; whereas Laurent polynomials do not have poles away from co-ordinate hyperplanes.
\end{proof}

\bibliography{Symp_bib}{}

\begin{thebibliography}{10}

\bibitem{ASei10}
M.~Abouzaid and P.~Seidel.
\newblock {An open string analogue of Viterbo functoriality}.
\newblock {\em Geom. Topol.}, 14:627--718, 2010.

\bibitem{ACGK12}
M~Akhtar, T.~Coates, S.~Galkin, and A.~M. Kasprzyk.
\newblock {Minkowski polynomials and mutations}.
\newblock {\em SIGMA}, 8, 2012.

\bibitem{Au07}
D.~Auroux.
\newblock {Mirror symmetry and T-duality in the complement of an anticanonical
  divisor}.
\newblock {\em J. G{\"o}kova Geom. Topol.}, 1:51--91, 2007.

\bibitem{Au09}
D.~Auroux.
\newblock {Special Lagrangian fibrations, wall-crossing, and mirror symmetry}.
\newblock In {\em {Geometry, analysis, and algebraic geometry}}, volume~13 of
  {\em {Surveys in Differential Geometry}}, pages 1--47. Intl. Press, 2009.

\bibitem{Au17}
D.~Auroux.
\newblock {Speculations on homological mirror symmetry for hypersurfaces in
  $(\mathbb{C}^*)^n$}.
\newblock {\em arXiv:1705.06667}, 2017.

\bibitem{AGM01}
D.~Auroux, D.~Gayet, and J.-P. Mohsen.
\newblock {Symplectic hypersurfaces in the complement of an isotropic
  submanifold}.
\newblock {\em Math. Ann.}, 321(4):739--754, 2001.

\bibitem{CdS01}
A.~{{C}annas da Silva}.
\newblock {\em {Lectures on Symplectic Geometry}}, volume 1764 of {\em {Lecture
  Notes in Math.}}
\newblock Springer-Verlag, 2001.

\bibitem{CW13}
F.~Charest and C.~Woodward.
\newblock Floer trajectories and stabilizing divisors.
\newblock {\em J. Fixed Point Theory Appl.}, 19(2):1165--1236, 2017.

\bibitem{CO06}
C.-H. Cho and Y.-G. Oh.
\newblock {Floer cohomology and disc instantons of Lagrangian torus fibers in
  Fano toric manifolds}.
\newblock {\em Asian J. Math.}, 10(4):773--814, 2006.

\bibitem{CEL10}
K.~Cieliebak, T.~Ekholm, and J.~Latschev.
\newblock {Compactness for holomorphic curves with switching {L}agrangian
  boundary conditions}.
\newblock {\em J. Symplectic Geom.}, 8(3):267--298, 2010.

\bibitem{CM13}
J.~A. {Cruz Morales}.
\newblock {\em {Topics in quantum differential equations related to mirror
  symmetry of Fano manifolds}}.
\newblock PhD thesis, {Tokyo Metropolitan University}, 2013.

\bibitem{CMG13}
J.~A. {Cruz Morales} and S.~Galkin.
\newblock {Upper bounds for mutations of potentials}.
\newblock {\em SIGMA}, 9(005), 2013.

\bibitem{Do96}
S.~K. Donaldson.
\newblock {Symplectic submanifolds and almost-complex geometry}.
\newblock {\em J. Diff. Geom.}, 44(4):666--705, 1996.

\bibitem{EnPo09}
M.~Entov and L.~Polterovich.
\newblock {Rigid subsets of symplectic manifolds}.
\newblock {\em Compositio Math.}, 145:773--826, 2009.

\bibitem{FZ02}
S.~Fomin and A.~Zelevinski.
\newblock {Cluster algebras. I. Foundations.}
\newblock {\em J. Amer}, 15(2):497--529, 2002.

\bibitem{FO310b}
K.~Fukaya, Y.-G. Oh, H.~Ohta, and K.~Ono.
\newblock {Lagrangian Floer theory on compact toric manifolds: survey}.
\newblock {\em arXiv:1011.4044}, 2010.

\bibitem{FO3Book}
K.~Fukaya, Y.-G. Oh, H.~Ohta, and K.~Ono.
\newblock {\em {Lagrangian Intersection Floer Theory: Anomaly and
  Obstruction}}, volume~46 of {\em {Stud. Adv. Math.}}
\newblock American Mathematical Society, International Press, 2010.

\bibitem{FO311b}
K.~Fukaya, Y.-G. Oh, H.~Ohta, and K.~Ono.
\newblock {Spectral invariants with bulk quasimorphisms and Lagrangian Floer
  theory}.
\newblock {\em arXiv:1105.5123}, 2011.

\bibitem{GU12}
S.~Galkin and A.~Usnich.
\newblock {Mutations of potentials}.
\newblock {\em Preprint IPMU 10-0100}, 2012.

\bibitem{GHK15}
M.~Gross, P.~Hacking, and S.~Keel.
\newblock {Birational geometry of cluster algebras}.
\newblock {\em Algebr. Geom.}, 2(2):137--175, 2015.

\bibitem{HV00}
K.~Hori and C.~Vafa.
\newblock {Mirror Symmetry}.
\newblock {\em arXiv:hep-th/0002222}, 2000.

\bibitem{KP11}
L.~Katzarkov and V.~Przyjalkowski.
\newblock {Landau-Ginzburg models - old and new}.
\newblock {\em Proceedings of 18th G\"okova Geometry-Topology Conference},
  pages 97--124, 2011.

\bibitem{KS06}
M.~Kontsevich and Y.~Soibelman.
\newblock {Affine structures and non-Archimedean analytic spaces}.
\newblock In {\em {The unity of Mathematics}}, volume 244 of {\em {Progr.
  Math.}}, pages 321--385. Birkh{\"a}user Boston, Boston, MA, 2006.

\bibitem{MD11}
D.~McDuff.
\newblock {Displacing Lagrangian toric fibers via probes}.
\newblock In {\em {Low-dimensional and Symplectic Topology}}, volume~82 of {\em
  Proc. Sympos. Pure Math.}, pages 131--161. AMS, 2011.

\bibitem{Mo13}
J.-P. Mohsen.
\newblock {Transversalit\'e quantitative en g\'eom\'etrie symplectique:
  sous-vari\'et\'es et hypersurfaces}.
\newblock {\em arXiv:1307.0837}, 2013.

\bibitem{Sei00}
P.~Seidel.
\newblock {Graded Lagrangian submanifolds}.
\newblock {\em Bull. Soc. Math. France}, 128(1):103--149, 2000.

\bibitem{Sei08}
P.~Seidel.
\newblock {A biased view of symplectic cohomology}.
\newblock {\em Current Developments in Mathematics}, 2006:211--253, 2008.

\bibitem{SeiBook08}
P.~{S}eidel.
\newblock {\em {Fukaya Categories and Picard-Lefschetz Theory}}.
\newblock European Mathematical Society, Zurich, 2008.

\bibitem{SeiBook13}
P.~Seidel.
\newblock {\em {Lectures on categorical dynamics}}.
\newblock Author's website, 2013.

\bibitem{STW15}
V.~Shende, D.~Treumann, and H.~Williams.
\newblock {On the combinatorics of exact Lagrangian surfaces}.
\newblock {\em arXiv:1603.07449}, 2016.

\bibitem{She15}
N.~Sheridan.
\newblock {Homological Mirror Symmetry for Calabi-Yau hypersurfaces in
  projective space}.
\newblock {\em Invent. Math.}, 199(1):1--186, 2015.

\bibitem{She13}
N.~Sheridan.
\newblock {On the Fukaya category of a Fano hypersurface in projective space}.
\newblock {\em Publ. Math. Inst. Hautes {\'E}tudes Sci.}, 2016.

\bibitem{Sym03}
M.~Symington.
\newblock {Four dimensions from two in symplectic topology}.
\newblock In {\em {Proceedings of the 2001 Georgia International Topology
  Conference}}, pages 153--208, 2003.

\bibitem{To17}
D.~Tonkonog.
\newblock {From symplectic cohomology to Lagrangian enumerative geometry}.
\newblock {\em arXiv:1711.03292}, 2017.

\bibitem{Vi13}
R.~Vianna.
\newblock {On exotic Lagrangian tori in $\mathbb{C}P^2$}.
\newblock {\em Geom. Topol.}, 18:2419--2476, 2014.

\bibitem{Vi14}
R.~Vianna.
\newblock {Infinitely many exotic monotone Lagrangian tori in $\mathbb{C}P^2$}.
\newblock {\em J. Topol.}, 9(2):535--551, 2016.

\bibitem{Vi16}
R~Vianna.
\newblock {Infinitely many monotone Lagrangian tori in del Pezzo surfaces}.
\newblock {\em arXiv:1602.03356}, 2016.

\end{thebibliography}
\bibliographystyle{plain}

\end{document}